\documentclass{cocv}

\usepackage{anyfontsize}
\usepackage{graphicx}
\usepackage{tikz}
\usetikzlibrary{arrows.meta,positioning,intersections,shapes,arrows}
\usepackage{amsmath,amssymb,amsfonts,mathtools,bm}
\usepackage{xcolor}
\usepackage{enumitem}
\usepackage[hidelinks]{hyperref}
\allowdisplaybreaks

\theoremstyle{definition}
\newtheorem{assmptn}[thrm]{Assumption}
\theoremstyle{definition}
\newtheorem*{dfntnintro}{Definition}
\theoremstyle{plain}

\DeclarePairedDelimiterX{\setdef}[2]{\{}{\}}{#1\,\delimsize\vert\,\mathopen{} #2}
\DeclarePairedDelimiterX{\scprod}[2]{\langle}{\rangle}{#1,#2}
\newcommand{\dd}{\mathrm{d}}
\newcommand{\Lb}{\mathcal{L}_{\mathrm{b}}}
\newcommand{\Lp}[1]{\mathrm{L}^{#1}}
\newcommand{\Wkp}[1]{\mathrm{W}^{#1}}
\newcommand{\Hk}[1]{\mathrm{H}^{#1}}
\newcommand{\X}{\mathcal{X}}
\newcommand{\U}{\mathcal{U}}
\newcommand{\Y}{\mathcal{Y}}
\newcommand{\loc}{\mathrm{loc}}
\newcommand{\ext}{\operatorname{ext}}
\newcommand{\dom}{\operatorname{dom}}
\newcommand{\im}{\operatorname{im}}
\newcommand{\R}{\mathbb{R}}
\newcommand{\N}{\mathbb{N}}
\newcommand{\sbvek}[2]{\left[\begin{smallmatrix}#1\\#2\end{smallmatrix}\right]}
\newcommand{\spvek}[2]{\left(\begin{smallmatrix}#1\\#2\end{smallmatrix}\right)}

\begin{document}

\title{Prescribed-performance position tracking for infinite-dimensional passive systems}
\runningtitle{Position tracking for infinite-dimensional passive systems}
\thanks{Corresponding author: Timo Reis.}
\thanks{This work was supported by the FRS-FNRS (Belgium), grant number CR 40010909, and by the Deutsche Forschungs\-gemeinschaft (DFG, German Research Foundation), project number 362536361, \emph{Adaptive control of coupled rigid and flexible multibody systems with port-Hamiltonian structure}.}

\author{Anthony Hastir}
\address{Department of Mathematics and Namur Institute for Complex Systems (naXys), University of Namur, Rue de Bruxelles 61, 5000 Namur, Belgium; \email{anthony.hastir@unamur.be}}
\author{Timo Reis}
\address{Institut f\"ur Mathematik, Technische Universit\"at Ilmenau, Weimarer Stra\ss e 25, 98693 Ilmenau, Germany; \email{timo.reis@\allowbreak tu-ilmenau.de}}
\runningauthors{A. Hastir and T. Reis}

\begin{abstract}
We study funnel control for linear infinite-dimensional passive systems whose controlled output is obtained by integrating the passive output. In mechanical applications, this corresponds to position control based on a measured passive velocity output. Within the system-node framework, we propose a nonlinear recursive funnel controller with a dynamic gain supervisor for systems with distributed or boundary control and observation. The supervisor automatically increases a damping gain whenever the normalized recursive error approaches a prescribed monitoring level. The controller uses the tracking error and its derivative, with the latter obtained directly from the passive output rather than by numerical differentiation. We prove global existence and uniqueness of the closed-loop trajectory, prescribed funnel performance, and global boundedness of the derivative tracking error and the passive output. Under an additional position-compatible strict dissipativity condition, the supervised gain is increased only finitely many times and eventually becomes constant. Consequently, the gain, the internal state, and the control input are globally bounded, and the recursive-error funnel margin is uniform in time. The results are illustrated by a boundary-actuated Euler--Bernoulli beam with uniformly positive distributed damping.
\end{abstract}


\maketitle
\markboth{A. HASTIR AND T. REIS}{POSITION TRACKING FOR INFINITE-DIMENSIONAL PASSIVE SYSTEMS}

\section*{Introduction}

We consider position tracking for linear distributed-parameter systems arising,
in particular, from continuum mechanics. Our aim is to develop a theory for an
abstract class of infinite-dimensional systems that encompasses both boundary
and distributed actuation and observation. A characteristic feature of many
mechanical systems in this class is that the variable to be controlled is a
position or displacement, whereas the passive output is a velocity. Indeed,
force and velocity form a power-conjugate pair: their inner product equals the
instantaneous power supplied to the system. Position control by means of a
co-located force therefore naturally leads to a situation in which the
controlled output is the time integral of a passive output.

A typical example is a boundary-actuated Euler--Bernoulli beam. If a
transverse force $u$ is applied at the beam endpoint, the corresponding
co-located passive output is the endpoint velocity $v$. The controlled
quantity of interest, however, is often the endpoint displacement $y$, so that
$\dot y=v$. Hence, position control introduces an additional integration
between the passive velocity output and the controlled output. The same
force--velocity--position structure occurs for vibrating strings and rods,
Timoshenko beams, elastic plates, and networks of flexible mechanical
components.

Motivated by this observation, we consider an abstract class of
infinite-dimensional systems that includes the beam above and many other
distributed-parameter mechanical systems. The underlying velocity-output
system is represented by
\begin{equation}
\label{eq:VelocityNode}
\spvek{\dot{x}(t)}{v(t)}
=
\sbvek{A\&B}{[1mm]C\&D}
\spvek{x(t)}{u(t)},
\end{equation}
where $\X$ and $\U$ are Hilbert spaces, $x(t)\in\X$ is the state,
$u(t)\in\U$ is the input, and $v(t)\in\U$ is the co-located velocity
output.

The operator matrix in \eqref{eq:VelocityNode} is understood as a
\emph{system node}, an operator-theoretic formulation of a linear
input-state-output system allowing unbounded control and observation
operators. It covers both distributed and boundary control and observation. In the
boundary-control case, the boundary conditions and the associated
compatibility between the state $x$ and the input $u$ are encoded in the
domain of the system node.

We assume that the velocity-output system is \emph{impedance passive}. More
precisely, the state norm is chosen such that
$\frac12\|x(t)\|_\X^2$ represents the stored energy, and every sufficiently
regular trajectory satisfies
$\frac12\frac{\dd}{\dd t}\|x(t)\|_\X^2 \leq \langle u(t),v(t)\rangle_\U$.
Thus, the increase in stored energy cannot exceed the power supplied through
the input-output pair $(u,v)$. Equality corresponds to an energy-preserving
system, whereas the inequality allows for internal dissipation. The
Euler--Bernoulli beam above satisfies this passivity relation.

The controlled position output $y(t)\in\U$ is related to the passive velocity
output by
$\dot y(t)=v(t)$.
Hence, we consider impedance-passive infinite-dimensional systems in which
the passive output is integrated once to obtain the controlled output.

Our objective is to make the position output $y$ track a prescribed reference
signal $y_{\rm ref}$. With the tracking error
$e:=y-y_{\rm ref}$,
we have
$\dot e=v-\dot y_{\rm ref}$.
The controller is formulated in terms of $e$ and $\dot e$.

We employ funnel control to prescribe the transient and asymptotic behaviour
of the position tracking error. For a prescribed positive function
$\varphi:\R_{\geq0}\to\R_{>0}$, the control objective is
$\varphi(t)\|e(t)\|_\U<1$ for all $t\geq0$.
Thus, $1/\varphi(t)$ specifies a time-varying error bound, which may allow a
large initial error while enforcing higher accuracy after the transient
phase. The funnel function need not be monotone.

To account for the additional integration between the passive output and the
controlled output, we introduce a recursive error and a supervised feedback
controller. Choose the design parameters
\[
\alpha,\beta,\delta,\tau_{\rm f},\Delta k,\nu>0,
\qquad
r_{\rm a}\in(0,1),
\qquad
k_{{\rm a},0}\geq0,
\]
and consider the controller
\begin{subequations}
\label{eq:ClosedLoopControlVelocity}
\begin{align}
e_{\rm r}(t)
&:=
\dot e(t)
+
\left(
\frac{\alpha}
{1-\varphi(t)^2\|e(t)\|_\U^2}
+
\frac{\dot\varphi(t)}{\varphi(t)}
+
\delta
\right)e(t),
\label{eq:RecursiveErrorIntro}
\\
\eta(t)
&:=
\tau_{\rm f}\varphi(t)e_{\rm r}(t),
\\
u(t)
&=
u_{\ext}(t)
-
k_{\rm a}(t)e_{\rm r}(t)
-
\frac{\beta e_{\rm r}(t)}
{1-\|\eta(t)\|_\U^2},
\label{eq:IntroController}
\\
\dot k_{\rm a}(t)
&=
\nu\sigma(t),
\qquad
k_{\rm a}(0)=k_{{\rm a},0},
\qquad
\sigma(t)\in\{0,1\}.
\end{align}
\end{subequations}
Here, $u_{\ext}:\R_{\geq0}\to\U$ is a prescribed additive input and $\sigma$ is generated by a supervisor with two modes.
In the holding mode, $\sigma=0$ and the gain remains constant. When
the normalized recursive error reaches the monitoring level
$r_{\rm a}$, the supervisor switches to the gain-increase mode
$\sigma=1$ for the fixed duration
\[
T_{\rm inc}:=\frac{\Delta k}{\nu},
\]
so that $k_{\rm a}$ is increased by precisely $\Delta k$.
The supervisor then returns to the holding mode. The switching times
and the recursive concatenation of the two fixed-mode trajectories
into a single supervised closed-loop trajectory are defined in
Subsection~\ref{subsec:SupervisedGain}. Thus, the closed loop combines
continuous system and controller dynamics with a discrete switching
mechanism for the gain.

As the position error approaches the funnel boundary, the coefficient
multiplying $e$ in \eqref{eq:RecursiveErrorIntro} becomes large. Keeping
$e_{\rm r}$ bounded therefore forces $\dot e$ to develop an inward-pointing
component and drives the position error away from the boundary. The additive
input $u_{\ext}$ may represent a nominal feedforward input or compensate for
known forcing; it also provides the freedom needed to satisfy the initial
compatibility condition.

The singular term in \eqref{eq:IntroController} is used to keep the
recursive error within
$\tau_{\rm f}\varphi(t)\|e_{\rm r}(t)\|_\U<1$.
This auxiliary condition is not an additional tracking objective. Rather,
together with the position-funnel condition, it controls the derivative of
the tracking error, since
\[
\dot e(t)
=
e_{\rm r}(t)
-
\left(
\frac{\alpha}
{1-\varphi(t)^2\|e(t)\|_\U^2}
+
\frac{\dot\varphi(t)}{\varphi(t)}
+
\delta
\right)e(t).
\]
The term $\dot\varphi/\varphi$ accounts for the motion of the
position-funnel boundary, whereas $\delta e$ provides an additional decay
margin in the prescribed error dynamics.

The supervised linear term $-k_{\rm a}e_{\rm r}$ serves a different
purpose. It supplies an adjustable amount of additional damping without requiring
knowledge of a system-dependent gain threshold. The singular feedback
guarantees funnel feasibility, while the supervisor increases the linear gain
only when the normalized recursive error approaches the prescribed monitoring
level. Under an additional strict dissipativity condition, only finitely many
such increases are needed.

The controller is inspired by the recursive funnel-control construction for
systems with known strict relative degree developed in \cite{Ber2018}, but its
specific form is modified here to make the infinite-dimensional closed-loop
analysis possible. The supervised gain preserves the model-free character of
the controller: no parameters or operators of the underlying system enter the
feedback law or the switching rule. The resulting feedback interconnection is
shown schematically in Figure~\ref{fig:blockdiag}.

\tikzstyle{block} = [draw, thick, fill=white, rectangle,
    minimum height=3em, minimum width=1em]
\tikzstyle{sum} = [draw, thick, fill=white, circle, node distance=1cm]
\tikzstyle{input} = [coordinate]
\tikzstyle{output} = [coordinate]

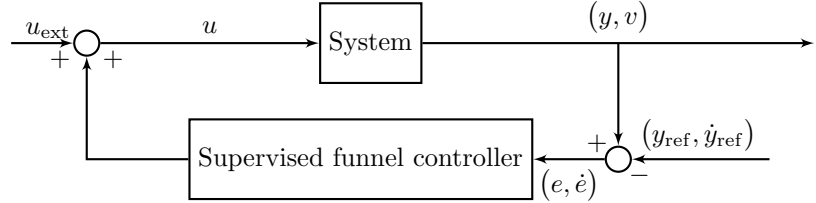
\begin{figure}[ht]
    \centering
\begin{tikzpicture}[auto, node distance=2cm,>=latex',thick]

    \node [input, name=input] {};
    \node [sum, right of=input, xshift=0em] (sum2) {};
    \node [block, right of=sum2, xshift=5em] (system) {System};
    \node [output, right of=system,xshift=11em] (output) {};

    \node at ([yshift=-0.1em,xshift=1em]sum2.south) {$+$};
    \node at ([yshift=-0.1em,xshift=-1em]sum2.south) {$+$};

    \draw [->] (sum2) -- node[name=u] {$u$} (system);
    \draw [->] (system) --
        node [name=y] {$\bigl(y,v\bigr)$} (output);

    \node[sum,below of=y,yshift=-2.5em] (sum3) {};
    \node[input,right of=sum3] (ref) {};

    \draw[->] (y) -- (sum3);
    \draw[->] (ref) -- (sum3);

    \node at ([yshift=4,xshift=30]sum3.north)
        {$\bigl(y_{\rm ref},\dot y_{\rm ref}\bigr)$};
    \node at ([yshift=-1,xshift=8]sum3.south) {$-$};
    \node at ([yshift=1,xshift=-8]sum3.north) {$+$};

    \node [block, left of=sum3, xshift=-4em] (controller2)
        {Supervised funnel controller};

    \draw [->] (sum3) --
        node {$\bigl(e,\dot e\bigr)$} (controller2);

    \draw [->] (controller2) -- ++(-2.5,0) -| (sum2.south);

    \draw [->] (input) --
        node[right,xshift=-1em,yshift=0.5em]
        {$u_{\mathrm{ext}}$} (sum2);

\end{tikzpicture}
\caption{Feedback interconnection of the supervised recursive funnel
controller and the system with additive external input $u_{\ext}$.}
\label{fig:blockdiag}
\end{figure}

Funnel control was introduced in \cite{IlcRyaSan2002} and has since been
developed for broad classes of finite-dimensional systems. Extensions include
multi-input multi-output systems \cite{IlcRyaTre2005}, nonlinear systems with
strict relative degree \cite{Ber2018}, and systems subject to input
constraints \cite{Ber2024}. Applications range from chemical reactors
\cite{IlcTre2004} and wind turbines \cite{Hac2014} to robotic manipulators
\cite{HacKen2012} and electrical circuits \cite{BerRei2014}. For surveys on
the theory and applications of funnel control, we refer to
\cite{IlchmannRyan2008,BerIlcRya2021}.

For infinite-dimensional systems, funnel control has so far been established
only for particular system classes. Early approaches transferred
finite-dimensional Byrnes--Isidori techniques to infinite-dimensional
systems \cite{IlchmannSeligTrunk2016}. Related results, including systems with
moderate nonlinearities, were developed in \cite{HasWinDoc2023}. These
approaches impose restrictive assumptions on the control and observation
operators and generally do not include boundary control. Nevertheless, they
cover relevant examples such as the moving-water-tank system studied in
\cite{Ber2020,Ber2022}.

Funnel control for boundary control systems was first considered for the heat
equation in \cite{ReiSel2015}. A broader class of passive boundary control
systems was studied in \cite{MarTimFel2021} using maximally monotone operator
methods. This approach was subsequently applied to the nonlinear monodomain
equation with boundary control in
\cite{BergerBreitenPucheReis2021}. More recently, the system-node framework
has been used to develop a unified theory for funnel control of
impedance-passive infinite-dimensional systems with both boundary and
distributed control and observation \cite{GovHasPauRei2025}.

The existing passivity-based results for infinite-dimensional systems concern
the case in which the controlled output is itself power conjugate to the
input. In mechanical systems, this typically means that a force input is used
to control a velocity output. These results do not apply directly to position
tracking, since the controlled position output is the integral of the passive
velocity output and therefore does not form an impedance-passive pair with the
input.

The controller \eqref{eq:IntroController} retains the connection to the
passive output through its explicit use of
$\dot e = v-\dot y_{\rm ref}$.
At the same time, the position error enters the recursively defined variable
$e_{\rm r}$ through the singular position-error feedback in
\eqref{eq:RecursiveErrorIntro}. The resulting closed-loop dynamics must
therefore be analysed on a product space containing both the system state and
the position error. The supervised gain introduces an additional scalar
controller state and a switching rule, but does not alter the principal
passive system-node relation.

The main control-theoretic contribution of this work is the development of a
nonlinear recursive funnel controller for position tracking in
impedance-passive infinite-dimensional systems, including systems with unbounded boundary
control and observation operators. The feedback law uses only the position
tracking error and its derivative, where the latter is obtained directly from
the passive velocity output. In particular, no numerical differentiation of
the position signal and no system parameters are required for its
implementation. The dynamic supervisor provides additional damping whenever
the normalized recursive error approaches a prescribed monitoring level,
without requiring knowledge of a system-dependent gain threshold.

Under impedance passivity and surjectivity of the observation operator, we
prove that the resulting closed loop has a unique global trajectory and
satisfies the prescribed position-funnel constraint with a uniform margin.
The recursive error remains strictly inside its funnel on every finite time
interval, and the recursive error, the derivative tracking error, and the
passive velocity output are globally bounded. These conclusions apply to both
distributed and boundary control and observation within the system-node
framework.

Under an additional position-compatible strict dissipativity condition, the
supervisor is activated only finitely many times and the supervised gain
eventually becomes constant. Consequently, the gain, the internal state, and
the control input are globally bounded, and the recursive-error funnel margin
is uniform in time. The required gain level enters only the analysis and is
not needed for the implementation of the controller. We verify the abstract
conditions for a boundary-actuated Euler--Bernoulli beam, where uniformly
positive distributed damping yields the required strict dissipativity
property.

The paper is organised as follows. Section~\ref{sec:prelim} recalls the
required facts on system nodes and introduces impedance-passive systems with
integrated position output. Section~\ref{sec:ass} collects the assumptions on
the system, the prescribed closed-loop data, the supervisor, and the initial
values. In Section~\ref{sec:main}, the closed-loop system is formulated as a
nonlinear evolution equation, and global existence and funnel performance are
established using maximal monotone operator methods and a recursive
construction across the switching times.
Section~\ref{sec:GlobalBoundedness} introduces an additional
position-compatible strict dissipativity condition and proves that the
supervisor increases the gain only finitely many times. This yields global
bounds for the supervised gain, the internal state, and the control input.
Section~\ref{sec:ex} illustrates the results for a boundary-actuated
Euler--Bernoulli beam.
\subsection*{Notation and basic operator-theoretic notions}

All spaces considered in this article are real Hilbert spaces. The norm and
inner product on a Hilbert space $\X$ are denoted by
$\|\cdot\|_\X$ and $\scprod{\cdot}{\cdot}_{\X}$, respectively; the
subscript is omitted whenever no ambiguity arises. For Hilbert spaces $\X$
and $\Y$, let $\Lb(\X,\Y)$ denote the space of bounded linear operators from
$\X$ to $\Y$, and set $\Lb(\X):=\Lb(\X,\X)$. The identity operator on $\X$
is denoted by $I_{\X}$, with the subscript omitted when clear. For
$x\in\X$ and $r>0$, let $\mathcal B_r^{\X}(x)$ denote the open ball of
radius $r$ centered at $x$; again, the superscript is omitted when clear.

For a linear operator
$A:\dom(A)\subset\X\to\Y$, the symbols $\dom(A)$, $\im A$, and $\ker A$
denote its domain, range, and kernel. The operator is called \emph{closed} if
its graph is closed in $\X\times\Y$, and \emph{densely defined} if
$\dom(A)$ is dense in $\X$. For a closed operator, $\dom(A)$ is understood
to be equipped with its graph norm.

We use the standard conventions for Lebesgue and Sobolev spaces from
\cite{AdamsFournier2003}. For spaces of functions with values in a Hilbert
space $\X$, the target space is indicated after the domain; for instance,
$\Lp{p}(\Omega;\X)$ denotes the space of $p$-integrable $\X$-valued
functions on $\Omega\subset\R^d$. Integration of Hilbert-space-valued
functions is understood in the Bochner sense; see \cite{DiestelUhl1977}.

\begin{dfntnintro}[Monotone operators]
Let $\mathcal H$ be a real Hilbert space and
$\mathcal M:\dom(\mathcal M)\subset\mathcal H\to\mathcal H$ a (not necessarily linear) operator.
It is called \emph{monotone} if, for all
$w_1,w_2\in\dom(\mathcal M)$,
\[
\scprod{\mathcal Mw_1-\mathcal Mw_2}{w_1-w_2}_{\mathcal H}
\geq0.
\]
A monotone operator $\mathcal M$ is called \emph{maximally monotone} if there exists some $\lambda>0$ with
$\im(\lambda I_{\mathcal H}+\mathcal M)=\mathcal H$.
\end{dfntnintro}

Maximal monotonicity is conventionally defined by requiring that the graph of
$\mathcal M$ admits no proper extension to a monotone relation in
$\mathcal H\times\mathcal H$. By Minty's theorem
\cite[Thm.~2.2]{Barbu2010}, for a monotone operator this is equivalent to the
range condition used above. Moreover, if the range condition holds for one
$\lambda>0$, then it holds for every $\lambda>0$.

\section{Impedance-passive system nodes}
\label{sec:prelim}

The systems considered below are described by system nodes with state space
$\X$, input space $\U$, and output space $\Y$. We write such a system node as
\begin{equation}
\label{eq:sysnode}
S
=
\sbvek{A\&B}{C\&D}
\colon
\dom(S)\subset\X\times\U
\longrightarrow
\X\times\Y.
\end{equation}
In the position-control setting considered here, $\Y=\U$, and the
system-node output $v$ is the passive velocity output. The controlled
position output $y$ is defined separately by $\dot y=v$.
We use the notion of a system node from
\cite[Def.~4.7.2]{sta2005}.

\begin{dfntn}[System node]
Let $\X$, $\U$, and $\Y$ be Hilbert spaces. A system node with state
space $\X$, input space $\U$, and output space $\Y$ is a linear
operator $S$ as in \eqref{eq:sysnode} such that:
\begin{enumerate}[label=(\alph{*})]
\item
$S$ is closed;
\item
the first component operator
$A\&B:\dom(S)\subset\X\times\U\to\X$
is closed;
\item
for every $u\in\U$, there exists some $x\in\X$ such that
$\spvek{x}{u}\in\dom(S)$;
\item
the main operator
\[
\dom(A)
:=
\setdef{x\in\X}{\spvek{x}{0}\in\dom(S)},
\qquad Ax
:=
A\&B\spvek{x}{0},
\quad x\in\dom(A),
\]
generates a strongly continuous semigroup on $\X$.
\end{enumerate}
\end{dfntn}

\begin{dfntn}[Strong trajectory]
\label{def:StrongTrajectory}
Let $S$ in \eqref{eq:sysnode} be a system node on $(\X,\U,\U)$, and
let $\mathcal I\subset\R_{\geq0}$ be an interval. A triple $(x,u,v)$
with
$x\in\Wkp{1,\infty}_{\loc}(\mathcal I;\X)$,
$u,v\in\Lp{\infty}_{\loc}(\mathcal I;\U)$
is called a \emph{strong trajectory} of the system node $S$ on
$\mathcal I$ if
$\spvek{x(t)}{u(t)}\in\dom(S)$
and \eqref{eq:VelocityNode} holds
for almost every $t\in\mathcal I$.
\end{dfntn}

\begin{rmrk}
\label{rem:StrongTrajectory}
The notion of a strong trajectory used here is intermediate between the
classical and generalized trajectories considered in
\cite[Sec.~4.3]{sta2005}. Every classical trajectory is a strong trajectory.
Conversely, every strong trajectory is a generalized trajectory, that is, it
can be approximated by classical trajectories in the topology used in
\cite{sta2005}. The first assertion follows directly from the definitions,
whereas the second follows by a standard time-mollification argument,
combining \cite[Thm.~4.3.9]{sta2005} with the argument in
\cite[Lem.~A.5]{Reis2026_Dissip}.
\end{rmrk}
For a system node, the operator $C\&D:\dom(S)\to\Y$
is bounded when $\dom(S)$ is endowed with the graph norm of $A\&B$;
see \cite[Lem.~4.7.3]{sta2005}. In particular, the observation operator
$C:\dom(A)\to\Y$,
$Cx:=C\&D\spvek{x}{0}$,
is bounded when $\dom(A)$ is endowed with the graph norm of $A$.

The following elementary consequence of the Hilbert space structure
will be used to construct a bounded lifting of the output.

\begin{prpstn}[Right inverse of the observation operator]
\label{prop:ObservationRightInverse}
Let $S$ be a system node on $(\X,\U,\Y)$. Then the observation operator
$C:\dom(A)\to\Y$
is surjective if and only if it admits a bounded right inverse, that is,
if and only if there exists
$Q\in\Lb(\Y,\dom(A))$
such that
$CQ=I_\Y$.
In this case,
$Q\in\Lb(\Y,\X)$ and $AQ\in\Lb(\Y,\X)$.
\end{prpstn}
\begin{proof}
The first assertion follows from \cite[Prop.~9]{GovHasPauRei2025}. Since the
canonical embedding $\dom(A)\hookrightarrow\X$ and
$A:\dom(A)\to\X$ are bounded, the remaining assertions follow from
$Q\in\Lb(\Y,\dom(A))$.
\end{proof}
Let $\X_{-1}$ be the extrapolation space associated with $A$, and let
$A_{-1}\in\Lb(\X,\X_{-1})$ denote the extrapolation of $A$. Then there
exists a unique control operator $B\in\Lb(\U,\X_{-1})$ such that
\begin{equation}
\label{eq:ABform}
\dom(S)
=
\setdef{\spvek{x}{u}\in\X\times\U}
{A_{-1}x+Bu\in\X},\qquad
A\&B\spvek{x}{u}
=
A_{-1}x+Bu;
\end{equation}
see \cite[Lem.~4.7.3]{sta2005}.

We now specialize to the passive system class considered in this
article.

\begin{dfntn}[Impedance passivity]
\label{def:passive}
A system node
$S=\sbvek{A\&B}{C\&D}$ on $(\X,\U,\U)$ is called \emph{impedance passive} if, for every
$\spvek{x}{u}\in\dom(S)$,
\begin{equation}
\label{eq:KYP}
\scprod{
A\&B\spvek{x}{u}
}{x}_\X
\leq
\scprod{
C\&D\spvek{x}{u}
}{u}_\U.
\end{equation}
\end{dfntn}
\begin{rmrk}
\label{rem:PositiveResolvent}
Setting $u=0$ in \eqref{eq:KYP} gives
$\scprod{Ax}{x}_\X\leq0$
for all $x\in\dom(A)$.
Since $A$ generates a strongly continuous semigroup,
$\mu I-A$ is surjective for all sufficiently large $\mu>0$.
Together with the dissipativity of $A$, the Lumer--Phillips theorem
therefore shows that $A$ is maximally dissipative and generates a
contraction semigroup. Consequently, $\mu I-A$ is boundedly
invertible for every $\mu>0$, and
$\|(\mu I-A)^{-1}\|_{\Lb(\X)} \leq \tfrac{1}{\mu}$;
see \cite[Thm.~II.3.15]{EngelNagel2000}.
\end{rmrk}

For every $\mu>0$, the corresponding transfer operator is defined by
\[
P(\mu)u
:=
C\&D
\spvek{
\displaystyle(\mu I-A_{-1})^{-1}Bu
}{
\displaystyle u
},
\qquad
u\in\U.
\]
By \cite[Lem.~4.7.3]{sta2005}, the mapping
$u \longmapsto \spvek{(\mu I-A_{-1})^{-1}Bu}{u}$
is bounded from $\U$ into $\dom(S)$ endowed with the graph norm.
Hence the graph-norm boundedness of $C\&D$ implies
$P(\mu)\in\Lb(\U)$. Impedance passivity implies, by \cite[Thm.~4.2]{Sta02},
$\scprod{P(\mu)u}{u}_\U\geq0$ for all $\mu>0$ and $u\in\U$.
If the observation operator $C:\dom(A)\to\U$ is surjective, then
this nonnegativity is strengthened to coercivity. More precisely,
for every $\mu>0$ there exists some $m_\mu>0$ such that, for every
$u\in\U$,
\begin{equation}
\label{eq:TransferFunctionCoercivity}
\scprod{P(\mu)u}{u}_\U
\geq
m_\mu\|u\|_\U^2.
\end{equation}
This result is proved in
\cite[Prop.~10]{GovHasPauRei2025} for complex Hilbert spaces. Applying
it to the complexification of the system node and restricting the
resulting estimate to the original real input space gives
\eqref{eq:TransferFunctionCoercivity}. In particular, the
Lax--Milgram theorem implies that $P(\mu)$ is boundedly invertible.
The following quantitative consequence will be needed in the
maximal-monotonicity argument.

\begin{lmm}[Quantitative transfer coercivity]
\label{lem:QuantitativeTransferCoercivity}
Suppose that $C:\dom(A)\to\U$ is surjective, and let $m_1>0$ be a
coercivity constant in \eqref{eq:TransferFunctionCoercivity}
corresponding to $\mu=1$. Then, for every $\lambda\geq1$ and
$u\in\U$,
\[
\scprod{P(\lambda)u}{u}_\U
\geq
\frac{m_1}{\lambda}\|u\|_\U^2.
\]
\end{lmm}
\begin{proof}
Consider the complexification of the system node and its transfer
function, again denoted by $P$, with the complex inner products
taken linear in the first argument. The complexified system node is
again impedance passive. For fixed
$u\in\U\setminus\{0\}$ from the original real input space, viewed
as an element of its complexification, define
$p_u(z) := \scprod{P(z)u}{u}_\U$, $\operatorname{Re}z>0$.
The scalar function $p_u$ is analytic on the right half-plane.
Impedance passivity gives
\[
\operatorname{Re}p_u(z)
\geq
\operatorname{Re}z\,
\left\|
(zI-A_{-1})^{-1}Bu
\right\|_\X^2
\geq0.
\]
By \eqref{eq:TransferFunctionCoercivity},
$p_u(1)>0$. Since $\operatorname{Re}p_u$ is a nonnegative harmonic
function on the right half-plane and
$\operatorname{Re}p_u(1)>0$, the Harnack estimate used
in the proof of \cite[Thm.~11.14]{Rudin1987} shows that
$\operatorname{Re}p_u$ cannot vanish at an interior point. Hence
\[
\operatorname{Re}p_u(z)>0
\qquad
\text{for }\operatorname{Re}z>0.
\]
Hence
$g_u(z):=\frac{1}{p_u(z)}$
also has positive real part.

Set
\[
\phi(z):=\frac{z-1}{z+1},
\qquad
\psi(w):=\frac{w-g_u(1)}{w+g_u(1)}.
\]
Since $g_u(1)>0$, both $\phi$ and $\psi$ map the right half-plane
conformally onto the unit disc, with $\phi(1)=0$ and
$\psi(g_u(1))=0$. Hence
$F:=\psi\circ g_u\circ\phi^{-1}$
is a holomorphic self-map of the unit disc satisfying $F(0)=0$.
The Schwarz lemma \cite[Thm.~12.2]{Rudin1987} therefore yields, for
$\lambda\geq1$,
\[
\left|
\frac{g_u(\lambda)-g_u(1)}
{g_u(\lambda)+g_u(1)}
\right|
\leq
\frac{\lambda-1}{\lambda+1}.
\]
For real $\lambda>0$, the complexified transfer operator maps the
original real input space into itself. Hence $p_u(1)$ and
$p_u(\lambda)$, and therefore also $g_u(1)$ and $g_u(\lambda)$, are
positive real numbers. It follows that
$g_u(\lambda)\leq\lambda g_u(1)$. Consequently,
\[
p_u(\lambda)
\geq
\frac{p_u(1)}{\lambda}
\geq
\frac{m_1}{\lambda}\|u\|_\U^2.\qedhere
\]
\end{proof}

\section{Assumptions}
\label{sec:ass}

We now state the assumptions on the system, the prescribed closed-loop
data, the supervised gain, and the initial values.

\begin{assmptn}[System class]
\label{assum:SystemClass}
The system node
$S=\sbvek{A\&B}{C\&D}$
on $(\X,\U,\U)$ is impedance passive. Moreover, the observation
operator $C:\dom(A)\to\U$ is surjective.
\end{assmptn}

Let
$y_{\rm ref}:\R_{\geq0}\to\U$
be a prescribed reference signal and let
$u_{\ext}:\R_{\geq0}\to\U$
be a prescribed external input. The position tracking error and its
derivative are given by
$e=y-y_{\rm ref}$, $\dot e=v-\dot y_{\rm ref}$.

\begin{assmptn}[Closed-loop data]
\label{assum:FunnelData}
The following conditions are satisfied:
\begin{enumerate}[label=(\alph{*})]
\item
The reference signal satisfies
$y_{\rm ref} \in \Wkp{3,\infty}(\R_{\geq0};\U)$.
\item
The external input satisfies
$u_{\ext} \in \Wkp{2,\infty}(\R_{\geq0};\U)$.
\item
The funnel function satisfies
$\varphi \in \Wkp{3,\infty}(\R_{\geq0})$, $\inf_{t\geq0}\varphi(t)>0$.
\item
The controller parameters satisfy
$\alpha,\beta,\delta,\tau_{\rm f}>0$.
\end{enumerate}
\end{assmptn}

No monotonicity of $\varphi$ is required. The reciprocal
$1/\varphi$ determines the prescribed boundary for the position
tracking error. No independent second funnel function and no
compatibility condition between two funnel functions are imposed.
Instead, the recursively defined error is constrained by the scaled
funnel function $\tau_{\rm f}\varphi$.

\begin{assmptn}[Supervisor parameters]
\label{assum:SupervisorData}
The initial gain and the supervisor parameters satisfy
\[
k_{{\rm a},0}\geq0,
\qquad
r_{\rm a}\in(0,1),
\qquad
\Delta k>0,
\qquad
\nu>0.
\]
\end{assmptn}

Here, $r_{\rm a}$ is the monitoring level for the normalized
recursive error, $\Delta k$ is the amount by which the gain is
increased during one activation of the supervisor, and $\nu$ is the
corresponding rate of increase. Thus, every gain-increase phase has
the fixed duration
$T_{\rm inc} := \frac{\Delta k}{\nu}$.
The precise recursive definition of the gain and its switching times
is given in Subsection~\ref{subsec:SupervisedGain}.

\begin{assmptn}[Initial values]
\label{assum:InitialValues}
Let $x_0\in\X$ and $y_0\in\U$. The initial position error lies inside
the prescribed funnel, that is,
$\varphi(0) \|y_0-y_{\rm ref}(0)\|_\U <1$.
Further, there exists an element $u_0\in\U$ such that
$\spvek{x_0}{u_0}\in\dom(S)$.
For this element, set
\[
e_{\rm r}(0)
:={}
C\&D\spvek{x_0}{u_0}
-
\dot y_{\rm ref}(0)
+
\left(
\frac{\alpha}
{1-\varphi(0)^2
\|y_0-y_{\rm ref}(0)\|_\U^2}
+
\frac{\dot\varphi(0)}{\varphi(0)}
+
\delta
\right)
\bigl(
y_0-y_{\rm ref}(0)
\bigr).
\]
The initial normalized recursive error lies below the monitoring
level, that is,
$\tau_{\rm f}\varphi(0) \|e_{\rm r}(0)\|_\U < r_{\rm a}$.
Moreover, the initial external input is compatible with the feedback
law:
\begin{equation}
\label{eq:InitialControllerCompatibility}
u_{\ext}(0)
=
u_0
+
k_{{\rm a},0}e_{\rm r}(0)
+
\frac{\beta e_{\rm r}(0)}
{1-\tau_{\rm f}^2\varphi(0)^2
\|e_{\rm r}(0)\|_\U^2}.
\end{equation}
\end{assmptn}

Condition
$\spvek{x_0}{u_0}\in\dom(S)$
is the compatibility condition between the initial state and the
initial system input. The condition
$\varphi(0) \|y_0-y_{\rm ref}(0)\|_\U <1$
places the initial position error inside its funnel, while
$\tau_{\rm f}\varphi(0) \|e_{\rm r}(0)\|_\U < r_{\rm a} <1$
places the normalized recursive error inside its funnel and below the
initial monitoring level. Finally,
\eqref{eq:InitialControllerCompatibility} requires the feedback law
to hold at the initial time. If $u_{\ext}(0)$ may be chosen freely,
this condition determines its value for the selected admissible
element $u_0$.

\section{Global existence and funnel performance}
\label{sec:main}

In this section, we establish existence and uniqueness of the global
closed-loop trajectory and prove the prescribed funnel properties. We first
analyse the continuous closed-loop dynamics for each of the two supervisor
modes. For this purpose, the supervised gain is included as an additional
scalar state. A state transformation then leads to a nonlinear evolution
equation on an extended product space with a time-independent principal
system-node relation. After proving global solvability in each mode, the
corresponding trajectories are joined recursively at the switching times.

For brevity, set
$q := \frac{\dot\varphi}{\varphi} + \delta$.
The closed-loop system is described by
\eqref{eq:VelocityNode}, the relations
$\dot y=v$, $e=y-y_{\rm ref}$,
and the controller
\eqref{eq:ClosedLoopControlVelocity}. For strong trajectories,
$\dot e=v-\dot y_{\rm ref}$.

We introduce the normalized errors
$\chi(t)
:=
\varphi(t)e(t)$,
$\eta(t)
:=
\tau_{\rm f}\varphi(t)e_{\rm r}(t)$.
The position-funnel condition and the recursive-error condition are
equivalent to
$\|\chi(t)\|_\U<1$, $\|\eta(t)\|_\U<1$,
respectively. On the open unit ball of $\U$, define the barrier map
\begin{equation}
\label{eq:BarrierMap}
\gamma:\mathcal B_1^\U(0)\to\U,
\qquad
\gamma(\xi)
:=
\frac{\xi}{1-\|\xi\|_\U^2}.
\end{equation}
The controller can then be written as
\begin{equation}
\label{eq:NormalizedController}
\tau_{\rm f}\varphi(t)
\bigl(
u(t)-u_{\ext}(t)
\bigr)
=
-k_{\rm a}(t)\eta(t)
-\beta\gamma\bigl(\eta(t)\bigr).
\end{equation}
Moreover, the definition of $e_{\rm r}$ yields
\begin{equation}
\label{eq:NormalizedVelocityRelation}
\tau_{\rm f}\varphi(t)
\bigl(
v(t)-\dot y_{\rm ref}(t)
\bigr)
=
\eta(t)
-
\tau_{\rm f}\alpha\gamma\bigl(\chi(t)\bigr)
-
\tau_{\rm f}q(t)\chi(t).
\end{equation}
Differentiating $\chi=\varphi e$ gives
\begin{equation}
\label{eq:ChiDynamics}
\dot\chi(t)
=
\tau_{\rm f}^{-1}\eta(t)
-
\alpha\gamma\bigl(\chi(t)\bigr)
-
\delta\chi(t).
\end{equation}
For the analysis, we first fix
$\sigma\in\{0,1\}$ and study the corresponding continuous closed-loop
dynamics separately. The gain then satisfies
$\dot k_{\rm a}(t)=\nu\sigma$, where $\sigma=0$ denotes the holding
mode and $\sigma=1$ the gain-increase mode. After global trajectories
have been established for both fixed modes, they are concatenated
according to the switching construction in
Subsection~\ref{subsec:SupervisedGain}.

\subsection{Auxiliary estimates and lifting operators}

We first recall the properties of the barrier map
\eqref{eq:BarrierMap} that will be used below.

\begin{lmm}[Properties of the barrier map]
\label{lem:StrongMonotonicityGamma}
The map $\gamma$ in \eqref{eq:BarrierMap} is bijective and maximally
monotone. Moreover,
for all
$\xi_1,\xi_2\in\mathcal B_1^\U(0)$,
\begin{equation}
\label{eq:StrongMonotonicityGamma}
\scprod{
\gamma(\xi_1)-\gamma(\xi_2)
}{
\xi_1-\xi_2
}_\U
\geq
\|\xi_1-\xi_2\|_\U^2.
\end{equation}
In particular,
$\gamma^{-1}:\U\to\mathcal B_1^\U(0)$
is globally Lipschitz continuous with Lipschitz constant one.
\end{lmm}

The proof is given in Appendix~\ref{app:AuxiliaryResults}.

\begin{rmrk}
\label{rem:UniqueInitialInput}
The element $u_0$ in
Assumption~\ref{assum:InitialValues} is unique.

Indeed, suppose that $u_{0,1},u_{0,2}\in\U$ both satisfy the
conditions in Assumption~\ref{assum:InitialValues}. For $j=1,2$, set
$v_{0,j} := C\&D\spvek{x_0}{u_{0,j}}$
and
\[
e_{{\rm r},0,j}
:={}
v_{0,j}
-
\dot y_{\rm ref}(0)
+
\left(
\frac{\alpha}
{1-\varphi(0)^2
\|y_0-y_{\rm ref}(0)\|_\U^2}
+
q(0)
\right)
\bigl(
y_0-y_{\rm ref}(0)
\bigr).
\]
Define
$\eta_{0,j} := \tau_{\rm f}\varphi(0)e_{{\rm r},0,j}$.
Then
\[
\eta_{0,1}-\eta_{0,2}
=
\tau_{\rm f}\varphi(0)
\bigl(
v_{0,1}-v_{0,2}
\bigr).
\]
The initial controller compatibility condition gives
\[
\begin{aligned}
u_{0,1}-u_{0,2}
=
-\frac{1}
{\tau_{\rm f}\varphi(0)}
\biggl[
k_{{\rm a},0}
\bigl(
\eta_{0,1}-\eta_{0,2}
\bigr)
+
\beta
\bigl(
\gamma(\eta_{0,1})
-
\gamma(\eta_{0,2})
\bigr)
\biggr].
\end{aligned}
\]
Since $\dom(S)$ is a linear subspace, impedance passivity applied to
the difference of the corresponding system-node elements yields
\[
0
\leq
\scprod{
v_{0,1}-v_{0,2}
}{
u_{0,1}-u_{0,2}
}_\U
=
-\frac{1}
{\tau_{\rm f}^2\varphi(0)^2}
\Big(
k_{{\rm a},0}
\|\eta_{0,1}-\eta_{0,2}\|_\U^2
+
\beta
\scprod{
\eta_{0,1}-\eta_{0,2}
}{
\gamma(\eta_{0,1})
-
\gamma(\eta_{0,2})
}_\U
\Big).
\]
It follows from
\eqref{eq:StrongMonotonicityGamma} that
$\eta_{0,1}=\eta_{0,2}$.
Consequently,
$v_{0,1}=v_{0,2}$ and $u_{0,1}=u_{0,2}$.
\end{rmrk}

By Assumption~\ref{assum:SystemClass} and
Proposition~\ref{prop:ObservationRightInverse}, we fix a bounded
right inverse
$Q\in\Lb(\U,\dom(A))$, $CQ=I_\U$.
In particular,
$Q\in\Lb(\U,\X)$, $AQ\in\Lb(\U,\X)$.
By Remark~\ref{rem:PositiveResolvent}, we may fix some $\mu>0$.
We also fix a corresponding coercivity constant $m_\mu>0$ in
\eqref{eq:TransferFunctionCoercivity} and set
\begin{equation}
\begin{aligned}
R_\mu
&:=
(\mu I-A_{-1})^{-1}
\in\Lb(\X_{-1},\X),
\\
L_\mu
&:=
R_\mu B-QP(\mu)
\in\Lb(\U,\X),
\\
G_\mu
&:=
\mu R_\mu B-AQP(\mu)
\in\Lb(\U,\X).
\label{eq:GMuDefinition}
\end{aligned}
\end{equation}

\begin{lmm}[Zero-output lifting]
\label{lem:ZeroOutputLifting}
Suppose that Assumption~\ref{assum:SystemClass} holds, and let
$Q$, $\mu$, $R_\mu$, $L_\mu$, and $G_\mu$ be fixed as above.
Then, for every $p\in\U$,
\[
\spvek{L_\mu p}{p}
\in\dom(S),
\qquad
\sbvek{A\&B}{C\&D}
\spvek{L_\mu p}{p}
=
\spvek{G_\mu p}{0}.
\]
\end{lmm}
The proof is given in Appendix~\ref{app:AuxiliaryResults}.

\subsection{The transformed closed-loop system}

Besides the normalized errors $\chi$ and $\eta$, we introduce the
transformed state
\begin{equation}
\label{eq:TransformedState}
\widehat z(t)
:={}
\tau_{\rm f}\varphi(t)
\left(
x(t)
-
L_\mu u_{\ext}(t)
-
Q\dot y_{\rm ref}(t)
\right)
+
\tau_{\rm f}Qq(t)\chi(t)
+
\tau_{\rm f}\alpha Q\gamma\bigl(\chi(t)\bigr).
\end{equation}
By \eqref{eq:NormalizedController},
\begin{equation}
\label{eq:TransformationDomainIdentity}
\spvek{
\widehat z
}{
-k_{\rm a}\eta-\beta\gamma(\eta)
}
={}
\tau_{\rm f}\varphi
\spvek{x}{u}
-
\tau_{\rm f}\varphi
\spvek{L_\mu u_{\ext}}{u_{\ext}}
-
\tau_{\rm f}\varphi
\spvek{Q\dot y_{\rm ref}}{0}
+
\tau_{\rm f}
\spvek{Qq\chi}{0}
+
\tau_{\rm f}\alpha
\spvek{Q\gamma(\chi)}{0}.
\end{equation}
Lemma~\ref{lem:ZeroOutputLifting}, the identity $CQ=I_\U$, and
linearity of the system node show that the element on the left-hand
side belongs to $\dom(S)$ almost everywhere.

Applying the system node to
\eqref{eq:TransformationDomainIdentity} and using
\eqref{eq:NormalizedVelocityRelation} gives, for almost every
$t\geq0$ and some $\widehat f(t)\in\X$,
\begin{equation}
\label{eq:TransformedSystemNode}
\spvek{
\widehat f(t)
}{
\eta(t)
}
=
\sbvek{A\&B}{C\&D}
\spvek{
\widehat z(t)
}{
-k_{\rm a}(t)\eta(t)
-\beta\gamma\bigl(\eta(t)\bigr)
}.
\end{equation}

For the state equation, define
\[
\mathfrak d(t)
:={}
\tau_{\rm f}\varphi(t)
\bigl(
G_\mu u_{\ext}(t)
-
L_\mu\dot u_{\ext}(t)
+
AQ\dot y_{\rm ref}(t)
-
Q\ddot y_{\rm ref}(t)
\bigr).
\]
For $t\geq0$ and $\chi\in\mathcal B_1^\U(0)$, define
$K(t,\chi) := q(t)Q + \alpha Q\gamma'(\chi) \in\Lb(\U,\X)$.
Differentiating \eqref{eq:TransformedState} and using
\eqref{eq:ChiDynamics} and
\eqref{eq:TransformedSystemNode}, we obtain
\begin{equation}
\label{eq:TransformedClosedLoop}
\begin{aligned}
\dot{\widehat z}
={}&
\widehat f
+
K(t,\chi)\eta
+
(q-\delta)\widehat z
+
\mathfrak d
+
\tau_{\rm f}
\bigl(
\dot q-q^2
\bigr)Q\chi
-
\tau_{\rm f}qAQ\chi
\\
&-
\tau_{\rm f}\alpha
AQ\gamma(\chi)
-
\tau_{\rm f}\alpha
\bigl(
2q-\delta
\bigr)
Q\gamma(\chi)
-
\tau_{\rm f}\alpha^2
Q\gamma'(\chi)\gamma(\chi)
-
\tau_{\rm f}\alpha\delta
Q\gamma'(\chi)\chi,
\\
\dot\chi
={}&
\tau_{\rm f}^{-1}\eta
-
\alpha\gamma(\chi)
-
\delta\chi,
\\
\dot k_{\rm a}
={}&
\nu\sigma,
\\
\spvek{
\widehat f
}{
\eta
}
={}&
\sbvek{A\&B}{C\&D}
\spvek{
\widehat z
}{
-k_{\rm a}\eta-\beta\gamma(\eta)
}.
\end{aligned}
\end{equation}
Assumption~\ref{assum:FunnelData} implies
$\mathfrak d \in \Wkp{1,\infty}(\R_{\geq0};\X)$.
Moreover,
$q$ and
$\dot q-q^2$
are bounded and Lipschitz continuous. The principal system-node
relation in \eqref{eq:TransformedClosedLoop} contains no explicit
time dependence. The remaining time dependence occurs only through
bounded Lipschitz-continuous coefficients and forcing terms.

\begin{lmm}[Smooth truncation of the barrier map]
\label{lem:TruncatedBarrier}
For every $R\in(0,1)$, there exists a continuously Fr\'echet
differentiable map
$\Gamma_R:\U\to\U$
with the following properties:
\begin{enumerate}[label=(\alph{*})]
\item
$\Gamma_R(\xi)=\gamma(\xi)$ for every
$\xi\in\overline{\mathcal B_R^\U(0)}$;
\item
$\Gamma_R$ is globally Lipschitz continuous and, for all
$\xi_1,\xi_2\in\U$,
\[
\scprod{
\Gamma_R(\xi_1)-\Gamma_R(\xi_2)
}{
\xi_1-\xi_2
}_\U
\geq
\|\xi_1-\xi_2\|_\U^2;
\]
\item
$\Gamma_R':\U\to\Lb(\U)$ is globally bounded and globally
Lipschitz continuous;
\item
the mappings
$\xi\longmapsto\Gamma_R'(\xi)\Gamma_R(\xi)$,
$\xi\longmapsto\Gamma_R'(\xi)\xi$
are globally Lipschitz continuous.
\end{enumerate}
\end{lmm}
The proof is given in Appendix~\ref{app:AuxiliaryResults}.

\subsection{A product operator on the extended state space}

The transformation introduces the derivative $\gamma'(\chi)$. We
therefore use the differentiable truncation from
Lemma~\ref{lem:TruncatedBarrier}. For $R\in(0,1)$, let
$\Gamma_R:\U\to\U$
be the corresponding truncation, and set
$K_R(t,\chi) := q(t)Q + \alpha Q\Gamma_R'(\chi)$.

Define the extended Hilbert space
$\mathcal H_{\rm a} := \X\times\U\times\R$
with its natural inner product. To define the principal
operator on the entire real gain axis, set
$\kappa(k) := \max\{k,0\}$, $k\in\R$.
Along the actual closed-loop trajectory, the gain is nonnegative and
therefore $\kappa(k_{\rm a})=k_{\rm a}$.

For
$w=(\widehat z,\chi,k)\in\mathcal H_{\rm a}$,
$\widehat f\in\X$, and
$\eta\in\mathcal B_1^\U(0)$, consider the relation
\begin{equation}
\label{eq:ExtendedTruncatedSystemNodeRelation}
\spvek{
\widehat z
}{
-\kappa(k)\eta-\beta\gamma(\eta)
}
\in\dom(S),
\qquad
\spvek{
\widehat f
}{
\eta
}
=
\sbvek{A\&B}{C\&D}
\spvek{
\widehat z
}{
-\kappa(k)\eta-\beta\gamma(\eta)
}.
\end{equation}
Define the fixed set
\[
\mathcal D_{\rm a}
:=
\setdef*{
(\widehat z,\chi,k)\in\mathcal H_{\rm a}
}{
\begin{array}{l}
\text{there exist }\widehat f\in\X
\text{ and }\eta\in\mathcal B_1^\U(0)
\\
\text{satisfying
\eqref{eq:ExtendedTruncatedSystemNodeRelation}}
\end{array}
}.
\]
For $R\in(0,1)$ and $t\geq0$, define
\[
\mathcal A_{R,{\rm a}}(t):
\mathcal D_{\rm a}
\subset\mathcal H_{\rm a}
\to
\mathcal H_{\rm a}
\]
by
\[
\mathcal A_{R,{\rm a}}(t)
\begin{pmatrix}
\widehat z\\
\chi\\
k
\end{pmatrix}
:=
\begin{pmatrix}
-\widehat f-K_R(t,\chi)\eta\\
-\tau_{\rm f}^{-1}\eta\\
0
\end{pmatrix},
\]
where $\widehat f$ and $\eta$ satisfy
\eqref{eq:ExtendedTruncatedSystemNodeRelation}.
The domain $\mathcal D_{\rm a}$ contains neither $t$ nor $R$ and is
therefore common to all operators
$\mathcal A_{R,{\rm a}}(t)$.
\begin{prpstn}[Extended truncated product operator]
\label{prop:TruncatedProductOperator}
Suppose that Assumptions~\ref{assum:SystemClass} and
\ref{assum:FunnelData} hold, and let $R\in(0,1)$.
Then the elements $\widehat f$ and $\eta$ in
\eqref{eq:ExtendedTruncatedSystemNodeRelation}
are uniquely determined. Consequently,
$\mathcal A_{R,{\rm a}}(t)$ is single-valued for every
$t\geq0$.

Moreover, there exists a constant $\omega_R\geq0$, independent of
$t$, such that
$\mathcal M_{R,{\rm a}}(t)
:=
\mathcal A_{R,{\rm a}}(t)
+
\omega_RI_{\mathcal H_{\rm a}}$
is maximally monotone for every $t\geq0$. Its domain is the fixed set
$\mathcal D_{\rm a}$.
\end{prpstn}

The proof is given in Appendix~\ref{app:TruncatedProductOperator}.
The key point is to incorporate the gain-dependent
linear term into the feedback map associated with the normalized
recursive error. More precisely, we use
$\Phi_k:
\mathcal B_1^\U(0)\to\U$,
$\Phi_k(\eta)
:=
\kappa(k)\eta+\beta\gamma(\eta)$.
For every $k\in\R$, this map is bijective and satisfies
\[
\scprod{
\Phi_k(\eta_1)-\Phi_k(\eta_2)
}{
\eta_1-\eta_2
}_\U
\geq
\beta
\|\eta_1-\eta_2\|_\U^2.
\]
The mixed terms caused by different gain values are controlled by the
additional scalar component of the product-space norm. The resulting
hypomonotonicity constant is independent of the magnitude of $k$.

For
$w=(\widehat z,\chi,k)\in\mathcal H_{\rm a}$,
define
\[
\mathcal F_{R,{\rm a}}(t,w)
:=
\begin{pmatrix}
\mathcal F_{R,1}(t,\widehat z,\chi)\\
-\alpha\Gamma_R(\chi)-\delta\chi\\
0
\end{pmatrix},
\]
where
\begin{multline*}
\mathcal F_{R,1}(t,\widehat z,\chi)
:=
\bigl(q(t)-\delta\bigr)\widehat z
+
\tau_{\rm f}
\bigl(
\dot q(t)-q(t)^2
\bigr)Q\chi
-
\tau_{\rm f}q(t)AQ\chi
\\
-
\tau_{\rm f}\alpha AQ\Gamma_R(\chi)
-
\tau_{\rm f}\alpha
\bigl(
2q(t)-\delta
\bigr)
Q\Gamma_R(\chi)
-
\tau_{\rm f}\alpha^2
Q\Gamma_R'(\chi)\Gamma_R(\chi)
-
\tau_{\rm f}\alpha\delta
Q\Gamma_R'(\chi)\chi.
\end{multline*}
By Lemma~\ref{lem:TruncatedBarrier},
$\mathcal F_{R,{\rm a}}(t,\cdot)$ is globally Lipschitz continuous on
$\mathcal H_{\rm a}$, uniformly in $t\geq0$. Moreover,
$t\mapsto\mathcal F_{R,{\rm a}}(t,w)$ is Lipschitz continuous
uniformly for $w$ in bounded subsets of $\mathcal H_{\rm a}$.

For $\sigma\in\{0,1\}$, define
\[
g_\sigma(t)
:=
\begin{pmatrix}
\mathfrak d(t)\\
0\\
\nu\sigma
\end{pmatrix}.
\]

\begin{prpstn}[Truncated evolution in a fixed mode]
\label{prop:TruncatedWellPosedness}
Suppose that Assumptions~\ref{assum:SystemClass} and
\ref{assum:FunnelData} hold. Let
$R\in(0,1)$,
$0\leq t_0<T$,
$\sigma\in\{0,1\}$, and
$w_0\in\mathcal D_{\rm a}$.
Then the evolution equation
\begin{equation}
\label{eq:TruncatedEvolutionEquation}
\dot w_R(t)
+
\mathcal M_{R,{\rm a}}(t)w_R(t)
={}
\mathcal F_{R,{\rm a}}\bigl(t,w_R(t)\bigr)
+
\omega_Rw_R(t)
+
g_\sigma(t),
\qquad
w_R(t_0)
={}
w_0,
\end{equation}
has a unique solution
$w_R \in \Wkp{1,\infty} \bigl( [t_0,T]; \mathcal H_{\rm a} \bigr)$.
Moreover,
$w_R(t)\in\mathcal D_{\rm a}$ for almost every $t\in[t_0,T]$,
and
\[
\mathcal A_{R,{\rm a}}(\cdot)w_R(\cdot)
\in
\Lp{\infty}
\bigl(
[t_0,T];
\mathcal H_{\rm a}
\bigr).
\]
\end{prpstn}

The proof is given in
Appendix~\ref{app:TruncatedEvolutionEquation}. The maximal monotone
part is identical in both modes. The mode parameter occurs only in
the bounded forcing term $g_\sigma$.

Writing
$w_R=(\widehat z_R,\chi_R,k_R)$,
let $\widehat f_R$ and $\eta_R$ denote the uniquely determined
elements in
\eqref{eq:ExtendedTruncatedSystemNodeRelation}. Expanding
\eqref{eq:TruncatedEvolutionEquation} gives
\begin{equation}
\label{eq:ExpandedExtendedTruncatedSystem}
\begin{aligned}
\dot{\widehat z}_R
&=
\widehat f_R
+
K_R(t,\chi_R)\eta_R
+
\mathcal F_{R,1}
\bigl(
t,\widehat z_R,\chi_R
\bigr)
+
\mathfrak d,
\\
\dot\chi_R
&=
\tau_{\rm f}^{-1}\eta_R
-
\alpha\Gamma_R(\chi_R)
-
\delta\chi_R,
\\
\dot k_R
&=
\nu\sigma,
\\
\spvek{
\widehat f_R
}{
\eta_R
}
&=
\sbvek{A\&B}{C\&D}
\spvek{
\widehat z_R
}{
-\kappa(k_R)\eta_R-\beta\gamma(\eta_R)
}.
\end{aligned}
\end{equation}

The following pointwise domain property permits restarting the
continuous evolution at a switching time.

\begin{lmm}[Pointwise domain property]
\label{lem:PointwiseDomainProperty}
Under the assumptions of
Proposition~\ref{prop:TruncatedWellPosedness}, the solution
$w_R$ of \eqref{eq:TruncatedEvolutionEquation} satisfies
$w_R(t)\in\mathcal D_{\rm a}$ for every $t\in[t_0,T]$.
\end{lmm}

\begin{proof}
By Proposition~\ref{prop:TruncatedWellPosedness}, there exists a set
$E\subset[t_0,T]$ of full measure such that
$w_R(t)\in\mathcal D_{\rm a}$ for all $t\in E$,
and
$\mathcal A_{R,{\rm a}}(t)w_R(t)$
is uniformly bounded on $E$.

Fix $t\in[t_0,T]$ and choose a sequence
$(t_n)_{n\in\N}\subset E$ with $t_n\to t$. Since
$w_R$ is continuous,
$w_R(t_n)\to w_R(t)$ in $\mathcal H_{\rm a}$.
Since
$\mathcal A_{R,{\rm a}}(\cdot)w_R(\cdot)$
is uniformly bounded on $E$, the sequence
\[
\bigl(
\mathcal A_{R,{\rm a}}(t_n)w_R(t_n)
\bigr)_{n\in\N}
\]
is bounded in $\mathcal H_{\rm a}$. After passing to a subsequence,
it therefore converges weakly to some $a\in\mathcal H_{\rm a}$.

The explicit time dependence of
$\mathcal A_{R,{\rm a}}(t)$ occurs only through $q(t)$ in
$K_R(t,\chi)$. Since the corresponding element $\eta$ satisfies
$\|\eta\|_\U<1$, there exists some $c_R>0$ such that, for all
$w\in\mathcal D_{\rm a}$ and $r,s\geq0$,
\[
\left\|
\mathcal A_{R,{\rm a}}(r)w
-
\mathcal A_{R,{\rm a}}(s)w
\right\|_{\mathcal H_{\rm a}}
\leq
c_R|r-s|.
\]
Hence
\[
\mathcal A_{R,{\rm a}}(t)w_R(t_n)
-
\mathcal A_{R,{\rm a}}(t_n)w_R(t_n)
\longrightarrow0
\quad\text{in }\mathcal H_{\rm a},
\]
and consequently
$\mathcal A_{R,{\rm a}}(t)w_R(t_n)\rightharpoonup a$.
Moreover,
\[
\bigl(
\mathcal A_{R,{\rm a}}(t)+\omega_RI_{\mathcal H_{\rm a}}
\bigr)w_R(t_n)
\rightharpoonup
a+\omega_Rw_R(t).
\]
The graph of a maximally monotone operator is strongly-weakly closed.
It follows that
$w_R(t)\in\mathcal D_{\rm a}$.
\end{proof}

\subsection{Equivalence of the closed-loop formulations}

\begin{prpstn}[Closed-loop equivalence]
\label{prop:ClosedLoopTransformation}
Suppose that Assumptions~\ref{assum:SystemClass} and
\ref{assum:FunnelData} hold, and let
$\mathcal I\subset\R_{\geq0}$ be an interval.

Suppose first that $(x,u,v)$ is a strong trajectory of
\eqref{eq:VelocityNode} on $\mathcal I$, that
$y\in\Wkp{1,\infty}_{\loc}(\mathcal I;\U)$,
$k_{\rm a}
\in
\Wkp{1,\infty}_{\loc}(\mathcal I)$,
$k_{\rm a}\geq0$,
and that
$\dot y=v$, $e=y-y_{\rm ref}$,
and \eqref{eq:ClosedLoopControlVelocity} hold almost everywhere.
Assume that
$\|\chi(t)\|_\U<1$ for all $t\in\mathcal I$,
and
$\|\eta(t)\|_\U<1$ for almost every $t\in\mathcal I$.
Define $\widehat z$ by \eqref{eq:TransformedState}. Then there exists
$\widehat f \in \Lp{\infty}_{\loc}(\mathcal I;\X)$
such that the first, second, and fourth relations in
\eqref{eq:TransformedClosedLoop} hold almost everywhere.

Conversely, suppose that
$\widehat z
\in
\Wkp{1,\infty}_{\loc}(\mathcal I;\X)$,
$\chi
\in
\Wkp{1,\infty}_{\loc}(\mathcal I;\U)$,
$k_{\rm a}
\in
\Wkp{1,\infty}_{\loc}(\mathcal I)$,
and that there exist
\[
\widehat f
\in
\Lp{\infty}_{\loc}(\mathcal I;\X),
\qquad
\eta
\in
\Lp{\infty}_{\loc}(\mathcal I;\U)
\]
satisfying the first, second, and fourth relations in
\eqref{eq:TransformedClosedLoop}. Assume that
$k_{\rm a}\geq0$,
$\|\chi(t)\|_\U<1$ for all $t\in\mathcal I$, and
$\|\eta(t)\|_\U<1$ for almost every $t\in\mathcal I$.
Define
\begin{equation}
\label{eq:ReconstructionFormulas}
\begin{aligned}
x
:={}&
\frac{1}{\tau_{\rm f}\varphi}
\left(
\widehat z
-
\tau_{\rm f}Qq\chi
-
\tau_{\rm f}\alpha Q\gamma(\chi)
\right)
+
L_\mu u_{\ext}
+
Q\dot y_{\rm ref},
\\
e
:={}&
\varphi^{-1}\chi,
\qquad
y
:=
y_{\rm ref}+e,
\qquad
e_{\rm r}
:=
\frac{1}{\tau_{\rm f}\varphi}\eta,
\\
u
:={}&
u_{\ext}
-
\frac{1}{\tau_{\rm f}\varphi}
\left(
k_{\rm a}\eta
+
\beta\gamma(\eta)
\right),
\\
v
:={}&
\dot y_{\rm ref}
+
\frac{1}{\tau_{\rm f}\varphi}
\left(
\eta
-
\tau_{\rm f}\alpha\gamma(\chi)
-
\tau_{\rm f}q\chi
\right).
\end{aligned}
\end{equation}
Then $(x,u,v)$ is a strong trajectory of
\eqref{eq:VelocityNode},
$y\in\Wkp{1,\infty}_{\loc}(\mathcal I;\U)$,
and $\dot y=v$ and the first three relations in
\eqref{eq:ClosedLoopControlVelocity} hold almost everywhere.
The two transformations are inverse to each other.
\end{prpstn}

The proof is given in
Appendix~\ref{app:ClosedLoopTransformation}.

\subsection{A priori funnel estimates}

We first observe that the normalized position error admits an estimate
which is independent of the system state and of the supervised gain.
\begin{lmm}[Uniform position-funnel margin]
\label{lem:UniformPositionFunnelMargin}
Let $\mathcal I\subset\R_{\geq0}$ be an interval and let
$t_0\in\mathcal I$. Suppose that
$\chi
\in
\Wkp{1,\infty}_{\loc}(\mathcal I;\U)$,
$\eta
\in
\Lp{\infty}_{\loc}(\mathcal I;\U)$
satisfy
$\dot\chi = \tau_{\rm f}^{-1}\eta - \alpha\gamma(\chi) - \delta\chi$
almost everywhere and
$\|\chi(t)\|_\U<1$ for all $t\in\mathcal I$,
as well as
$\|\eta(t)\|_\U<1$ for almost every $t\in\mathcal I$.
Then, for every $t\in\mathcal I$ with $t\geq t_0$,
\[
\|\chi(t)\|_\U
\leq
\max\{
\|\chi(t_0)\|_\U,r_*
\},
\]
where $r_*\in(0,1)$ is the unique number satisfying
\begin{equation}
\label{eq:InnerMarginRadius}
\frac{r_*}{1-r_*^2}
=
\frac{1}{\alpha\tau_{\rm f}}.
\end{equation}
\end{lmm}
\begin{proof}
The function
$t\longmapsto\frac12\|\chi(t)\|_\U^2$
is locally absolutely continuous, and
\[
\frac12\frac{\dd}{\dd t}\|\chi\|_\U^2
=
\tau_{\rm f}^{-1}
\scprod{\chi}{\eta}_\U
-
\alpha
\scprod{\chi}{\gamma(\chi)}_\U
-
\delta\|\chi\|_\U^2
\leq
\tau_{\rm f}^{-1}\|\chi\|_\U
-
\alpha
\frac{\|\chi\|_\U^2}
{1-\|\chi\|_\U^2}.
\]
The right-hand side is nonpositive whenever
$\|\chi\|_\U\geq r_*$. A first-exit argument proves the assertion.
\end{proof}
We next show that the normalized recursive error remains
strictly separated from the boundary of the unit ball on every finite
time interval. The system-node relation yields a transfer identity in
which the coercivity of $P(\mu)$ provides a bound for
$\gamma(\eta)$. Since $\|\gamma(\eta)\|_\U$ tends to infinity as
$\|\eta\|_\U$ approaches one, this bound implies the desired finite-horizon
funnel margin. The term involving the nonnegative gain $k$ contributes
with the same favorable sign and therefore does not weaken the estimate.

\begin{prpstn}[Finite-horizon recursive-error funnel margin]
\label{prop:FiniteHorizonRecursiveErrorMargin}
Suppose that Assumptions~\ref{assum:SystemClass} and
\ref{assum:FunnelData} hold, and let
$0\leq t_0<T$. Suppose that
$\widehat z
\in
\Wkp{1,\infty}([t_0,T];\X)$,
$\widehat f
\in
\Lp{\infty}([t_0,T];\X)$,
$k \in \Lp{\infty}([t_0,T])$, $k\geq0$,
and
$\eta \in \Lp{\infty}([t_0,T];\U)$
satisfy
\begin{equation}
\label{eq:FiniteHorizonGainNodeRelation}
\spvek{
\widehat f
}{
\eta
}
=
\sbvek{A\&B}{C\&D}
\spvek{
\widehat z
}{
-k\eta-\beta\gamma(\eta)
}
\end{equation}
almost everywhere on $[t_0,T]$. Assume additionally that
$\|\eta(t)\|_\U<1$
for almost every $t\in[t_0,T]$. Then
$\gamma(\eta) \in \Lp{\infty}([t_0,T];\U)$.
Consequently, there exists some
$\rho\in(0,1)$ such that
$\|\eta(t)\|_\U \leq \rho$
for almost every $t\in[t_0,T]$.
\end{prpstn}

\begin{proof}
The resolvent representation of
\eqref{eq:FiniteHorizonGainNodeRelation} gives
\begin{equation}
\label{eq:TransferIdentityForEta}
P(\mu)
\left(
k\eta+\beta\gamma(\eta)
\right)
+
\eta
=
C(\mu I-A)^{-1}
\bigl(
\mu\widehat z-\widehat f
\bigr).
\end{equation}
The right-hand side belongs to
$\Lp{\infty}([t_0,T];\U)$. Hence its norm is bounded by some
$C_T>0$ almost everywhere.

Since
$\eta = \bigl( 1-\|\eta\|_\U^2 \bigr)\gamma(\eta)$,
we have
\[
\scprod{P(\mu)\eta}{\gamma(\eta)}_\U
=
\bigl(
1-\|\eta\|_\U^2
\bigr)
\scprod{
P(\mu)\gamma(\eta)
}{
\gamma(\eta)
}_\U
\geq0.
\]
Taking the inner product of the left-hand side of
\eqref{eq:TransferIdentityForEta} with $\gamma(\eta)$, estimating the
right-hand side by $C_T\|\gamma(\eta)\|_\U$, and using
$k\geq0$ and
\eqref{eq:TransferFunctionCoercivity} gives
\[
\beta m_\mu \|\gamma(\eta)\|_\U^2
\leq
C_T\|\gamma(\eta)\|_\U
\]
almost everywhere. Thus
$\|\gamma(\eta)\|_\U \leq \frac{C_T}{\beta m_\mu}$
almost everywhere. The strict recursive-error funnel margin follows from
$\|\gamma(\xi)\|_\U\to\infty$ as $\|\xi\|_\U\to1$.
\end{proof}

\subsection{Global evolution in a fixed supervisor mode}

We next remove the position-error truncation and obtain a global trajectory for
each fixed value of $\sigma$.

\begin{prpstn}[Global trajectory in a fixed mode]
\label{prop:FixedModeGlobalTrajectory}
Suppose that Assumptions~\ref{assum:SystemClass} and
\ref{assum:FunnelData} hold. Let
$t_0\geq0$, $\sigma\in\{0,1\}$, and
$w_0=(\widehat z_0,\chi_0,k_0)\in\mathcal D_{\rm a}$ satisfy
$k_0\geq0$ and
$\|\chi_0\|_\U<1$.
Then there exists a unique global transformed trajectory
$\widehat z \in \Wkp{1,\infty}_{\loc} \bigl( [t_0,\infty); \X \bigr)$, $\chi
\in
\Wkp{1,\infty}_{\loc}
\bigl(
[t_0,\infty);
\U
\bigr)$,
$k
\in
\Wkp{1,\infty}_{\loc}
\bigl(
[t_0,\infty)
\bigr)$,
$\widehat f
\in
\Lp{\infty}_{\loc}
\bigl(
[t_0,\infty);
\X
\bigr)$,
and $\eta
\in
\Lp{\infty}_{\loc}
\bigl(
[t_0,\infty);
\U
\bigr)$
satisfying \eqref{eq:TransformedClosedLoop} with
$k_{\rm a}=k$ and the fixed mode $\sigma$.

Moreover,
$k(t) = k_0+\nu\sigma(t-t_0)$,
and
$\|\chi(t)\|_\U \leq \max\{\|\chi_0\|_\U,r_*\} <1$
for all $t\geq t_0$. For every $T>t_0$, there exists some
$\rho\in(0,1)$ such that
$\|\eta(t)\|_\U \leq \rho$
for almost every $t\in[t_0,T]$.
\end{prpstn}

\begin{proof}
Set
$\rho_\chi := \max\{\|\chi_0\|_\U,r_*\} <1$
and choose
$R\in(\rho_\chi,1)$.
For $T>t_0$, let
$w_R=(\widehat z_R,\chi_R,k_R)$
be the solution of the truncated evolution equation on $[t_0,T]$
provided by
Proposition~\ref{prop:TruncatedWellPosedness}. Let
$\widehat f_R$ and $\eta_R$ be the associated elements in
\eqref{eq:ExpandedExtendedTruncatedSystem}.

Define
\[
t_R
:=
\inf
\setdef{
t\in[t_0,T]
}{
\|\chi_R(t)\|_\U\geq R
},
\]
where $\inf\emptyset:=T$. On $[t_0,t_R)$, the truncated trajectory
satisfies the untruncated transformed equations. Hence
Lemma~\ref{lem:UniformPositionFunnelMargin} gives
$\|\chi_R(t)\|_\U\leq\rho_\chi$ for all $t\in[t_0,t_R)$.
If $t_R<T$, continuity of $\chi_R$ would imply
$R \leq \|\chi_R(t_R)\|_\U \leq \rho_\chi < R$,
which is impossible. Hence $t_R=T$, and the truncated trajectory
solves the untruncated transformed system on $[t_0,T]$.

The boundedness of
$\mathcal A_{R,{\rm a}}(\cdot)w_R(\cdot)$
implies
$\eta_R\in\Lp{\infty}([t_0,T];\U)$
and
$\widehat f_R\in\Lp{\infty}([t_0,T];\X)$.
Proposition~\ref{prop:FiniteHorizonRecursiveErrorMargin} therefore gives a strict recursive-error funnel margin on $[t_0,T]$.
The radius $R$ is independent of $T$. Uniqueness of the truncated
evolution implies consistency of the finite-horizon trajectories,
which therefore define a global trajectory on $[t_0,\infty)$.

To prove uniqueness, let a second transformed trajectory with the
stated properties be given and fix $T>t_0$. By continuity of the two
normalized position errors, there exists some $R_T\in(0,1)$ such that
both remain in $\mathcal B_{R_T}^{\U}(0)$ on $[t_0,T]$. Hence both
trajectories solve the same truncated evolution equation on this
interval and therefore coincide. Since $T>t_0$ was arbitrary, the
global trajectory is unique.
\end{proof}

\subsection{The supervised gain}
\label{subsec:SupervisedGain}

We now construct a single supervised closed-loop trajectory
from the global fixed-mode trajectories provided by
Proposition~\ref{prop:FixedModeGlobalTrajectory}. Starting in the
holding mode, we determine the next trigger time and then concatenate
the holding-mode trajectory with a gain-increase trajectory of fixed
duration. This procedure is repeated recursively.
Let
$T_{\rm inc} = \frac{\Delta k}{\nu}$. Set
$s_0:=0$, $k_0:=k_{{\rm a},0}$.
The time $s_n$ denotes the beginning of the $n$th holding phase and
$k_n$ is the gain value during this phase.

Suppose that the closed-loop trajectory has been constructed up to
$s_n$ and that
$k_{\rm a}(s_n)=k_n$.
Starting from the state at $s_n$, consider the unique global
trajectory in the holding mode $\sigma=0$ provided by
Proposition~\ref{prop:FixedModeGlobalTrajectory}. Define
\begin{equation}
\label{eq:SupervisorTriggerTime}
\tau_n
:=
\inf
\left\{
t>s_n:
\operatorname*{ess\,sup}_{s\in(s_n,t)}
\|\eta(s)\|_\U
\geq
r_{\rm a}
\right\},
\end{equation}
where the infimum of the empty set is understood as infinity.

If $\tau_n=\infty$, the supervisor remains in the holding mode for all
subsequent times and the construction terminates. If
$\tau_n<\infty$, the holding-mode trajectory is retained on
$[s_n,\tau_n]$. Starting from its value at $\tau_n$, the gain-increase
mode $\sigma=1$ is applied on
$[\tau_n,s_{n+1}]$, $s_{n+1} := \tau_n+T_{\rm inc}$.
During this interval,
$k_{\rm a}(t) = k_n+\nu(t-\tau_n)$,
and therefore
$k_{\rm a}(s_{n+1}) = k_n+\Delta k =: k_{n+1}$.
The supervisor then returns to the holding mode and the procedure is
repeated.

By Lemma~\ref{lem:PointwiseDomainProperty}, the transformed state
belongs to $\mathcal D_{\rm a}$ at every switching time. Hence each
new continuous mode starts from an admissible initial value.

The use of the essential supremum in
\eqref{eq:SupervisorTriggerTime} makes the switching rule independent
of the representative chosen for the essentially bounded signal
$\eta$. If the monitoring level is already violated immediately after
the beginning of a holding phase, then $\tau_n=s_n$, and the next
gain-increase phase starts without a positive holding interval.

The beginning times of successive gain-increase phases satisfy
$\tau_{n+1} \geq \tau_n+T_{\rm inc}$.
Consequently, only finitely many gain-increase phases can begin on a
bounded time interval. In particular, the switching times cannot
accumulate in finite time.

\subsection{Global existence and funnel performance}

\begin{thrm}
\label{thm:Main_Thm_Exist_Uniq}
Suppose that Assumptions~\ref{assum:SystemClass},
\ref{assum:FunnelData},
\ref{assum:SupervisorData}, and
\ref{assum:InitialValues} hold.

Then there exists a unique quintuple
$(x,y,u,v,k_{\rm a})$
such that $(x,u,v)$ is a strong trajectory of
\eqref{eq:VelocityNode} on $\R_{\geq0}$,
$y
\in
\Wkp{1,\infty}(\R_{\geq0};\U)$,
$k_{\rm a}
\in
\Wkp{1,\infty}_{\loc}(\R_{\geq0})$,
$x(0)=x_0$, $y(0)=y_0$, $k_{\rm a}(0)=k_{{\rm a},0}$,
and, with
$e:=y-y_{\rm ref}$,
the relations
$\dot y=v$
and \eqref{eq:ClosedLoopControlVelocity} hold almost everywhere.
The gain $k_{\rm a}$ is generated by the recursive
supervisor construction in Subsection~\ref{subsec:SupervisedGain},
with the trigger times defined by
\eqref{eq:SupervisorTriggerTime}. The controller is well defined almost everywhere on
$\R_{\geq0}$. Moreover,
$\varphi(t)\|e(t)\|_\U<1$ for all $t\geq0$,
and
$\tau_{\rm f}\varphi(t)
\|e_{\rm r}(t)\|_\U
<1$ for almost every $t\geq0$.
In fact, there exists some $\varepsilon_\chi>0$ such that, for every
$t\geq0$,
\begin{equation}
\label{eq:UniformPositionFunnelDistance}
\varphi(t)\|e(t)\|_\U
\leq
1-\varepsilon_\chi.
\end{equation}
For every $T>0$, there exists some
$\varepsilon_{\eta,T}>0$ such that, for almost every $t\in[0,T]$,
\begin{equation}
\label{eq:FiniteHorizonRecursiveFunnelDistance}
\tau_{\rm f}\varphi(t)
\|e_{\rm r}(t)\|_\U
\leq
1-\varepsilon_{\eta,T}.
\end{equation}

Furthermore,
$0\leq\dot k_{\rm a}(t)\leq\nu$ for almost every $t\geq0$,
and, for every $t\geq0$,
\begin{equation}
\label{eq:FiniteHorizonGainBound}
0
\leq
k_{\rm a}(t)
\leq
k_{{\rm a},0}+\nu t.
\end{equation}
In particular,
$e,\dot e,e_{\rm r},v \in \Lp{\infty}(\R_{\geq0};\U)$,
whereas
$u
\in
\Lp{\infty}_{\loc}(\R_{\geq0};\U)$,
$k_{\rm a}
\in
\Lp{\infty}_{\loc}(\R_{\geq0})$.
\end{thrm}

\begin{proof}
Let $u_0$ be the element from
Assumption~\ref{assum:InitialValues}, and define
\[
\begin{aligned}
\chi_0
&:=
\varphi(0)
\bigl(
y_0-y_{\rm ref}(0)
\bigr),
\qquad
\eta_0
:=
\tau_{\rm f}\varphi(0)e_{\rm r}(0),
\\
\widehat z_0
&:={}
\tau_{\rm f}\varphi(0)
\left(
x_0
-
L_\mu u_{\ext}(0)
-
Q\dot y_{\rm ref}(0)
\right)
+
\tau_{\rm f}Qq(0)\chi_0
+
\tau_{\rm f}\alpha Q\gamma(\chi_0).
\end{aligned}
\]
By Assumption~\ref{assum:InitialValues},
$\|\chi_0\|_\U<1$, $\|\eta_0\|_\U<r_{\rm a}<1$.

The initial controller compatibility gives
\[
\tau_{\rm f}\varphi(0)
\bigl(
u_0-u_{\ext}(0)
\bigr)
=
-k_{{\rm a},0}\eta_0
-\beta\gamma(\eta_0).
\]
Using
$\spvek{x_0}{u_0}\in\dom(S)$,
Lemma~\ref{lem:ZeroOutputLifting}, and
$\im Q\subset\dom(A)$, we obtain
\[
\spvek{
\widehat z_0
}{
-k_{{\rm a},0}\eta_0-\beta\gamma(\eta_0)
}
\in\dom(S).
\]
By the definition of $e_{\rm r}(0)$, the corresponding system-node
output has second component $\eta_0$. Hence
$w_0:=(\widehat z_0,\chi_0,k_{{\rm a},0})
\in\mathcal D_{\rm a}$.

Starting from $w_0$, construct a single global transformed
trajectory by recursively concatenating the fixed-mode trajectories
described above: the holding-mode trajectory with $\sigma=0$ is used
until the next trigger time defined by
\eqref{eq:SupervisorTriggerTime}, and it is followed by the
gain-increase trajectory with $\sigma=1$ on the corresponding interval
of length $T_{\rm inc}$. Each continuous segment exists
globally by
Proposition~\ref{prop:FixedModeGlobalTrajectory}, and its terminal
state belongs to $\mathcal D_{\rm a}$ by
Lemma~\ref{lem:PointwiseDomainProperty}. Thus every mode change is
admissible.

Since every gain-increase phase has the fixed positive duration
$T_{\rm inc}$, only finitely many switches occur on each bounded time
interval. The continuous segments therefore define a trajectory
$(\widehat z,\chi,k_{\rm a})
\in
\Wkp{1,\infty}_{\loc}
\bigl(
\R_{\geq0};
\mathcal H_{\rm a}
\bigr)$.
The associated quantities
$\widehat f
\in
\Lp{\infty}_{\loc}(\R_{\geq0};\X)$,
$\eta
\in
\Lp{\infty}_{\loc}(\R_{\geq0};\U)$
satisfy the transformed closed-loop equations almost everywhere.

Let $r_*$ be given by
\eqref{eq:InnerMarginRadius} and set
$\rho_\chi := \max\{ \|\chi_0\|_\U,r_* \} <1$.
Lemma~\ref{lem:UniformPositionFunnelMargin}, applied successively on
the continuous mode intervals, gives
$\|\chi(t)\|_\U\leq\rho_\chi$ for all $t\geq0$.
Thus
\eqref{eq:UniformPositionFunnelDistance} holds with
$\varepsilon_\chi := 1-\rho_\chi$.

Fix $T>0$. Only finitely many switches occur on $[0,T]$. Hence
$\widehat z
\in
\Wkp{1,\infty}([0,T];\X)$,
$\widehat f
\in
\Lp{\infty}([0,T];\X)$,
and
$k_{\rm a} \in \Lp{\infty}([0,T])$, $k_{\rm a}\geq0$.
Proposition~\ref{prop:FiniteHorizonRecursiveErrorMargin} therefore yields
$\gamma(\eta) \in \Lp{\infty}([0,T];\U)$
and a strict recursive-error funnel margin on $[0,T]$. This proves
\eqref{eq:FiniteHorizonRecursiveFunnelDistance}.

Proposition~\ref{prop:ClosedLoopTransformation} reconstructs
$(x,y,u,v,k_{\rm a})$ and yields the system, position, recursive-error,
and feedback relations. The gain equation
$\dot k_{\rm a}=\nu\sigma$
holds by construction on every holding and gain-increase interval and
hence almost everywhere on $\R_{\geq0}$. Thus the reconstructed
quintuple is a global closed-loop trajectory. The estimate
\eqref{eq:FiniteHorizonGainBound} follows directly from
$0\leq\dot k_{\rm a}\leq\nu$.

To prove uniqueness, suppose that another supervised closed-loop
trajectory with the same initial values is given. Both trajectories
coincide in the initial holding mode by uniqueness in
Proposition~\ref{prop:FixedModeGlobalTrajectory}. Hence their first
trigger times, defined by the same essential-supremum condition,
coincide. The subsequent gain-increase trajectories are again equal
by fixed-mode uniqueness. Repeating the argument at the following
switching times and using the fact that only finitely many switches
occur on bounded time intervals shows that the two trajectories
coincide globally.

Finally,
$e=\varphi^{-1}\chi$, $e_{\rm r} = (\tau_{\rm f}\varphi)^{-1}\eta$,
and
$\dot e = e_{\rm r} - \left( \frac{\alpha} {1-\|\chi\|_\U^2} + q \right)e$.
Since $\varphi$ is bounded away from zero,
$\|\chi(t)\|_\U\leq\rho_\chi<1$, and
$\|\eta(t)\|_\U<1$ almost everywhere, it follows that
$e,e_{\rm r},\dot e \in \Lp{\infty}(\R_{\geq0};\U)$.
Hence
$v = \dot e+\dot y_{\rm ref} \in \Lp{\infty}(\R_{\geq0};\U)$,
and therefore
$y = y_{\rm ref}+e \in \Wkp{1,\infty}(\R_{\geq0};\U)$.

On every finite interval,
Proposition~\ref{prop:FiniteHorizonRecursiveErrorMargin} gives
$\gamma(\eta) \in \Lp{\infty}_{\loc}(\R_{\geq0};\U)$,
while \eqref{eq:FiniteHorizonGainBound} gives
$k_{\rm a} \in \Lp{\infty}_{\loc}(\R_{\geq0})$.
The feedback law
\[
u
=
u_{\ext}
-
\frac{1}{\tau_{\rm f}\varphi}
\left(
k_{\rm a}\eta
+
\beta\gamma(\eta)
\right)
\]
therefore implies
$u \in \Lp{\infty}_{\loc}(\R_{\geq0};\U)$.
\end{proof}

\section{Global boundedness}
\label{sec:GlobalBoundedness}

The preceding theorem guarantees global existence and funnel
feasibility under impedance passivity alone. In particular,
$e,\dot e,e_{\rm r},v
\in
\Lp{\infty}(\R_{\geq0};\U)$,
$u\in\Lp{\infty}_{\loc}(\R_{\geq0};\U)$
and
$k_{\rm a}\in\Lp{\infty}_{\loc}(\R_{\geq0})$. The position error is uniformly separated from its funnel boundary,
whereas the recursive-error margin is initially known only on finite
time intervals. Moreover, neither the internal state nor the number
of gain increases is controlled by the preceding result.

We now introduce a qualitative system condition under which the
supervisor is activated only finitely many times. As a consequence,
the supervised gain, the internal state, and the control input remain
globally bounded, and the recursive-error funnel margin becomes
uniform in time.

\begin{assmptn}[Position-compatible strict dissipativity]
\label{assum:PositionCompatibleDissipativity}
There exist a bounded position operator
$\mathcal J_{\rm p}\in\Lb(\X,\U)$,
a bounded, self-adjoint, and coercive operator
$\Pi\in\Lb(\X)$,
and constants
$\omega_{\rm s}>0$ and
$c_{\rm p},c_{\rm v}\geq0$
such that, for every
\[
\spvek{x}{u}\in\dom(S),
\qquad
\spvek{f}{v}
=
\sbvek{A\&B}{C\&D}
\spvek{x}{u},
\]
one has
$\mathcal J_{\rm p}f=v$
and
\begin{equation}
\label{eq:PositionCompatibleDissipativity}
\scprod{\Pi x}{f}_\X
\leq
\scprod{u}{v}_\U
-
\omega_{\rm s}\|x\|_\X^2
+
c_{\rm p}\|\mathcal J_{\rm p}x\|_\U^2
+
c_{\rm v}\|v\|_\U^2.
\end{equation}
\end{assmptn}

Condition $\mathcal J_{\rm p}f=v$ means that
$\mathcal J_{\rm p}x$ is a bounded state-space realization of the
integrated velocity output. Indeed, along every strong
trajectory, $f=\dot x$ almost everywhere, and hence
\[
\frac{\dd}{\dd t}
\bigl(
\mathcal J_{\rm p}x-y
\bigr)
=
\mathcal J_{\rm p}f-\dot y
=
v-v
=
0
\quad\text{almost everywhere}.
\]
Consequently, for every $t\geq0$,
\begin{equation}
\label{eq:PositionCoordinateConservation}
\mathcal J_{\rm p}x(t)-y(t)
=
\mathcal J_{\rm p}x_0-y_0.
\end{equation}
The defect term involving $\mathcal J_{\rm p}x$ in
\eqref{eq:PositionCompatibleDissipativity} allows for
stationary configurations with a nonzero position component, such as
a statically deflected mechanical system. The term involving $v$ accommodates observation
effects which need not be controlled directly by the state norm.

For the global transformed trajectory from
Theorem~\ref{thm:Main_Thm_Exist_Uniq}, set
\[
\zeta(t)
:=
\mathcal J_{\rm p}\widehat z(t),
\qquad
K_\chi(t)
:=
K\bigl(t,\chi(t)\bigr)
=
q(t)Q+\alpha Q\gamma'\bigl(\chi(t)\bigr).
\]
It follows from \eqref{eq:PositionCoordinateConservation} and
\eqref{eq:TransformedState} that
\begin{equation}
\label{eq:ZetaRepresentation}
\zeta
={}
\tau_{\rm f}\varphi
\left(
y+\mathcal J_{\rm p}x_0-y_0-
\mathcal J_{\rm p}L_\mu u_{\ext}
-
\mathcal J_{\rm p}Q\dot y_{\rm ref}
\right)
+
\tau_{\rm f}\mathcal J_{\rm p}Qq\chi
+
\tau_{\rm f}\alpha
\mathcal J_{\rm p}Q\gamma(\chi).
\end{equation}
Further, collect the terms in the first equation of
\eqref{eq:TransformedClosedLoop} which do not contain
$\widehat f$, $K(t,\chi)\eta$, or
$(q-\delta)\widehat z$ by defining
\[
\begin{aligned}
\mathfrak h
:={}&
\mathfrak d
+
\tau_{\rm f}
\bigl(
\dot q-q^2
\bigr)Q\chi
-
\tau_{\rm f}qAQ\chi
-
\tau_{\rm f}\alpha AQ\gamma(\chi) -
\tau_{\rm f}\alpha
\bigl(
2q-\delta
\bigr)Q\gamma(\chi)
-
\tau_{\rm f}\alpha^2
Q\gamma'(\chi)\gamma(\chi)
-
\tau_{\rm f}\alpha\delta
Q\gamma'(\chi)\chi.
\end{aligned}
\]
The transformed closed-loop equations then take the form
\begin{equation}
\label{eq:GlobalBoundednessTransformedDynamics}
\begin{aligned}
\dot{\widehat z}
&=
\widehat f
+
K_\chi\eta
+
(q-\delta)\widehat z
+
\mathfrak h,
\\
\spvek{\widehat f}{\eta}
&=
\sbvek{A\&B}{C\&D}
\spvek{
\widehat z
}{
-k_{\rm a}\eta-\beta\gamma(\eta)
}.
\end{aligned}
\end{equation}

Let $\rho_\chi\in(0,1)$ be the uniform position-funnel bound from
Lemma~\ref{lem:UniformPositionFunnelMargin}. Since
$\|\chi(t)\|_\U\leq\rho_\chi$ for all $t\geq0$,
and $\|\eta(t)\|_\U<1$ almost everywhere,
\eqref{eq:ChiDynamics} gives
$\chi\in\Wkp{1,\infty}(\R_{\geq0};\U)$.
The representation
\eqref{eq:ZetaRepresentation}, the prescribed data and the
boundedness of $\gamma$, $\gamma'$, and $\gamma''$ on
$\overline{\mathcal B_{\rho_\chi}^{\U}(0)}$
therefore imply
\begin{equation}
\label{eq:BoundedGlobalCoefficients}
\zeta
\in
\Wkp{1,\infty}(\R_{\geq0};\U),
\qquad
\mathfrak h
\in
\Wkp{1,\infty}(\R_{\geq0};\X),
\qquad
K_\chi
\in
\Wkp{1,\infty}
\bigl(
\R_{\geq0};
\Lb(\U,\X)
\bigr).
\end{equation}

We first establish a global state bound which does not require the
gain to be sufficiently large.

\begin{prpstn}[Global bound for the transformed state]
\label{prop:GlobalTransformedStateBound}
Suppose that
Assumptions~\ref{assum:SystemClass},
\ref{assum:FunnelData},
\ref{assum:SupervisorData},
\ref{assum:InitialValues}, and
\ref{assum:PositionCompatibleDissipativity}
hold. Then the global transformed closed-loop trajectory satisfies
$\widehat z \in \Lp{\infty}(\R_{\geq0};\X)$.
\end{prpstn}

\begin{proof}
Define
\[
\mathcal V(t)
:=
\frac12
\scprod{\Pi\widehat z(t)}{\widehat z(t)}_\X.
\]
Since $\Pi$ is self-adjoint, the first relation in
\eqref{eq:GlobalBoundednessTransformedDynamics} gives, for almost every
$t\geq0$,
\[
\begin{aligned}
\dot{\mathcal V}
={}&
\scprod{\Pi\widehat z}{\widehat f}_\X
+
\scprod{\Pi\widehat z}{K_\chi\eta}_\X
+
2(q-\delta)\mathcal V
+
\scprod{\Pi\widehat z}{\mathfrak h}_\X.
\end{aligned}
\]
By the system-node relation in
\eqref{eq:GlobalBoundednessTransformedDynamics} and
Assumption~\ref{assum:PositionCompatibleDissipativity}, we have
\[
\mathcal J_{\rm p}\widehat f=\eta,
\qquad
\mathcal J_{\rm p}\widehat z=\zeta.
\]
Hence, applying
\eqref{eq:PositionCompatibleDissipativity} yields
\[
\begin{aligned}
\scprod{\Pi\widehat z}{\widehat f}_\X\leq{}&
-k_{\rm a}\|\eta\|_\U^2
-\beta\scprod{\gamma(\eta)}{\eta}_\U
-\omega_{\rm s}\|\widehat z\|_\X^2 +
c_{\rm p}\|\zeta\|_\U^2
+
c_{\rm v}\|\eta\|_\U^2.
\end{aligned}
\]
Since $k_{\rm a}\geq0$ and
$\scprod{\gamma(\eta)}{\eta}_\U\geq0$, the corresponding two terms
are nonpositive and may be discarded. By the global boundedness of
$K_\chi$ and $\mathfrak h$, the estimate
$\|\eta\|_\U<1$, and Young's inequality, for every
$\varepsilon>0$ there exists some $c_\varepsilon>0$ such that
\[
\begin{aligned}
&
\left|
\scprod{\Pi\widehat z}{K_\chi\eta}_\X
\right|
+
\left|
\scprod{\Pi\widehat z}{\mathfrak h}_\X
\right| \leq
2\varepsilon
\|\Pi\widehat z\|_\X^2
+
c_\varepsilon
\leq
2\varepsilon
\|\Pi\|_{\Lb(\X)}^2
\|\widehat z\|_\X^2
+
c_\varepsilon.
\end{aligned}
\]
Choose $\varepsilon>0$ sufficiently small that the resulting
quadratic term can be absorbed into
$-\omega_{\rm s}\|\widehat z\|_\X^2$.
The terms
$c_{\rm p}\|\zeta\|_\U^2$ and
$c_{\rm v}\|\eta\|_\U^2$
are globally bounded. Moreover,
\[
\mathcal V
\leq
\frac{\|\Pi\|_{\Lb(\X)}}{2}
\|\widehat z\|_\X^2.
\]
Consequently, there exist constants
$a_0,c_0>0$ such that, for almost every $t\geq0$,
\begin{equation}
\label{eq:UnscaledGlobalStorageEstimate}
\dot{\mathcal V}(t)
\leq
2\bigl(q(t)-\delta\bigr)\mathcal V(t)
-
a_0\mathcal V(t)
+
c_0.
\end{equation}
Since
$q-\delta=\dot\varphi/\varphi$, multiplication of
\eqref{eq:UnscaledGlobalStorageEstimate} by $\varphi^{-2}$ gives, for
almost every $t\geq0$,
\[
\frac{\dd}{\dd t}
\left(
\varphi(t)^{-2}\mathcal V(t)
\right)
\leq
-a_0\varphi(t)^{-2}\mathcal V(t)
+
c_0\varphi(t)^{-2}.
\]
Since $\varphi$ is bounded away from zero, there exists some
$c_1>0$ such that, for almost every $t\geq0$,
\[
\frac{\dd}{\dd t}
\left(
\varphi(t)^{-2}\mathcal V(t)
\right)
\leq
-a_0\varphi(t)^{-2}\mathcal V(t)+c_1.
\]
Gronwall's inequality yields
$\sup_{t\geq0} \varphi(t)^{-2}\mathcal V(t) <\infty$.
Since $\varphi$ is bounded and $\Pi$ is coercive, it follows that
$\widehat z \in \Lp{\infty}(\R_{\geq0};\X)$.
\end{proof}

The following result shows what happens after the gain has exceeded a
sufficiently large, but unknown, system-dependent level. This level
is used only in the proof and is not part of the supervisor.

\begin{lmm}[Regularity above a finite gain level]
\label{lem:HighGainRegularity}
Suppose that the assumptions of
Proposition~\ref{prop:GlobalTransformedStateBound} hold.
Let
\[
\rho_\chi
:=
\max\{
\|\chi(0)\|_\U,r_*
\}
\in(0,1)
\]
be the uniform position-error bound from
Lemma~\ref{lem:UniformPositionFunnelMargin}, and set
\[
K_\infty
:=
\sup_{\substack{t\geq0\\
\|\xi\|_\U\leq\rho_\chi}}
\|K(t,\xi)\|_{\Lb(\U,\X)}
<\infty.
\]
This constant bounds $K_\chi$ uniformly for the given trajectory and
for every continuation considered in the final assertion below.
Choose $k_*\geq0$ such that
\begin{equation}
\label{eq:HighGainThreshold}
k_*+\beta-c_{\rm v}
>
\frac{
\|\Pi\|_{\Lb(\X)}^2K_\infty^2
}{
4\omega_{\rm s}
}.
\end{equation}
Suppose that there exists some $T_*\geq0$ such that
$k_{\rm a}(t)\geq k_*$ for all $t\geq T_*$.
Then
$\widehat z
\in
\Wkp{1,\infty}(\R_{\geq0};\X)$,
$\widehat f
\in
\Lp{\infty}(\R_{\geq0};\X)$.
Moreover, there exists some $M_*>0$ such that, for almost every
$t\geq T_*$,
\begin{equation}
\label{eq:HighGainEtaEstimate}
\|\eta(t)\|_\U
\leq
\frac{M_*}
{1+m_\mu k_{\rm a}(t)},
\end{equation}
and
\begin{equation}
\label{eq:HighGainBarrierEstimate}
\|\gamma(\eta(t))\|_\U
\leq
\frac{M_*}{\beta m_\mu}.
\end{equation}
The constant $M_*$ can be chosen uniformly over all
global transformed trajectories which coincide with the given
trajectory on $[0,T_*]$, are assembled from the fixed-mode
trajectories of
Proposition~\ref{prop:FixedModeGlobalTrajectory}, and satisfy
\[
0\leq\dot k_{\rm a}(t)\leq\nu,
\qquad
k_{\rm a}(t)\geq k_*
\quad\text{for almost every }t\geq T_*.
\]
\end{lmm}
\begin{proof}
For $a\in(0,1]$, define
$\Delta_a\widehat z(t) := \widehat z(t+a)-\widehat z(t)$,
and analogously define the increments of
$\widehat f$, $\eta$, $\gamma(\eta)$, $k_{\rm a}$, and $\zeta$.

For each fixed $a\in(0,1]$, the system-node relation holds at both
$t$ and $t+a$ for almost every $t\geq0$. Subtracting these relations
and using linearity of the system node gives
\[
\spvek{
\Delta_a\widehat f
}{
\Delta_a\eta
}
=
\sbvek{A\&B}{C\&D}
\spvek{
\Delta_a\widehat z
}{
-k_{\rm a}(t+a)\Delta_a\eta
-
\bigl(
\Delta_a k_{\rm a}(t)
\bigr)\eta(t)
-
\beta
\bigl(
\gamma(\eta(t+a))-\gamma(\eta(t))
\bigr)
}.
\]
Since
$\mathcal J_{\rm p}\Delta_a\widehat z=\Delta_a\zeta$,
applying
\eqref{eq:PositionCompatibleDissipativity} to this difference and
using
\eqref{eq:StrongMonotonicityGamma} yields
\begin{equation}
\scprod{
\Pi\Delta_a\widehat z
}{
\Delta_a\widehat f
}_\X
\leq{}
-\omega_{\rm s}
\|\Delta_a\widehat z\|_\X^2
-
\bigl(
k_{\rm a}(t+a)+\beta-c_{\rm v}
\bigr)
\|\Delta_a\eta\|_\U^2
-
\bigl(
\Delta_a k_{\rm a}(t)
\bigr)
\scprod{
\eta(t)
}{
\Delta_a\eta(t)
}_\U
+
c_{\rm p}
\|\Delta_a\zeta\|_\U^2.
\label{eq:IncrementStrictDissipation}
\end{equation}
Subtracting the first equation in
\eqref{eq:GlobalBoundednessTransformedDynamics} at times $t+a$ and
$t$ gives
\begin{equation}
\label{eq:IncrementTransformedDynamics}
\frac{\dd}{\dd t}
\Delta_a\widehat z
=
\Delta_a\widehat f
+
K_\chi(t+a)\Delta_a\eta
+
\bigl(q(t+a)-\delta\bigr)\Delta_a\widehat z
+
\mathfrak r_a,
\end{equation}
where
\[
\mathfrak r_a(t)
:={}
\bigl(
K_\chi(t+a)-K_\chi(t)
\bigr)\eta(t)
+
\bigl(
q(t+a)-q(t)
\bigr)\widehat z(t)
+
\Delta_a\mathfrak h(t).
\]
By
\eqref{eq:BoundedGlobalCoefficients},
the representation
\eqref{eq:ZetaRepresentation},
Proposition~\ref{prop:GlobalTransformedStateBound}, and
$\|\eta\|_\U<1$, there exists some $c_{\rm rem}>0$,
independent of $a$, such that, for almost every $t\geq0$,
\begin{equation}
\label{eq:IncrementRemainderBound}
\|\mathfrak r_a(t)\|_\X
+
\|\Delta_a\zeta(t)\|_\U
\leq
c_{\rm rem}a.
\end{equation}
Furthermore,
Theorem~\ref{thm:Main_Thm_Exist_Uniq} gives
$|\Delta_a k_{\rm a}(t)| \leq \nu a$.

Condition \eqref{eq:HighGainThreshold} implies that there exists some
$c_*>0$ such that, for all $z\in\X$, $w\in\U$, and $k\geq k_*$,
\begin{equation}
\label{eq:HighGainQuadraticCoercivity}
\omega_{\rm s}\|z\|_\X^2
+
\bigl(
k+\beta-c_{\rm v}
\bigr)\|w\|_\U^2
-
\|\Pi\|_{\Lb(\X)}
K_\infty
\|z\|_\X\|w\|_\U
\geq
c_*
\left(
\|z\|_\X^2+\|w\|_\U^2
\right).
\end{equation}

Define
\[
\mathcal V_a(t)
:=
\frac12
\scprod{
\Pi\Delta_a\widehat z(t)
}{
\Delta_a\widehat z(t)
}_\X.
\]
Combining
\eqref{eq:IncrementStrictDissipation},
\eqref{eq:IncrementTransformedDynamics},
\eqref{eq:IncrementRemainderBound}, and
\eqref{eq:HighGainQuadraticCoercivity}, and applying Young's
inequality to the terms containing
$\Delta_a k_{\rm a}$ and $\mathfrak r_a$, gives new constants
$\widetilde c_1,\widetilde c_2>0$, independent of $a$, such that, for almost every
$t\geq T_*$,
\[
\dot{\mathcal V}_a(t)
\leq
2\bigl(q(t+a)-\delta\bigr)\mathcal V_a(t)
-
\widetilde c_1\|\Delta_a\widehat z(t)\|_\X^2
+
\widetilde c_2 a^2.
\]
Since
$q(t+a)-\delta = \frac{\dot\varphi(t+a)}{\varphi(t+a)}$,
the boundedness of $\Pi$ and the positive lower bound of
$\varphi$ give constants $c_3,c_4>0$ such that, for
almost every $t\geq T_*$,
\begin{equation}
\label{eq:ScaledIncrementStorageEstimate}
\frac{\dd}{\dd t}
\left(
\varphi(t+a)^{-2}\mathcal V_a(t)
\right)
\leq
-c_3
\varphi(t+a)^{-2}\mathcal V_a(t)
+
c_4a^2.
\end{equation}

The construction in the proof of
Theorem~\ref{thm:Main_Thm_Exist_Uniq} gives
$\widehat z\in\Wkp{1,\infty}_{\loc}(\R_{\geq0};\X)$. Hence there
exists some $c_5>0$ such that, for all sufficiently small $a>0$,
\[
\varphi(T_*+a)^{-2}\mathcal V_a(T_*)
\leq
c_5a^2.
\]
Gronwall's inequality, followed by the coercivity of
$\Pi$ and the boundedness of $\varphi$, shows that
\[
\sup_{t\geq T_*}
\|\widehat z(t+a)-\widehat z(t)\|_\X
\leq
c_6a.
\]
Let $T_*\leq s<t$. Choose $n\in\N$ sufficiently large such that
$a:=\frac{t-s}{n}$
lies in the range of the preceding estimate. Applying this estimate
successively at
$s,\ s+a,\ \ldots,\ s+(n-1)a$
gives
$\|\widehat z(t)-\widehat z(s)\|_\X \leq c_6(t-s)$.
Hence $\widehat z$ is Lipschitz continuous on
$[T_*,\infty)$ and therefore
$\widehat z \in \Wkp{1,\infty}([T_*,\infty);\X)$.
Together with the local regularity on $[0,T_*]$, this gives
$\widehat z \in \Wkp{1,\infty}(\R_{\geq0};\X)$.

The first equation in
\eqref{eq:GlobalBoundednessTransformedDynamics} now gives
\[
\widehat f
=
\dot{\widehat z}
-
K_\chi\eta
-
(q-\delta)\widehat z
-
\mathfrak h
\in
\Lp{\infty}(\R_{\geq0};\X).
\]

We record the uniformity of these estimates. Consider
any global transformed trajectory which coincides with the present
trajectory on $[0,T_*]$, is assembled from fixed-mode trajectories,
and satisfies
\[
0\leq\dot k_{\rm a}\leq\nu,
\qquad
k_{\rm a}\geq k_*
\quad\text{almost everywhere on }[T_*,\infty).
\]
The global bound for $\widehat z$ in
Proposition~\ref{prop:GlobalTransformedStateBound} is independent of
the subsequent mode selection. The bounds in
\eqref{eq:BoundedGlobalCoefficients}, the estimate
$|\Delta_a k_{\rm a}|\leq\nu a$, and the coercivity constant in
\eqref{eq:HighGainQuadraticCoercivity} are likewise uniform.
Since all such trajectories have the same value at $T_*$, the local
fixed-mode estimates yield a common constant in the initial
increment bound at $T_*$. Consequently, the preceding
difference-quotient argument gives uniform bounds for
$\widehat z$ in
$\Wkp{1,\infty}(\R_{\geq0};\X)$
and for $\widehat f$ in
$\Lp{\infty}(\R_{\geq0};\X)$.

The resolvent representation of the system-node relation in
\eqref{eq:GlobalBoundednessTransformedDynamics} is
\begin{equation}
\label{eq:GlobalTransferIdentityForEta}
P(\mu)
\left(
k_{\rm a}\eta+\beta\gamma(\eta)
\right)
+
\eta
=
C(\mu I-A)^{-1}
\bigl(
\mu\widehat z-\widehat f
\bigr).
\end{equation}
It follows that there exists some $M_*>0$, uniform over
the class of trajectories specified in the statement, such that the
norm of the right-hand side is bounded by $M_*$ almost everywhere. Since
$\eta = \bigl( 1-\|\eta\|_\U^2 \bigr)\gamma(\eta)$,
taking the inner product of the left-hand side of
\eqref{eq:GlobalTransferIdentityForEta} with $\gamma(\eta)$ and
estimating its right-hand side by $M_*\|\gamma(\eta)\|_\U$ gives
\[
\beta m_\mu\|\gamma(\eta)\|_\U^2
\leq
M_*\|\gamma(\eta)\|_\U.
\]
This proves
\eqref{eq:HighGainBarrierEstimate}.

Similarly, taking the inner product of the left-hand side of
\eqref{eq:GlobalTransferIdentityForEta} with $\eta$ and estimating
its right-hand side by $M_*\|\eta\|_\U$ gives
\[
\bigl(
1+m_\mu k_{\rm a}
\bigr)
\|\eta\|_\U^2
\leq
M_*\|\eta\|_\U,
\]
because $\gamma(\eta)$ is a nonnegative scalar multiple of $\eta$.
This proves
\eqref{eq:HighGainEtaEstimate}.
\end{proof}

We can now prove that the supervisor cannot increase the gain
infinitely often.

\begin{prpstn}[Finite supervisor activity]
\label{prop:FiniteSupervisorActivity}
Suppose that
Assumptions~\ref{assum:SystemClass},
\ref{assum:FunnelData},
\ref{assum:SupervisorData},
\ref{assum:InitialValues}, and
\ref{assum:PositionCompatibleDissipativity}
hold. Then the supervisor enters the gain-increase mode only finitely
many times. In particular,
$k_{\rm a} \in \Wkp{1,\infty}(\R_{\geq0})$.
\end{prpstn}

\begin{proof}
Assume, to the contrary, that the supervisor enters the
gain-increase mode infinitely often. Since every activation increases
the gain by $\Delta k>0$ and the gain never decreases, it follows
that
$k_{\rm a}(t)\longrightarrow\infty$ as $t\to\infty$.

Choose $k_*$ as in
\eqref{eq:HighGainThreshold}. Then there exists some $T_*\geq0$ such
that
\[
k_{\rm a}(t)\geq k_*
\qquad
\text{for all }t\geq T_*.
\]
Apply Lemma~\ref{lem:HighGainRegularity} and fix the
uniform constant $M_*>0$ furnished by its final assertion. Choose
$k_{\rm m}\geq k_*$ so large that
\[
\frac{M_*}
{1+m_\mu k_{\rm m}}
<
r_{\rm a}.
\]
Since gain-increase phases occur infinitely often and cannot
accumulate in finite time, we may choose an index $n\geq1$ such that
the beginning $s_n$ of the $n$th holding phase satisfies
\[
s_n\geq T_*,
\qquad
k_{\rm a}(s_n)\geq k_{\rm m}.
\]

Let $\widetilde\eta_n$ and $\widetilde k_n$ denote the normalized
recursive error and the gain of the global holding-mode trajectory
starting from the transformed state at $s_n$ that is used in the
definition of $\tau_n$. Concatenate the supervised trajectory on
$[0,s_n]$ with this holding-mode continuation. The resulting
trajectory coincides with the supervised trajectory on $[0,T_*]$,
satisfies
$0\leq\dot{\widetilde k}_n\leq\nu$, and has
\[
\widetilde k_n(t)
=
k_{\rm a}(s_n)
\geq
k_{\rm m}
\qquad
\text{for all }t\geq s_n.
\]
The uniform estimate from
Lemma~\ref{lem:HighGainRegularity} therefore gives, for almost every
$t\geq s_n$,
\[
\|\widetilde\eta_n(t)\|_\U
\leq
\frac{M_*}
{1+m_\mu\widetilde k_n(t)}
\leq
\frac{M_*}
{1+m_\mu k_{\rm m}}
<
r_{\rm a}.
\]
Consequently, for every $t>s_n$,
\[
\operatorname*{ess\,sup}_{s\in(s_n,t)}
\|\widetilde\eta_n(s)\|_\U
<
r_{\rm a}.
\]
Thus the set in
\eqref{eq:SupervisorTriggerTime} is empty and
$\tau_n=\infty$. No further gain increase occurs, contradicting the
assumption of infinitely many activations. Hence the supervisor enters the
gain-increase mode only finitely many times.

Since
$0\leq\dot k_{\rm a}\leq\nu$
and the gain is increased only finitely many times,
$k_{\rm a}$ and $\dot k_{\rm a}$ are globally essentially bounded.
Therefore,
$k_{\rm a} \in \Wkp{1,\infty}(\R_{\geq0})$.
\end{proof}

We now obtain the global boundedness result.

\begin{thrm}[Model-free global boundedness]
\label{thm:GlobalBoundedness}
Suppose that
Assumptions~\ref{assum:SystemClass},
\ref{assum:FunnelData},
\ref{assum:SupervisorData},
\ref{assum:InitialValues}, and
\ref{assum:PositionCompatibleDissipativity}
hold.

Then the global supervised closed-loop trajectory from
Theorem~\ref{thm:Main_Thm_Exist_Uniq} satisfies
\[
x
\in
\Lp{\infty}(\R_{\geq0};\X),
\qquad
k_{\rm a}
\in
\Wkp{1,\infty}(\R_{\geq0}),
\]
and
$u,v,e,\dot e,e_{\rm r} \in \Lp{\infty}(\R_{\geq0};\U)$.
Moreover, there exists some $\varepsilon_\eta>0$ such that, for
almost every $t\geq0$,
\begin{equation}
\label{eq:UniformRecursiveFunnelDistance}
\tau_{\rm f}\varphi(t)
\|e_{\rm r}(t)\|_\U
\leq
1-\varepsilon_\eta.
\end{equation}
Together with
\eqref{eq:UniformPositionFunnelDistance}, both normalized errors are
uniformly separated from their respective funnel boundaries.
\end{thrm}

\begin{proof}
By
Proposition~\ref{prop:FiniteSupervisorActivity}, there are only
finitely many gain-increase phases. Let $s_N$ denote the beginning of
the final holding phase. Then
$\tau_N=\infty$.
It follows from
\eqref{eq:SupervisorTriggerTime} that, for every $t>s_N$,
\[
\operatorname*{ess\,sup}_{s\in(s_N,t)}
\|\eta(s)\|_\U
<
r_{\rm a}.
\]
Hence
$\|\eta(t)\|_\U\leq r_{\rm a}$ for almost every $t\geq s_N$.
If $s_N>0$, Proposition~\ref{prop:FiniteHorizonRecursiveErrorMargin}
gives some $\rho_N\in(0,1)$ such that
$\|\eta(t)\|_\U\leq\rho_N$ for almost every $t\in[0,s_N]$.
If $s_N=0$, set $\rho_N:=0$.
Therefore,
$\|\eta(t)\|_\U \leq \rho_\eta := \max\{\rho_N,r_{\rm a}\} <1$
for almost every $t\geq0$. This proves
\eqref{eq:UniformRecursiveFunnelDistance} with
$\varepsilon_\eta := 1-\rho_\eta$.
It also gives
$\gamma(\eta) \in \Lp{\infty}(\R_{\geq0};\U)$.

Since
$k_{\rm a}\in\Lp{\infty}(\R_{\geq0})$,
the feedback law
\[
u
=
u_{\ext}
-
\frac{1}{\tau_{\rm f}\varphi}
\left(
k_{\rm a}\eta+\beta\gamma(\eta)
\right)
\]
implies
$u \in \Lp{\infty}(\R_{\geq0};\U)$.

Proposition~\ref{prop:GlobalTransformedStateBound} gives
$\widehat z \in \Lp{\infty}(\R_{\geq0};\X)$.
The reconstruction formula
\[
x
=
\frac{1}{\tau_{\rm f}\varphi}
\left(
\widehat z
-
\tau_{\rm f}Qq\chi
-
\tau_{\rm f}\alpha Q\gamma(\chi)
\right)
+
L_\mu u_{\ext}
+
Q\dot y_{\rm ref}
\]
then yields
$x \in \Lp{\infty}(\R_{\geq0};\X)$,
because $\chi$ remains in the fixed ball
$\overline{\mathcal B_{\rho_\chi}^{\U}(0)}$.
The remaining assertions follow from
Theorem~\ref{thm:Main_Thm_Exist_Uniq}.
\end{proof}

\begin{rmrk}[Model-free character]
\label{rem:GlobalBoundednessInterpretation}
Assumption~\ref{assum:PositionCompatibleDissipativity} is a
qualitative property of the system node. The operator $\Pi$, the constants $\omega_{\rm s}$ and $c_{\rm v}$,
and the chosen right inverse $Q$ enter the sufficient gain level
$k_*$ used in
Lemma~\ref{lem:HighGainRegularity}, but neither this level nor any of
these system quantities enters the controller.

The supervisor uses only the normalized recursive error
$\eta = \tau_{\rm f}\varphi e_{\rm r}$
and the freely chosen parameters
$r_{\rm a}$, $\Delta k$, and $\nu$. It increases the gain until the
closed-loop response no longer reaches the monitoring level. The
proof shows that this occurs after finitely many gain increases,
although the required final gain value is not known in advance.
\end{rmrk}

\section{Example: A boundary-actuated Euler--Bernoulli beam}
\label{sec:ex}

We illustrate
Theorems~\ref{thm:Main_Thm_Exist_Uniq}
and~\ref{thm:GlobalBoundedness}
by means of an Euler--Bernoulli beam actuated by a boundary force at
its right endpoint. The co-located passive output is the endpoint
velocity, whereas the controlled position output is the corresponding
endpoint displacement. We first verify the abstract assumptions and
then present a numerical closed-loop simulation of the supervised
recursive funnel controller.

Let $\ell>0$ and suppose that
$\rho,EI,\rho^{-1},(EI)^{-1}
\in
\Lp{\infty}([0,\ell])$
are positive almost everywhere. We further allow for a bounded
distributed viscous damping coefficient
$d\in\Lp{\infty}([0,\ell])$,
$d\geq0$ almost everywhere. For $t\geq0$ and $\xi\in(0,\ell)$, the
transverse displacement ${\bm w}$ satisfies
\[
\rho(\xi)
\frac{\partial^2{\bm w}}{\partial t^2}(\xi,t)
=
-
\frac{\partial^2}{\partial\xi^2}
\left(
EI(\xi)
\frac{\partial^2{\bm w}}{\partial\xi^2}(\xi,t)
\right)
-
d(\xi)
\frac{\partial{\bm w}}{\partial t}(\xi,t).
\]
The beam is clamped at its left endpoint,
${\bm w}(0,t) = \tfrac{\partial{\bm w}}{\partial\xi}(0,t) = 0$.
At the right endpoint, the bending moment vanishes and the shear force
is used as the input:
\[
\left(
EI\tfrac{\partial^2{\bm w}}{\partial\xi^2}
\right)(\ell,t)
=
0,
\qquad
-
\tfrac{\partial}{\partial\xi}
\left(
EI\tfrac{\partial^2{\bm w}}{\partial\xi^2}
\right)(\ell,t)
=
u(t).
\]
The corresponding passive output is the endpoint velocity
$v(t)
=
\frac{\partial{\bm w}}{\partial t}(\ell,t)$.
We use the curvature
$\bm\kappa
=
{\bm w}_{\xi\xi}$
and the momentum density
$\bm p
=
\rho{\bm w}_t$
as state variables. The state space is
$\X
=
\Lp{2}([0,\ell];\R^2)$,
equipped with the energy inner product induced by
\[
\left\|
\spvek{\bm\kappa}{\bm p}
\right\|_\X^2
=
\int_0^\ell
\left(
EI(\xi)|\bm\kappa(\xi)|^2
+
\rho(\xi)^{-1}|\bm p(\xi)|^2
\right)
\,\dd\xi.
\]
The corresponding velocity-output system node is
$S = \sbvek{A\&B}{C\&D} : \dom(S) \subset \X\times\R \to \X\times\R$,
where
\[
A\&B
\spvek{
\spvek{\bm\kappa}{\bm p}
}{
u
}
=
\spvek{
\displaystyle
\tfrac{\partial^2}{\partial\xi^2}
\bigl(
\rho^{-1}\bm p
\bigr)
}{
\displaystyle
-
\tfrac{\partial^2}{\partial\xi^2}
\bigl(
EI\bm\kappa
\bigr)
-
d\rho^{-1}\bm p
},
\qquad C\&D
\spvek{
\spvek{\bm\kappa}{\bm p}
}{
u
}
=
\bigl(
\rho^{-1}\bm p
\bigr)(\ell).
\]
Its domain is
\begin{equation}
\label{eq:BeamSystemNodeDomain}
\dom(S)
=
\setdef*{
\spvek{
\spvek{\bm\kappa}{\bm p}
}{
u
}
\in
\X\times\R
}{
\begin{array}{l}
EI\bm\kappa,\rho^{-1}\bm p
\in
\Hk{2}([0,\ell]),
\\
(\rho^{-1}\bm p)(0)=0,
\quad
\frac{\partial}{\partial\xi}
(\rho^{-1}\bm p)(0)=0,
\\
(EI\bm\kappa)(\ell)=0,
\quad
-\frac{\partial}{\partial\xi}
(EI\bm\kappa)(\ell)=u
\end{array}
}.
\end{equation}
Here and in the following, derivatives and traces are understood in
the Sobolev sense.

For $d=0$, this is precisely the endpoint case $\xi_0=\ell$ of the
point-actuated beam constructed in
\cite[Sec.~5, case~(b)]{GovHasPauRei2025}. Hence the corresponding
operator is a system node. The additional damping term corresponds to the bounded perturbation
$\spvek{\bm\kappa}{\bm p} \longmapsto \spvek{0}{-d\rho^{-1}\bm p}$
of the state equation. By the bounded perturbation theorem
\cite[Thm.~III.1.3]{EngelNagel2000}, and since the remaining
system-node properties are preserved under this bounded state
perturbation, the damped operator is again a system node with the
same domain.

To verify impedance passivity, set
$\bm v_{\rm b}
:=
\rho^{-1}\bm p$,
$\bm m
:=
EI\bm\kappa$.
For every
$\spvek{ \spvek{\bm\kappa}{\bm p} }{ u } \in\dom(S)$,
integration by parts gives
\begin{multline*}
\scprod{
A\&B
\spvek{
\spvek{\bm\kappa}{\bm p}
}{
u
}
}{
\spvek{\bm\kappa}{\bm p}
}_\X
=
\int_0^\ell
\left(
\bm m\,\bm v_{\rm b}''
-
\bm v_{\rm b}\,\bm m''
-
d|\bm v_{\rm b}|^2
\right)
\,\dd\xi
\\
=
\left[
\bm m\bm v_{\rm b}'
-
\bm m'\bm v_{\rm b}
\right]_0^\ell
-
\int_0^\ell
d(\xi)|\bm v_{\rm b}(\xi)|^2
\,\dd\xi
=
u
\bigl(
\rho^{-1}\bm p
\bigr)(\ell)
-
\int_0^\ell
d(\xi)
\left|
\rho(\xi)^{-1}\bm p(\xi)
\right|^2
\,\dd\xi
\leq
uv.
\end{multline*}
Thus the damped velocity-output system node is impedance passive.

The endpoint displacement is represented by the bounded functional
\begin{equation}
\label{eq:BeamPositionOperator}
\mathcal J_{\rm p}\in\Lb(\X,\R),
\qquad
\mathcal J_{\rm p}
\spvek{\bm\kappa}{\bm p}
:=
\int_0^\ell
(\ell-\xi)\bm\kappa(\xi)
\,\dd\xi.
\end{equation}
Indeed, the clamping conditions imply formally that
$\mathcal J_{\rm p}
\spvek{\bm\kappa}{\bm p}
=
{\bm w}(\ell)$.

Let
\[
\spvek{
\spvek{\bm\kappa}{\bm p}
}{
u
}
\in\dom(S),
\qquad
\spvek{f}{v}
=
\sbvek{A\&B}{C\&D}
\spvek{
\spvek{\bm\kappa}{\bm p}
}{
u
}.
\]
Since the first component of $f$ is
$(\rho^{-1}\bm p)''$, integration by parts and the left boundary
conditions give
\[
\mathcal J_{\rm p}f
=
\int_0^\ell
(\ell-\xi)
\frac{\partial^2}{\partial\xi^2}
\bigl(
\rho^{-1}\bm p
\bigr)(\xi)
\,\dd\xi
=
\bigl(
\rho^{-1}\bm p
\bigr)(\ell)
=
v.
\]
Thus $\mathcal J_{\rm p}f=v$ holds. In particular, along
every strong trajectory,
$\frac{\dd}{\dd t}
\mathcal J_{\rm p}x(t)
=
v(t)$.
Consequently, if the initial position satisfies the physical
compatibility condition
$y_0
=
\mathcal J_{\rm p}x_0$,
then, for every $t\geq0$,
\[
y(t)
=
\mathcal J_{\rm p}x(t)
=
{\bm w}(\ell,t).
\]
It remains to verify the surjectivity of the observation operator.
The main-operator domain is obtained from
\eqref{eq:BeamSystemNodeDomain} by setting $u=0$, and
$C
\spvek{\bm\kappa}{\bm p}
=
\bigl(
\rho^{-1}\bm p
\bigr)(\ell)$.
For $r\in\R$, set
$\bm\kappa_r:=0$,
$\bm p_r(\xi)
:=
\rho(\xi)r\frac{\xi^2}{\ell^2}$,
$\xi\in[0,\ell]$.
Then
$\rho^{-1}\bm p_r = r\frac{\xi^2}{\ell^2} \in \Hk{2}([0,\ell])$,
and
$(\rho^{-1}\bm p_r)(0)
=
\tfrac{\partial}{\partial\xi}
(\rho^{-1}\bm p_r)(0)
=
0$,
$(\rho^{-1}\bm p_r)(\ell)
=
r$.
Since $EI\bm\kappa_r=0$, it follows that
$\spvek{\bm\kappa_r}{\bm p_r}
\in\dom(A)$ and
$C
\spvek{\bm\kappa_r}{\bm p_r}
=
r$.
Thus $C:\dom(A)\to\R$ is surjective, and
Assumption~\ref{assum:SystemClass} is satisfied.

Let the prescribed data and supervisor parameters satisfy
Assumptions~\ref{assum:FunnelData}
and~\ref{assum:SupervisorData}, and set
$e(t):={\bm w}(\ell,t)-y_{\rm ref}(t)$.
The recursive error is
\[
e_{\rm r}(t)
:={}
\frac{\partial{\bm w}}{\partial t}(\ell,t)
-
\dot y_{\rm ref}(t)
+
\left(
\frac{\alpha}
{1-\varphi(t)^2|e(t)|^2}
+
\frac{\dot\varphi(t)}{\varphi(t)}
+
\delta
\right)e(t).
\]
The feedback law becomes
\[
u(t)
=
u_{\ext}(t)
-
k_{\rm a}(t)e_{\rm r}(t)
-
\frac{
\beta e_{\rm r}(t)
}{
1-
\tau_{\rm f}^2\varphi(t)^2
|e_{\rm r}(t)|^2
}.
\]
Suppose that the initial values satisfy
Assumption~\ref{assum:InitialValues} and that the physical compatibility
condition $y_0
=
\mathcal J_{\rm p}x_0$ holds. Then
Theorem~\ref{thm:Main_Thm_Exist_Uniq} yields a unique global
closed-loop trajectory. In particular, there exists some
$\varepsilon_\chi>0$ such that, for every $t\geq0$,
\[
\varphi(t)
\left|
{\bm w}(\ell,t)-y_{\rm ref}(t)
\right|
\leq
1-\varepsilon_\chi.
\]
Furthermore, for every $T>0$, there exists some
$\varepsilon_{\eta,T}>0$ such that, for almost every $t\in[0,T]$,
\[
\tau_{\rm f}\varphi(t)
|e_{\rm r}(t)|
\leq
1-\varepsilon_{\eta,T}.
\]
Moreover,
$e,\dot e,e_{\rm r},v\in\Lp{\infty}(\R_{\geq0})$,
$u\in\Lp{\infty}_{\loc}(\R_{\geq0})$.
\subsection{Global boundedness under distributed damping}

We now verify
Assumption~\ref{assum:PositionCompatibleDissipativity} under a
uniform positivity condition on the distributed damping.

\begin{prpstn}[Strict dissipativity under uniformly positive distributed damping]\label{prop:BeamStrictDissipativity}
Suppose that
there exists some $d_*>0$ such that $d(\xi)\geq d_*$ for almost every
$\xi\in(0,\ell)$. Then the
boundary-actuated Euler--Bernoulli beam satisfies
Assumption~\ref{assum:PositionCompatibleDissipativity} with the
position operator $\mathcal J_{\rm p}$ defined in
\eqref{eq:BeamPositionOperator}.
\end{prpstn}

\begin{proof}
Set
$\psi(\xi)
:=
\frac{\xi^2}{\ell^2}$,
$\xi\in[0,\ell]$,
and define the bounded operator
$\mathcal W:
\Lp{2}(0,\ell)
\to
\Hk{2}(0,\ell)$
by
$(\mathcal W\bm\kappa)(\xi)
:=
\int_0^\xi
(\xi-s)\bm\kappa(s)
\,\dd s
-
\psi(\xi)
\int_0^\ell
(\ell-s)\bm\kappa(s)
\,\dd s$.
Writing
$z:=\mathcal W\bm\kappa$,
one has
$z(0)=z'(0)=z(\ell)=0$
and
$z'' = \bm\kappa - \psi'' \mathcal J_{\rm p} \spvek{\bm\kappa}{\bm p}$.
By the Riesz representation theorem, there exists a bounded
self-adjoint operator $\mathcal R\in\Lb(\X)$ such that, for all
$\spvek{\bm\kappa_j}{\bm p_j}\in\X$, $j=1,2$,
\[
\scprod{
\mathcal R
\spvek{\bm\kappa_1}{\bm p_1}
}{
\spvek{\bm\kappa_2}{\bm p_2}
}_\X
={}
\int_0^\ell
\bm p_1\,
\mathcal W\bm\kappa_2
\,\dd\xi
+
\int_0^\ell
\bm p_2\,
\mathcal W\bm\kappa_1
\,\dd\xi.
\]
For every $\varepsilon>0$ satisfying
$\varepsilon \|\mathcal R\|_{\Lb(\X)} <1$,
the operator
$\Pi_\varepsilon := I_\X+\varepsilon\mathcal R$
is bounded and self-adjoint and, for every $x\in\X$,
\[
\scprod{\Pi_\varepsilon x}{x}_\X
\geq
\left(
1-\varepsilon
\|\mathcal R\|_{\Lb(\X)}
\right)
\|x\|_\X^2.
\]
Let
\[
x
=
\spvek{\bm\kappa}{\bm p},
\qquad
\spvek{f}{v}
=
\sbvek{A\&B}{C\&D}
\spvek{x}{u},
\]
and introduce
$\bm v_{\rm b}
:=
\rho^{-1}\bm p$,
$\bm m
:=
EI\bm\kappa$.
Then
\[
f
=
\spvek{
\bm v_{\rm b}''
}{
-\bm m''-d\bm v_{\rm b}
},
\qquad
v=\bm v_{\rm b}(\ell).
\]
The left boundary conditions give
$\mathcal W(\bm v_{\rm b}'')
=
\bm v_{\rm b}-\psi v$.
Using integration by parts and the boundary conditions for
$\bm m$ and $z$, we obtain
\begin{equation}
\scprod{\Pi_\varepsilon x}{f}_\X
=
uv
-
\int_0^\ell
d|\bm v_{\rm b}|^2
\,\dd\xi
+
\varepsilon
\int_0^\ell
\left(
-
EI|\bm\kappa|^2
+
\rho|\bm v_{\rm b}|^2
+
(\mathcal J_{\rm p}x)\bm m\psi''-
v\rho\psi\bm v_{\rm b}
-
d\bm v_{\rm b}z
\right)
\,\dd\xi.
\label{eq:BeamModifiedEnergyIdentity}
\end{equation}
Set
\[
E_\kappa
:=
\int_0^\ell
EI|\bm\kappa|^2
\,\dd\xi,
\qquad
E_{\rm v}
:=
\int_0^\ell
\rho|\bm v_{\rm b}|^2
\,\dd\xi.
\]
By uniform positivity of the damping,
\[
\int_0^\ell
d|\bm v_{\rm b}|^2
\,\dd\xi
\geq
\delta_{\rm d}E_{\rm v},
\qquad
\delta_{\rm d}
:=
\frac{d_*}
{\|\rho\|_{\Lp{\infty}(0,\ell)}}
>0.
\]
The boundedness of $\mathcal W$, $d$, $\rho$, and $EI$ implies that
there exist constants $c_0,c_1,c_2>0$ such that, for every
$\varepsilon\in(0,1]$,
\[
\varepsilon\left|
\mathcal J_{\rm p}x
\int_0^\ell
\bm m\psi''
\,\dd\xi
\right|
\leq
\frac{\varepsilon}{4}E_\kappa
+
c_1|\mathcal J_{\rm p}x|^2,\quad
\varepsilon
\left|
v
\int_0^\ell
\rho\psi\bm v_{\rm b}
\,\dd\xi
\right|
\leq
\frac{\delta_{\rm d}}{4}E_{\rm v}
+
c_2|v|^2,
\quad
\varepsilon
\left|
\int_0^\ell
d\bm v_{\rm b}z
\,\dd\xi
\right|
\leq
\frac{\varepsilon}{4}E_\kappa
+
\varepsilon c_0E_{\rm v}.
\]
Choose $\varepsilon\in(0,1]$ sufficiently small that
$\varepsilon\|\mathcal R\|_{\Lb(\X)}<1$ and $\varepsilon(1+c_0)
\leq
\frac{\delta_{\rm d}}{4}$.
Then \eqref{eq:BeamModifiedEnergyIdentity} gives
\[
\scprod{\Pi_\varepsilon x}{f}_\X
\leq
uv
-
\frac{\varepsilon}{2}E_\kappa
-
\frac{\delta_{\rm d}}{2}E_{\rm v}
+
c_1|\mathcal J_{\rm p}x|^2
+
c_2|v|^2.
\]
Since
$\|x\|_\X^2
=
E_\kappa+E_{\rm v}$,
Assumption~\ref{assum:PositionCompatibleDissipativity} holds with
$\Pi=\Pi_\varepsilon$,
$\omega_{\rm s}
=
\tfrac12
\min\{
\varepsilon,\delta_{\rm d}
\}$, $c_{\rm p}=c_1$,
$c_{\rm v}=c_2$.
Together with
$\mathcal J_{\rm p}f=v$, this proves the assertion.
\end{proof}
Combining Proposition~\ref{prop:BeamStrictDissipativity} with
Theorem~\ref{thm:GlobalBoundedness} gives the following conclusion.
If $d$ is bounded from below by $d_*>0$, then the internal beam state
and the applied boundary force are globally bounded for arbitrary
parameters
$\alpha,\beta,\delta,\tau_{\rm f},\Delta k,\nu>0$,
$r_{\rm a}\in(0,1)$,
$k_{{\rm a},0}\geq0$,
provided that the initial values satisfy the stated compatibility
conditions. Moreover, the recursive-error funnel margin is uniform
in time. No knowledge of the beam coefficients or of the constants
occurring in the strict dissipation estimate is required for the
implementation of the controller or the supervisor.

\subsection{Numerical simulations}

We finally illustrate the closed-loop behaviour of the supervised
recursive funnel controller. We specialize the boundary-actuated beam
to
$\ell=1$,
$EI\equiv\rho\equiv1$,
$d\equiv3$.
In particular, the distributed damping is uniformly positive, so that
Proposition~\ref{prop:BeamStrictDissipativity} and
Theorem~\ref{thm:GlobalBoundedness} apply.

The controlled output is the endpoint displacement
$y(t)={\bm w}(1,t)$,
whereas the passive output
$v(t)=\tfrac{\partial{\bm w}}{\partial t}(1,t)$
is evaluated directly from the velocity component of the
semidiscrete state. Hence the implementation does not numerically
differentiate the computed position signal.

The reference signal is a smoothed piecewise affine function on
$[0,10]$. It moves successively from $0$ to $1$, then to $-1$, to
$0.7$, to $-0.6$, to $0.9$, to $-0.4$, and finally back to $0$.
The corresponding transition times are
\[
0,\quad
1.3,\quad
2.8,\quad
4.0,\quad
5.4,\quad
6.8,\quad
8.3,\quad
9.5.
\]
Short polynomial transition layers are inserted near the breakpoints
so that the resulting reference belongs to
$\Wkp{3,\infty}(\R_{\geq0})$. Both $y_{\rm ref}$ and
$\dot y_{\rm ref}$ are evaluated analytically.

The position-funnel boundary is prescribed by
\[
\frac{1}{\varphi(t)}
=
0.30+0.60e^{-t/1.5}.
\]
Thus, the admissible position-error half-width decreases from $0.9$
at the initial time to the asymptotic value $0.3$. The derivative
$\dot\varphi$ occurring in the recursive error is also evaluated
analytically. The recursive-error funnel has the boundary
$\frac{1}{\tau_{\rm f}\varphi(t)}$,
and the supervisor monitoring level, expressed in terms of
$e_{\rm r}$, is
$\frac{r_{\rm a}}{\tau_{\rm f}\varphi(t)}$.

We use the controller parameters
$\alpha=1$,
$\beta=1$,
$\delta=0.5$,
$\tau_{\rm f}=0.25$,
and the supervisor parameters
$k_{{\rm a},0}=0$,
$r_{\rm a}=0.8$,
$\Delta k=2$,
$\nu=10$.
Consequently, every activation of the supervisor has duration
$T_{\rm inc}
=
\frac{\Delta k}{\nu}
=
0.2$
and increases the supervised gain by $2$. The initial beam state is zero.
Since
$y_{\rm ref}(0)=\dot y_{\rm ref}(0)=0$,
the initial errors satisfy
$e(0)=\dot e(0)=e_{\rm r}(0)=0$.
The initial controller compatibility condition is therefore satisfied
with
$u_{\ext}\equiv0$.
No input saturation is imposed.

For the spatial discretization, we use an equidistant mesh with
$N=20$ elements. The displacement and velocity are approximated in
the clamped cubic Hermite finite-element space. In the implemented
state-space realization, the displacement degrees of freedom are
replaced by the corresponding slope degrees of freedom, represented
in the associated quadratic derivative space. The resulting semidiscrete
state has dimension
$4N=80$.
The semidiscrete realization preserves the identities
$\dot y=v$ and the co-located energy relation up to machine
precision.

The hybrid closed-loop system is integrated with MATLAB's
variable-step implicit solver \texttt{ode15s}. We use relative
tolerance $10^{-7}$, absolute tolerance $10^{-9}$, and maximal step
size $10^{-2}$. The beginning of each gain-increase phase is located
by event detection when
$|\eta(t)|
=
\tau_{\rm f}\varphi(t)|e_{\rm r}(t)|$
reaches $r_{\rm a}$. Funnel-boundary events are monitored
independently as a numerical safeguard.

In the displayed simulation, one gain-increase phase is sufficient.
The supervised gain starts at zero, increases linearly to $2$ during
an interval of length $0.2$, and then remains constant. Both the
position error and the recursive error remain strictly inside their
respective funnels throughout the simulation.

We also tested finer spatial discretizations. No visually discernible
change of the closed-loop response or of the beam animation was
observed. The main effect of further mesh refinement was a
substantial increase in the sizes of the generated figure and video
files. We therefore use $N=20$ for the results presented below.

The numerical results are shown in
Figure~\ref{fig:position_funnel_simulation}. The upper-left panel
shows the endpoint displacement, the reference signal, and the
position funnel. The lower-left panel displays the boundary force
together with the supervised gain $k_{\rm a}$, using separate
vertical axes. The upper-right panel shows the position error and its
funnel. The lower-right panel displays the recursive error, its funnel
boundary
$\pm1/(\tau_{\rm f}\varphi)$,
and the supervisor monitoring levels
$\pm r_{\rm a}/(\tau_{\rm f}\varphi)$.

\begin{figure}[t]
  \centering
  \IfFileExists{position_funnel_eb_beam_closed_loop.pdf}{%
  \includegraphics[width=\textwidth]{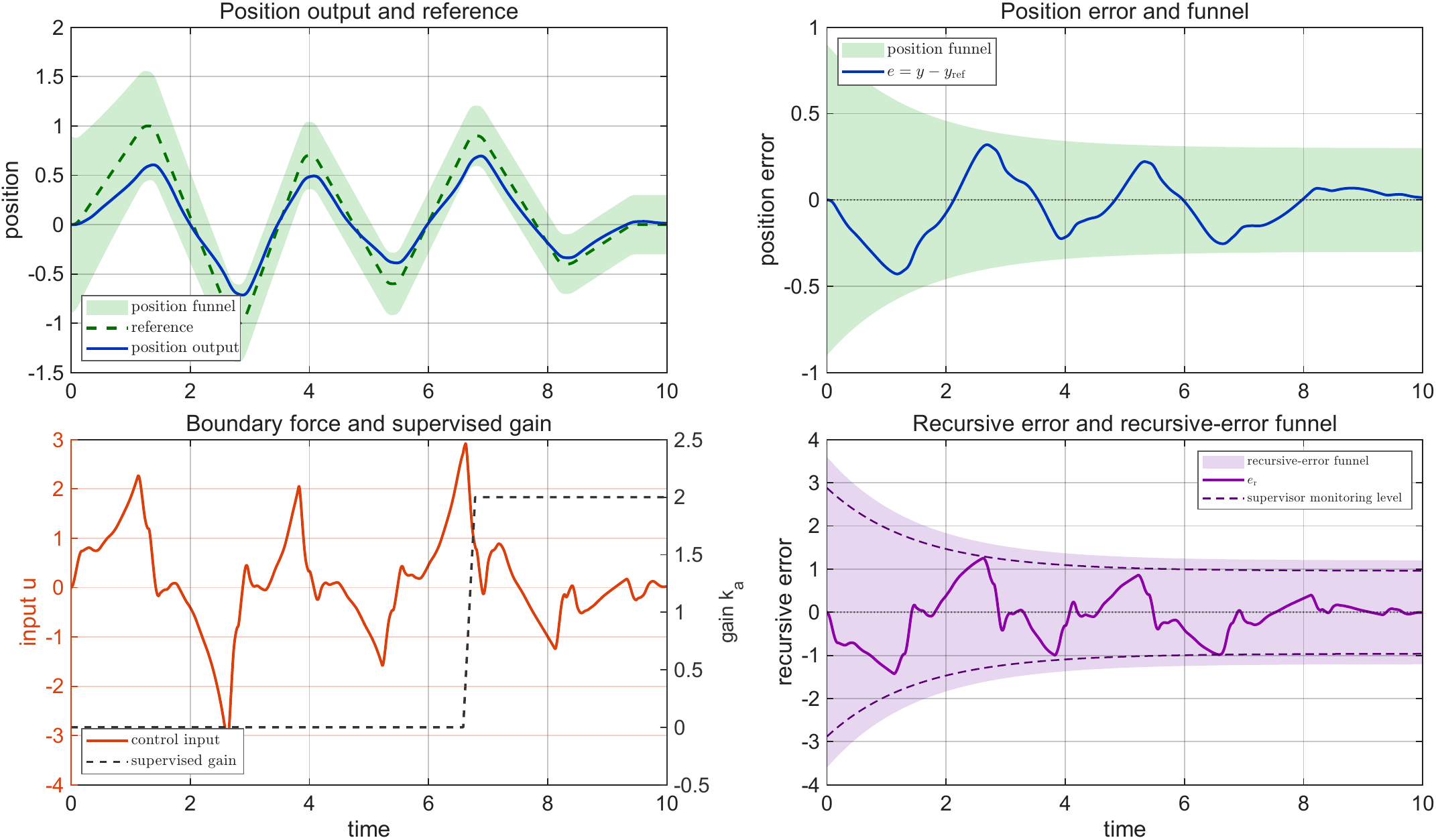}%
}{%
  \fbox{\parbox[c][70mm][c]{0.92\textwidth}{\centering
  Missing file:
  \texttt{position\_funnel\_eb\_beam\_closed\_loop.pdf}}}%
}
  \caption{Numerical closed-loop simulation of the supervised recursive
  funnel controller. Upper left: endpoint position, reference signal, and
  position funnel. Lower left: applied boundary force and supervised gain.
  Upper right: position error and position funnel. Lower right: recursive
  error, recursive-error funnel, and supervisor monitoring level.}
  \label{fig:position_funnel_simulation}
\end{figure}

Figure~\ref{fig:beam_snapshot} shows the deformation of the beam at
$t=7.28$. The displacement used in the visualization is scaled by the
factor $0.38$. The color indicates the relative magnitude of the
discrete curvature, with blue corresponding to small and red to large
bending. The green block marks the reference position, the transparent
green region represents the position funnel, and the vertical arrow
represents the applied boundary force. The snapshot is intended only
as a qualitative illustration of the closed-loop deformation; the
quantitative funnel performance is shown in
Figure~\ref{fig:position_funnel_simulation}. The animation video and
the MATLAB codes used to generate the simulation, plots, animation,
and snapshot are provided as supplementary material.

\begin{figure}[t]
  \centering
  \IfFileExists{position_funnel_eb_beam_snapshot_t728.pdf}{%
  \includegraphics[width=0.5\textwidth]{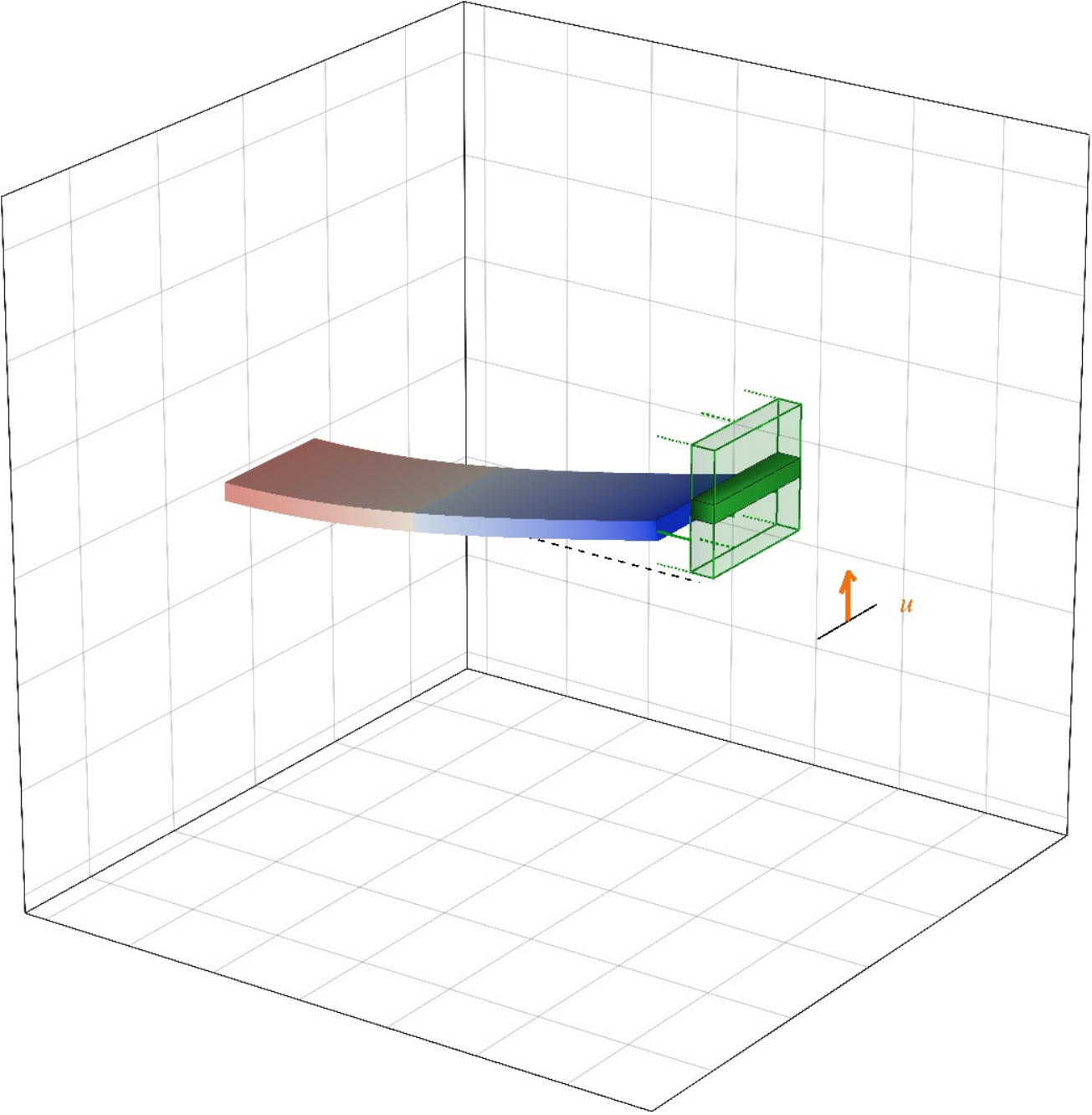}%
}{%
  \fbox{\parbox[c][45mm][c]{0.5\textwidth}{\centering
  Missing file:
  \texttt{position\_funnel\_eb\_beam\_snapshot\_t728.pdf}}}%
}
  \caption{Snapshot of the closed-loop beam animation at $t=7.28$.}
  \label{fig:beam_snapshot}
\end{figure}

\section*{Conflict of interest statement}

The authors declare no conflicts of interest.

\section*{Data availability statement}

The MATLAB codes used for the numerical simulations and the animation video
are provided as Supporting Information. No additional datasets were generated
or analysed in this study.

\section*{Use of generative artificial intelligence}

OpenAI ChatGPT was used to assist with language editing of the manuscript
and with the development and debugging of parts of the numerical
implementation. All mathematical results, proofs, numerical code, and
reported results were independently reviewed and verified by the authors,
who take full responsibility for the content of the manuscript.

\section*{Supporting information}

The MATLAB codes and the animation video described in
Section~\ref{sec:ex} are provided as Supporting Information.

\section{Conclusion}
\label{sec:conclusion}

We have developed a nonlinear prescribed-performance controller for position
tracking in infinite-dimensional impedance-passive systems. The main
difficulty is that the controlled position output is obtained by integrating
the passive velocity output and therefore does not itself form a passive pair
with the control input. A recursive funnel construction combined with a
dynamic gain supervisor overcomes this additional integration while retaining
a model-free feedback structure.

The resulting closed loop admits a unique global trajectory and achieves the
prescribed position-funnel performance for a broad class of distributed and
boundary control systems. Under an additional position-compatible strict
dissipativity condition, the supervisor is activated only finitely many times,
and the internal state, the supervised gain, and the control input remain
globally bounded. The boundary-actuated Euler--Bernoulli beam demonstrates
that these conditions can be verified for a relevant distributed-parameter
mechanical system and that the proposed controller can be implemented using
directly available position and velocity signals.

\appendix
\section{Proofs of auxiliary results}
\label{app:AuxiliaryResults}

\subsection{Proof of Lemma~\ref{lem:StrongMonotonicityGamma}}

\begin{proof}
The map $\gamma$ is continuously Fr\'echet differentiable on
$\mathcal B_1(0)$. For $\eta\in\mathcal B_1(0)$ and $h\in\U$,
\[
\gamma'(\eta)h
=
\frac{h}{1-\|\eta\|_\U^2}
+
\frac{2\scprod{h}{\eta}_\U}
{(1-\|\eta\|_\U^2)^2}\eta.
\]
Consequently,
\[
\scprod{\gamma'(\eta)h}{h}_\U
=
\frac{\|h\|_\U^2}{1-\|\eta\|_\U^2}
+
\frac{2\scprod{h}{\eta}_\U^2}
{(1-\|\eta\|_\U^2)^2}
\geq
\|h\|_\U^2.
\]
Since $\mathcal B_1(0)$ is convex, the line segment joining
$\eta_1$ and $\eta_2$ lies in $\mathcal B_1(0)$. Hence
\[
\scprod{
\gamma(\eta_1)-\gamma(\eta_2)
}{
\eta_1-\eta_2
}_\U
=
\int_0^1
\scprod{
\gamma'\bigl(\eta_2+\vartheta(\eta_1-\eta_2)\bigr)
(\eta_1-\eta_2)
}{
\eta_1-\eta_2
}_\U
\,\dd\vartheta
\geq
\|\eta_1-\eta_2\|_\U^2,
\]
which proves \eqref{eq:StrongMonotonicityGamma}.
The scalar map
$s\longmapsto\frac{s}{1-s^2}$
is a strictly increasing bijection from $[0,1)$ onto
$\R_{\geq0}$. Since $\gamma$ is radial with this radial component,
that is,
$\gamma(s\omega)=\frac{s}{1-s^2}\omega$ for $s\in[0,1)$ and $\|\omega\|_\U=1$,
it is bijective.

For every $\tau>0$, the map
$\eta\longmapsto\eta+\tau\gamma(\eta)$
is radial and its radial component
$s\mapsto
s+\tau\tfrac{s}{1-s^2}$
is a bijection from $[0,1)$ onto $\R_{\geq0}$. Hence
$\im(I_\U+\tau\gamma)=\U$, i.e., $\gamma$ is maximally monotone.

Finally, let $r_j=\gamma(\eta_j)$, $j=1,2$. By
\eqref{eq:StrongMonotonicityGamma},
\[
\|\eta_1-\eta_2\|_\U^2
\leq
\scprod{r_1-r_2}{\eta_1-\eta_2}_\U
\leq
\|r_1-r_2\|_\U\|\eta_1-\eta_2\|_\U.
\]
Thus, $\|\gamma^{-1}(r_1)-\gamma^{-1}(r_2)\|_\U
\leq
\|r_1-r_2\|_\U$.
\end{proof}

\subsection{Proof of Lemma~\ref{lem:ZeroOutputLifting}}

\begin{proof}
Let $p\in\U$. By \eqref{eq:ABform},
\[
\spvek{R_\mu Bp}{p}
\in\dom(S)
\;\text{ and }
\;
A\&B\spvek{R_\mu Bp}{p}
=
\mu R_\mu Bp.
\]
The definition of the transfer operator gives
$\sbvek{A\&B}{C\&D} \spvek{R_\mu Bp}{p} = \spvek{\mu R_\mu Bp}{P(\mu)p}$.
Since $QP(\mu)p\in\dom(A)$ and $CQ=I_\U$, we further have
$\sbvek{A\&B}{C\&D} \spvek{QP(\mu)p}{0} = \spvek{AQP(\mu)p}{P(\mu)p}$.
Subtracting the two identities and using
\eqref{eq:GMuDefinition} proves the
assertion.
\end{proof}

\subsection{Proof of Lemma~\ref{lem:TruncatedBarrier}}

\begin{proof}
Fix $R\in(0,1)$ and set
$R_*:=\frac{1+R}{2}$, $d_R := \frac{2R}{(1-R^2)^2}$.
Define $\vartheta_R:\R_{\geq0}\to\R_{>0}$ by
\[
\vartheta_R(s)
:=
\begin{cases}
\displaystyle
\frac{1}{1-s^2},
&
0\leq s\leq R,
\\[3mm]
\displaystyle
\frac{1}{1-R^2}
+
d_R
\left[
s-R
-
\frac{(s-R)^2}{2(R_*-R)}
\right],
&
R<s<R_*,
\\[3mm]
\displaystyle
\frac{1}{1-R^2}
+
\frac{d_R}{2}(R_*-R),
&
s\geq R_*.
\end{cases}
\]
The function $\vartheta_R$ is continuously differentiable, its
derivative is globally Lipschitz continuous, and
$\vartheta_R(s)\geq1$,
$\vartheta_R'(s)\geq0$
for all $s\geq0$.
Moreover,
$\vartheta_R(s)=\frac{1}{1-s^2}$
for $s\in[0,R]$,
and $\vartheta_R$ is constant on $[R_*,\infty)$.

Define
$\Gamma_R(\xi)
:=
\vartheta_R(\|\xi\|_\U)\xi$,
$\xi\in\U$.
Then
$\Gamma_R(\xi)=\gamma(\xi)$ for all $\xi\in\overline{\mathcal B_R^\U(0)}$.
For $\xi\neq0$ and $h\in\U$, writing $s=\|\xi\|_\U$, we have
\begin{equation}
\label{eq:SmoothTruncationDerivative}
\Gamma_R'(\xi)h
=
\vartheta_R(s)h
+
\frac{\vartheta_R'(s)}{s}
\scprod{\xi}{h}_\U\xi.
\end{equation}
At $\xi=0$ one has
$\Gamma_R'(0)=I_\U$.
Define
\[
b_R(s)
:=
\begin{cases}
\vartheta_R'(s)/s,&s>0,\\
2,&s=0.
\end{cases}
\]
A direct inspection of the explicit definition of $\vartheta_R$
shows that $b_R$ is continuous and piecewise continuously
differentiable with bounded derivative. Moreover,
$b_R(s)=0$ for $s\geq R_*$.
Hence $b_R$ is globally bounded and globally Lipschitz continuous.

Formula~\eqref{eq:SmoothTruncationDerivative} can therefore be
written as
\[
\Gamma_R'(\xi)
=
\vartheta_R(\|\xi\|_\U)I_\U
+
b_R(\|\xi\|_\U)
\scprod{\xi}{\,\cdot\,}_\U\xi.
\]
For all $\xi,\zeta\in\U$, the corresponding rank-one operators satisfy
\[
\left\|
\scprod{\xi}{\,\cdot\,}_\U\xi
-
\scprod{\zeta}{\,\cdot\,}_\U\zeta
\right\|_{\Lb(\U)}
\leq
\bigl(
\|\xi\|_\U+\|\zeta\|_\U
\bigr)
\|\xi-\zeta\|_\U.
\]
Since $b_R$ is globally Lipschitz continuous and vanishes on
$[R_*,\infty)$, this estimate shows that there exists some
$c_R>0$ such that, for all $\xi,\zeta\in\U$,
\[
\|
\Gamma_R'(\xi)-\Gamma_R'(\zeta)
\|_{\Lb(\U)}
\leq
c_R\|\xi-\zeta\|_\U.
\]
The same representation also shows that
$\Gamma_R'$ is globally bounded. In particular, its right-hand side
extends continuously to $\xi=0$. Consequently, $\Gamma_R$ is
continuously Fr\'echet differentiable
on $\U$, and its derivative
$\Gamma_R':\U\to\Lb(\U)$
is globally bounded and globally Lipschitz continuous.

The operator in \eqref{eq:SmoothTruncationDerivative} is
self-adjoint. To make its coercivity explicit, let $\xi\neq0$,
$s=\|\xi\|_\U$, and decompose
\[
h=h_\perp+h_\parallel,
\qquad
h_\parallel
:=
\frac{\scprod{\xi}{h}_\U}{s^2}\xi,
\qquad
h_\perp\perp\xi.
\]
Then \eqref{eq:SmoothTruncationDerivative} gives
\[
\Gamma_R'(\xi)h
=
\vartheta_R(s)h_\perp
+
\bigl(
\vartheta_R(s)+s\vartheta_R'(s)
\bigr)h_\parallel.
\]
Thus,
\[
\scprod{\Gamma_R'(\xi)h}{h}_\U
=
\vartheta_R(s)\|h_\perp\|_\U^2
+
\bigl(
\vartheta_R(s)+s\vartheta_R'(s)
\bigr)
\|h_\parallel\|_\U^2.
\]
Since $\vartheta_R(s)\geq1$ and
$\vartheta_R'(s)\geq0$, both coefficients are at least one, and hence
\[
\scprod{\Gamma_R'(\xi)h}{h}_\U
\geq
\|h_\perp\|_\U^2+\|h_\parallel\|_\U^2
=
\|h\|_\U^2.
\]
For $\xi=0$, the same estimate follows from
$\Gamma_R'(0)=I_\U$. Hence it holds for all $\xi,h\in\U$. Integration along the line segment joining
$\xi_1$ and $\xi_2$ gives
\[
\scprod{
\Gamma_R(\xi_1)-\Gamma_R(\xi_2)
}{
\xi_1-\xi_2
}_\U
\geq
\|\xi_1-\xi_2\|_\U^2.
\]
Thus $\Gamma_R$ is strongly monotone.

Since $\vartheta_R$ is constant on $[R_*,\infty)$,
$\Gamma_R$ is linear outside
$\mathcal B_{R_*}^\U(0)$.
Finally, for $\xi\neq0$,
\begin{align*}
\Gamma_R'(\xi)\xi
&=
\left[
\vartheta_R(\|\xi\|_\U)
+
\|\xi\|_\U\vartheta_R'(\|\xi\|_\U)
\right]\xi,
\\
\Gamma_R'(\xi)\Gamma_R(\xi)
&=
\vartheta_R(\|\xi\|_\U)
\left[
\vartheta_R(\|\xi\|_\U)
+
\|\xi\|_\U\vartheta_R'(\|\xi\|_\U)
\right]\xi.
\end{align*}
The scalar coefficients in these two radial maps are Lipschitz
continuous and constant outside the ball
$\mathcal B_{R_*}^\U(0)$. Hence both maps are globally Lipschitz
continuous. The same is true for $\Gamma_R$ because its derivative is
globally bounded.
\end{proof}

\section{Proof of Proposition~\ref{prop:TruncatedProductOperator}}
\label{app:TruncatedProductOperator}

\begin{proof}
Since $q$ is bounded and $\Gamma_R'$ is globally bounded and globally
Lipschitz continuous, there exists a constant $c_R>0$ such that, for
all $t\geq0$ and $\xi,\xi_1,\xi_2\in\U$,
\[
\begin{aligned}
\|K_R(t,\xi)\|_{\Lb(\U,\X)}
&\leq c_R,
\\
\|K_R(t,\xi_1)-K_R(t,\xi_2)\|_{\Lb(\U,\X)}
&\leq
c_R\|\xi_1-\xi_2\|_\U.
\end{aligned}
\]

We first prove that the elements in
\eqref{eq:ExtendedTruncatedSystemNodeRelation}
are uniquely determined. Suppose that
$(\widehat f_j,\eta_j)$, $j=1,2$, satisfy this relation for the same element
$w=(\widehat z,\chi,k)$.
Applying impedance passivity to the difference of the corresponding
system-node elements gives
\[
0
\leq{}
-\kappa(k)\|\eta_1-\eta_2\|_\U^2
-
\beta
\scprod{
\gamma(\eta_1)-\gamma(\eta_2)
}{
\eta_1-\eta_2
}_\U.
\]
By \eqref{eq:StrongMonotonicityGamma}, this implies
$\eta_1=\eta_2$. The first component of the system node then gives
$\widehat f_1=\widehat f_2$. Hence
$\mathcal A_{R,{\rm a}}(t)$ is
single-valued.

We next establish hypomonotonicity. Let
$w_j=(\widehat z_j,\chi_j,k_j)\in\mathcal D_{\rm a}$,
$j=1,2$, and let $\widehat f_j,\eta_j$ denote the corresponding elements in
\eqref{eq:ExtendedTruncatedSystemNodeRelation}. Write $\Delta$ for
the difference of quantities with indices $1$ and $2$, and set
$\kappa_j:=\kappa(k_j)$.
Impedance passivity gives
\begin{equation}
-\scprod{
\Delta\widehat f
}{
\Delta\widehat z
}_\X
\geq{}
\scprod{
\kappa_1\eta_1-\kappa_2\eta_2
}{
\Delta\eta
}_\U
+
\beta
\scprod{
\gamma(\eta_1)-\gamma(\eta_2)
}{
\Delta\eta
}_\U.
\label{eq:ExtendedOperatorPassivityDifference}
\end{equation}
Since $\kappa$ is nonnegative and Lipschitz continuous with Lipschitz
constant one,
\[
\scprod{
\kappa_1\eta_1-\kappa_2\eta_2
}{
\Delta\eta
}_\U
=
\kappa_1\|\Delta\eta\|_\U^2
+
(\kappa_1-\kappa_2)
\scprod{\eta_2}{\Delta\eta}_\U
\geq
-|k_1-k_2|\,
\|\Delta\eta\|_\U,
\]
where we have used $\|\eta_2\|_\U<1$. Together with
\eqref{eq:StrongMonotonicityGamma},
\eqref{eq:ExtendedOperatorPassivityDifference} therefore yields
\begin{equation}
-\scprod{
\Delta\widehat f
}{
\Delta\widehat z
}_\X
\geq
\beta\|\Delta\eta\|_\U^2
-
|\Delta k|\,\|\Delta\eta\|_\U.
\label{eq:ExtendedOperatorBasicDifferenceEstimate}
\end{equation}
Moreover,
\[
\|
K_R(t,\chi_1)\eta_1
-
K_R(t,\chi_2)\eta_2
\|_\X
\leq
c_R
\left(
\|\Delta\eta\|_\U
+
\|\Delta\chi\|_\U
\right),
\]
because $\|\eta_2\|_\U<1$. Hence
\[
\begin{aligned}
&
\scprod{
\mathcal A_{R,{\rm a}}(t)w_1
-
\mathcal A_{R,{\rm a}}(t)w_2
}{
w_1-w_2
}_{\mathcal H_{\rm a}} =
-\scprod{
\Delta\widehat f
}{
\Delta\widehat z
}_\X
-
\scprod{
K_R(t,\chi_1)\eta_1
-
K_R(t,\chi_2)\eta_2
}{
\Delta\widehat z
}_\X - \tau_{\rm f}^{-1}
\scprod{
\Delta\eta
}{
\Delta\chi
}_\U.
\end{aligned}
\]
Combining this identity with
\eqref{eq:ExtendedOperatorBasicDifferenceEstimate} and applying
Young's inequality gives a constant
$\omega_R\geq0$, independent of $t$, such that
\[
\scprod{
\mathcal A_{R,{\rm a}}(t)w_1
-
\mathcal A_{R,{\rm a}}(t)w_2
}{
w_1-w_2
}_{\mathcal H_{\rm a}}
\geq
-\omega_R
\|w_1-w_2\|_{\mathcal H_{\rm a}}^2.
\]
Thus
$\mathcal A_{R,{\rm a}}(t) + \omega_RI_{\mathcal H_{\rm a}}$
is monotone for every $t\geq0$.

We next derive the two estimates needed in the range
argument: an $O(\lambda_{\rm r}^{-1})$ bound for the observation
resolvent $CR_{\lambda_{\rm r}}Q$ and a coercivity estimate of order
$\lambda_{\rm r}^{-1}$ for the transfer operator
$P(\lambda_{\rm r})$. For $\lambda_{\rm r}\geq1$, set
$R_{\lambda_{\rm r}} := (\lambda_{\rm r} I-A_{-1})^{-1}$.
On $\X$, this is the usual resolvent $(\lambda_{\rm r} I-A)^{-1}$. Since
$Q\U\subset\dom(A)$, the resolvent identity gives
\[
\lambda_{\rm r} R_{\lambda_{\rm r}} Q
=
Q+R_{\lambda_{\rm r}} AQ,
\qquad
AR_{\lambda_{\rm r}} Q
=
R_{\lambda_{\rm r}} AQ.
\]
The contraction-resolvent estimate and the graph-norm boundedness of
$C$ therefore imply that there exists some $c_Q>0$, independent of
$\lambda_{\rm r}$, such that, for every $\lambda_{\rm r}\geq1$,
\begin{equation}
\label{eq:ExtendedResolventObservationEstimate}
\|CR_{\lambda_{\rm r}} Q\|_{\Lb(\U)}
\leq
\frac{c_Q}{\lambda_{\rm r}}.
\end{equation}
Let $m_1>0$ be a coercivity constant in
\eqref{eq:TransferFunctionCoercivity} corresponding to $\mu=1$.
Lemma~\ref{lem:QuantitativeTransferCoercivity} gives, for every
$\lambda_{\rm r}\geq1$ and $u\in\U$,
\begin{equation}
\label{eq:ExtendedTransferCoercivityEstimate}
\scprod{
P(\lambda_{\rm r})u
}{
u
}_\U
\geq
\frac{m_1}{\lambda_{\rm r}}
\|u\|_\U^2.
\end{equation}
For $k\in\R$, define
$\Phi_k:
\mathcal B_1^\U(0)
\to
\U$,
$\Phi_k(\eta)
:=
\kappa(k)\eta+\beta\gamma(\eta)$.
For $\eta_1,\eta_2\in\mathcal B_1^\U(0)$,
\[
\scprod{
\Phi_k(\eta_1)-\Phi_k(\eta_2)
}{
\eta_1-\eta_2
}_\U
=
\kappa(k)\|\eta_1-\eta_2\|_\U^2
+
\beta
\scprod{
\gamma(\eta_1)-\gamma(\eta_2)
}{
\eta_1-\eta_2
}_\U
\geq
\beta\|\eta_1-\eta_2\|_\U^2.
\]
Moreover, $\Phi_k$ is radial, and its radial component is the
strictly increasing bijection
$s \longmapsto \kappa(k)s + \beta\frac{s}{1-s^2}$
from $[0,1)$ onto $\R_{\geq0}$. Hence $\Phi_k$ is bijective. For all
$p_1,p_2\in\U$, its inverse satisfies
\begin{equation}
\label{eq:PhiInverseLipschitz}
\|
\Phi_k^{-1}(p_1)
-
\Phi_k^{-1}(p_2)
\|_\U
\leq
\frac{1}{\beta}
\|p_1-p_2\|_\U.
\end{equation}
We choose $c_R$ large enough so that it also serves as the constant
in \eqref{eq:ExtendedNDifferenceEstimate} below.
Increasing $\omega_R$ if necessary, we may additionally assume that
\begin{equation}
\label{eq:ExtendedOperatorShiftChoice}
\omega_R\geq1,
\qquad
\omega_R
\geq
\frac{c_Q^2c_R^2}{\beta m_1}.
\end{equation}
It remains to verify the range condition for
$\mathcal M_{R,{\rm a}}(t)$. Fix
$\lambda>0$ and
$h=(h_z,h_\chi,h_k)\in\mathcal H_{\rm a}$, and set
\[
\lambda_{\rm r}
:=
\lambda^{-1}+\omega_R,
\qquad
\xi
:=
\frac{h_\chi}{1+\lambda\omega_R},
\qquad
k
:=
\frac{h_k}{1+\lambda\omega_R}.
\]
Then $\lambda_{\rm r}\geq\omega_R\geq1$.
For $p\in\U$, define
\[
\eta(p)
:=
\Phi_k^{-1}(p),
\qquad
\chi(p)
:=
\xi
+
\frac{1}{\lambda_{\rm r}\tau_{\rm f}}\eta(p),
\]
and
\[
\mathcal N_{t,\xi,\lambda_{\rm r},k}(p)
:=
\left(
q(t)I_\U
+
\alpha\Gamma_R'
\bigl(
\chi(p)
\bigr)
\right)
\eta(p).
\]
By
\eqref{eq:PhiInverseLipschitz} and the global boundedness and
Lipschitz continuity of $\Gamma_R'$, for all
$t\geq0$, $\lambda_{\rm r}\geq1$, $k\in\R$, and $p_1,p_2\in\U$,
\begin{equation}
\label{eq:ExtendedNDifferenceEstimate}
\|
\mathcal N_{t,\xi,\lambda_{\rm r},k}(p_1)
-
\mathcal N_{t,\xi,\lambda_{\rm r},k}(p_2)
\|_\U
\leq
c_R
\|
\eta(p_1)-\eta(p_2)
\|_\U.
\end{equation}
Define
$\mathcal T_{t,\xi,\lambda_{\rm r},k}:\U\to\U$
by
\[
\mathcal T_{t,\xi,\lambda_{\rm r},k}(p)
:={}
\Phi_k^{-1}(p)
+
P(\lambda_{\rm r})p
-
CR_{\lambda_{\rm r}} Q
\mathcal N_{t,\xi,\lambda_{\rm r},k}(p).
\]
This operator is everywhere defined and continuous. Let
$p_1,p_2\in\U$, and set
\[
\eta_j
:=
\Phi_k^{-1}(p_j),
\qquad
\Delta p:=p_1-p_2,
\qquad
\Delta\eta:=\eta_1-\eta_2.
\]
Since
$p_j=\Phi_k(\eta_j)$,
the strong monotonicity of $\Phi_k$ gives
$\scprod{ \Delta\eta }{ \Delta p }_\U \geq \beta\|\Delta\eta\|_\U^2$.
Using
\eqref{eq:ExtendedResolventObservationEstimate},
\eqref{eq:ExtendedTransferCoercivityEstimate}, and
\eqref{eq:ExtendedNDifferenceEstimate}, we obtain
\begin{multline*}
\scprod{
\mathcal T_{t,\xi,\lambda_{\rm r},k}(p_1)
-
\mathcal T_{t,\xi,\lambda_{\rm r},k}(p_2)
}{
\Delta p
}_\U
\geq
\beta\|\Delta\eta\|_\U^2
+
\frac{m_1}{\lambda_{\rm r}}
\|\Delta p\|_\U^2
-
\frac{c_Qc_R}{\lambda_{\rm r}}
\|\Delta\eta\|_\U
\|\Delta p\|_\U
\\
\geq
\frac{\beta}{2}
\|\Delta\eta\|_\U^2
+
\left(
\frac{m_1}{\lambda_{\rm r}}
-
\frac{c_Q^2c_R^2}
{2\beta\lambda_{\rm r}^2}
\right)
\|\Delta p\|_\U^2
\geq
\frac{m_1}{2\lambda_{\rm r}}
\|\Delta p\|_\U^2,
\end{multline*}
where the last inequality follows from
\eqref{eq:ExtendedOperatorShiftChoice}. Hence
$\mathcal T_{t,\xi,\lambda_{\rm r},k} - \frac{m_1}{2\lambda_{\rm r}}I_\U$
is an everywhere defined continuous monotone operator and therefore
maximally monotone. Since
\[
\mathcal T_{t,\xi,\lambda_{\rm r},k}
-
\frac{m_1}{2\lambda_{\rm r}}I_\U
\]
is maximally monotone, Minty's range condition, applied with
parameter $2\lambda_{\rm r}/m_1$, gives
\[
\U
=
\im
\left[
I_\U
+
\frac{2\lambda_{\rm r}}{m_1}
\left(
\mathcal T_{t,\xi,\lambda_{\rm r},k}
-
\frac{m_1}{2\lambda_{\rm r}}I_\U
\right)
\right]
=
\im
\left(
\frac{2\lambda_{\rm r}}{m_1}
\mathcal T_{t,\xi,\lambda_{\rm r},k}
\right).
\]
Consequently,
$\im \mathcal T_{t,\xi,\lambda_{\rm r},k} = \U$.
The preceding estimate also gives injectivity.

Consequently, there exists a unique $p\in\U$ such that
\begin{equation}
\label{eq:ExtendedReducedResolventEquation}
\mathcal T_{t,\xi,\lambda_{\rm r},k}(p)
=
CR_{\lambda_{\rm r}}
\bigl(
\lambda^{-1}h_z
\bigr).
\end{equation}
Set
\[
\eta
:=
\Phi_k^{-1}(p),
\qquad
\chi
:=
\xi+\frac{1}{\lambda_{\rm r}\tau_{\rm f}}\eta,
\]
$n := \mathcal N_{t,\xi,\lambda_{\rm r},k}(p)$,
and define
$\widehat z := R_{\lambda_{\rm r}} \left( \lambda^{-1}h_z + Qn - Bp \right)$,
$\widehat f := \lambda_{\rm r}\widehat z - \lambda^{-1}h_z - Qn$.
Then
$A_{-1}\widehat z-Bp = \widehat f$.
Moreover,
\eqref{eq:ExtendedReducedResolventEquation} gives
$\eta = C\&D \spvek{ \widehat z }{ -p }$.
Since
$p = \Phi_k(\eta) = \kappa(k)\eta+\beta\gamma(\eta)$,
we obtain
\[
\spvek{
\widehat f
}{
\eta
}
=
\sbvek{A\&B}{C\&D}
\spvek{
\widehat z
}{
-\kappa(k)\eta-\beta\gamma(\eta)
}.
\]
Thus
$w:=(\widehat z,\chi,k)\in\mathcal D_{\rm a}$.

By construction,
$h_k = (1+\lambda\omega_R)k$ and
$h_\chi = (1+\lambda\omega_R)\chi - \lambda\tau_{\rm f}^{-1}\eta$.
Moreover, the definition of $\widehat f$ gives
\[
\lambda\widehat f
=
(1+\lambda\omega_R)\widehat z
-
h_z
-
\lambda Qn.
\]
Since
\[
Qn
=
Q\mathcal N_{t,\xi,\lambda_{\rm r},k}(p)
=
K_R(t,\chi)\eta,
\]
we obtain
\[
h_z
=
(1+\lambda\omega_R)\widehat z
-
\lambda\widehat f
-
\lambda K_R(t,\chi)\eta.
\]
Consequently,
$h = w + \lambda \mathcal M_{R,{\rm a}}(t)w$.
Hence
\[
\im
\bigl(
I_{\mathcal H_{\rm a}}
+
\lambda\mathcal M_{R,{\rm a}}(t)
\bigr)
=
\mathcal H_{\rm a}.
\]
Therefore,
$\mathcal M_{R,{\rm a}}(t)$ is maximally monotone.
\end{proof}
\section{Proof of Proposition~\ref{prop:TruncatedWellPosedness}}
\label{app:TruncatedEvolutionEquation}

\begin{proof}
After translating time and relabelling the shifted coefficient
functions, it suffices to consider $t_0=0$. For notational
simplicity, set
\[
\mathcal H:=\mathcal H_{\rm a},
\qquad
\mathcal D:=\mathcal D_{\rm a},
\qquad
\mathcal M_R(t):=\mathcal M_{R,{\rm a}}(t),
\qquad
g(t):=g_\sigma(t).
\]
Define
$\mathcal G_R(t,w) := \mathcal F_{R,{\rm a}}(t,w)+\omega_Rw$.
By Lemma~\ref{lem:TruncatedBarrier} and the boundedness and Lipschitz
continuity of the scalar coefficient functions, there exists a constant
$L_R>0$ such that, for all $t\geq0$ and $w_1,w_2\in\mathcal H$,
\begin{equation}
\label{eq:TotalPerturbationStateLipschitz}
\|\mathcal G_R(t,w_1)-\mathcal G_R(t,w_2)\|_{\mathcal H}
\leq
L_R\|w_1-w_2\|_{\mathcal H},
\end{equation}
and, for all $s,t\geq0$ and $w\in\mathcal H$,
\[
\|\mathcal G_R(t,w)-\mathcal G_R(s,w)\|_{\mathcal H}
\leq
L_R|t-s|\|w\|_{\mathcal H}.
\]
Moreover, $\mathcal G_R(t,0)=0$ for every $t\geq0$.
Proposition~\ref{prop:TruncatedProductOperator} shows that
$\mathcal M_R(t)$ is maximally monotone, and hence $m$-accretive, for
every $t\geq0$.

We first verify condition~{\rm(j)} of
\cite[Thm.~4.19]{Barbu2010}. For $\lambda>0$, let
\[
J_{\lambda,R}(t)
:=
\bigl(I_{\mathcal H}+\lambda\mathcal M_R(t)\bigr)^{-1},
\qquad
\mathcal M_{R,\lambda}(t)
:=
\frac{1}{\lambda}\bigl(I_{\mathcal H}-J_{\lambda,R}(t)\bigr)
\]
be the resolvent and the Yosida approximation of $\mathcal M_R(t)$.
Fix $y\in\mathcal H$ and set $w_t:=J_{\lambda,R}(t)y$. Then
\[
y=w_t+\lambda\mathcal M_{R,\lambda}(t)y,
\qquad
\mathcal M_{R,\lambda}(t)y=\mathcal M_R(t)w_t.
\]
Since the operators $\mathcal M_R(t)$ have the common domain
$\mathcal D$, $w_t\in\mathcal D$ and
\[
w_t
=
J_{\lambda,R}(s)
\left(
w_t+\lambda\mathcal M_R(s)w_t
\right),
\qquad
w_s
=
J_{\lambda,R}(s)y
=
J_{\lambda,R}(s)
\left(
w_t+\lambda\mathcal M_R(t)w_t
\right).
\]
Hence, by the nonexpansiveness of $J_{\lambda,R}(s)$,
\[
\|w_t-w_s\|_{\mathcal H}
\leq
\lambda
\|\bigl(\mathcal M_R(s)-\mathcal M_R(t)\bigr)w_t\|_{\mathcal H}.
\]
Consequently,
\[
\|\mathcal M_{R,\lambda}(t)y
-\mathcal M_{R,\lambda}(s)y\|_{\mathcal H}
\leq
\|\bigl(\mathcal M_R(s)-\mathcal M_R(t)\bigr)w_t\|_{\mathcal H}.
\]

Let $c_q$ be a Lipschitz constant of the mapping
$t\mapsto q(t)Q$. For
$w=(\widehat z,\chi,k)\in\mathcal D$, let
$\eta\in\mathcal B_1^\U(0)$ be the uniquely determined element in
\eqref{eq:ExtendedTruncatedSystemNodeRelation}. This element is independent of
time, and the explicit time dependence of $\mathcal M_R(t)$ occurs
only through $q$. Hence
\[
\bigl(\mathcal M_R(t)-\mathcal M_R(s)\bigr)w
=
\begin{pmatrix}
-(q(t)-q(s))Q\eta\\
0\\
0
\end{pmatrix}.
\]
Since $\|\eta\|_\U<1$, we obtain, for all
$w\in\mathcal D$ and $s,t\geq0$,
\begin{equation}
\label{eq:ProductOperatorTimeLipschitz}
\|\mathcal M_R(t)w-\mathcal M_R(s)w\|_{\mathcal H}
\leq
c_q|t-s|.
\end{equation}
It follows that
\[
\|\mathcal M_{R,\lambda}(t)y
-\mathcal M_{R,\lambda}(s)y\|_{\mathcal H}
\leq
c_q|t-s|
\left(
\|\mathcal M_{R,\lambda}(t)y\|_{\mathcal H}
+
\|y\|_{\mathcal H}
+
1
\right).
\]
Thus $\{\mathcal M_R(t)\}_{t\geq0}$ satisfies condition~{\rm(j)} of
\cite[Thm.~4.19]{Barbu2010}.

Fix $T>0$. We now solve
$\dot w(t)+\mathcal M_R(t)w(t)=\mathcal G_R(t,w(t))+g(t)$ with $w(0)=w_0$
on $[0,T]$ by successive approximation. Set
$w^{(0)}(t):=w_0$ on $[0,T]$.
Given an iterate $w^{(n)}$, the mapping
$t\mapsto\mathcal G_R(t,w^{(n)}(t))+g(t)$ belongs to
$\Wkp{1,1}([0,T];\mathcal H)$. Hence
\cite[Thm.~4.19]{Barbu2010} yields a unique next iterate
$w^{(n+1)}\in\Wkp{1,\infty}([0,T];\mathcal H)$ satisfying
\begin{equation}
\label{eq:TruncatedPicardIteration}
\begin{aligned}
\dot w^{(n+1)}(t)+\mathcal M_R(t)w^{(n+1)}(t)
&=
\mathcal G_R\bigl(t,w^{(n)}(t)\bigr)+g(t),
\\
w^{(n+1)}(0)&=w_0.
\end{aligned}
\end{equation}
For $n\geq1$, subtracting the equations
\eqref{eq:TruncatedPicardIteration} for the iterates
$w^{(n+1)}$ and $w^{(n)}$ gives
\[
\begin{aligned}
\frac{\dd}{\dd t}
\bigl(
w^{(n+1)}-w^{(n)}
\bigr)
&+
\mathcal M_R(t)w^{(n+1)}
-
\mathcal M_R(t)w^{(n)} =
\mathcal G_R\bigl(t,w^{(n)}\bigr)
-
\mathcal G_R\bigl(t,w^{(n-1)}\bigr).
\end{aligned}
\]
Taking the inner product with
$w^{(n+1)}-w^{(n)}$ and using the accretivity of
$\mathcal M_R(t)$ yields, for almost every $t\in[0,T]$,
\[
\frac{\dd}{\dd t}
\|w^{(n+1)}-w^{(n)}\|_{\mathcal H}
\leq
\left\|
\mathcal G_R\bigl(t,w^{(n)}\bigr)
-
\mathcal G_R\bigl(t,w^{(n-1)}\bigr)
\right\|_{\mathcal H}.
\]
Since all iterates have the same initial value, integration and
\eqref{eq:TotalPerturbationStateLipschitz} give
\[
\|w^{(n+1)}(t)-w^{(n)}(t)\|_{\mathcal H}
\leq
L_R
\int_0^t
\|w^{(n)}(s)-w^{(n-1)}(s)\|_{\mathcal H}
\,\dd s.
\]
Iteration gives
\[
\sup_{t\in[0,T]}
\|w^{(n+1)}(t)-w^{(n)}(t)\|_{\mathcal H}
\leq
\frac{(L_RT)^n}{n!}
\sup_{t\in[0,T]}
\|w^{(1)}(t)-w^{(0)}(t)\|_{\mathcal H}.
\]
Thus $(w^{(n)})_{n\in\N}$ converges uniformly on $[0,T]$ to some
$w_R\in\mathrm{C}([0,T];\mathcal H)$.

It remains to show that the limit has the asserted strong regularity.
For this purpose, let $f\in\Wkp{1,1}([0,T];\mathcal H)$ and let
$u\in\Wkp{1,\infty}([0,T];\mathcal H)$ solve
$\dot u(t)+\mathcal M_R(t)u(t)=f(t)$ with $u(0)=w_0$.
For $h\in(0,T)$, set
\[
\Delta_hu(t):=\frac{u(t+h)-u(t)}{h},
\qquad
\Delta_hf(t):=\frac{f(t+h)-f(t)}{h}.
\]
Subtracting the equation
$\dot u+\mathcal M_R(\cdot)u=f$ at times $t+h$ and $t$, and inserting
and subtracting $\mathcal M_R(t+h)u(t)$, gives
\begin{multline*}
\frac{\dd}{\dd t}\bigl(u(t+h)-u(t)\bigr)
+
\mathcal M_R(t+h)u(t+h)
-
\mathcal M_R(t+h)u(t)
\\
=
f(t+h)-f(t)
-
\bigl(\mathcal M_R(t+h)u(t)-\mathcal M_R(t)u(t)\bigr)
\end{multline*}
almost everywhere on $[0,T-h]$. Monotonicity and
\eqref{eq:ProductOperatorTimeLipschitz} yield, almost everywhere on
$[0,T-h]$,
\[
\frac{\dd}{\dd t}\|\Delta_hu(t)\|_{\mathcal H}
\leq
\|\Delta_hf(t)\|_{\mathcal H}+c_q.
\]
To estimate the initial difference quotient, compare $u$
with the constant function $t\mapsto w_0$. By accretivity of
$\mathcal M_R(t)$,
\[
\frac{\dd}{\dd t}
\|u(t)-w_0\|_{\mathcal H}
\leq
\|f(t)-\mathcal M_R(t)w_0\|_{\mathcal H}
\]
for almost every $t\in[0,T]$. Moreover,
\eqref{eq:ProductOperatorTimeLipschitz} yields
\[
\|\mathcal M_R(t)w_0\|_{\mathcal H}
\leq
\|\mathcal M_R(0)w_0\|_{\mathcal H}
+
c_qt.
\]
Integrating from $0$ to $h$ and using $u(0)=w_0$ therefore gives
\[
\|\Delta_hu(0)\|_{\mathcal H}
=
\frac{\|u(h)-w_0\|_{\mathcal H}}{h}
\leq
\|f\|_{\Lp{\infty}([0,T];\mathcal H)}
+
\|\mathcal M_R(0)w_0\|_{\mathcal H}
+
c_qT.
\]
Moreover, absolute continuity of $f$ and Fubini's theorem imply
\[
\int_0^{T-h}\|\Delta_hf(s)\|_{\mathcal H}\,\dd s
\leq
\|\dot f\|_{\Lp{1}([0,T];\mathcal H)}.
\]
Combining these estimates and letting $h\to0$ gives
\begin{equation}
\label{eq:LinearDerivativeEstimate}
\|\dot u\|_{\Lp{\infty}([0,T];\mathcal H)}
\leq
\|\mathcal M_R(0)w_0\|_{\mathcal H}
+
\|f\|_{\Lp{\infty}([0,T];\mathcal H)}
+
\|\dot f\|_{\Lp{1}([0,T];\mathcal H)}
+
2c_qT.
\end{equation}

The uniform convergence of the Picard iterates implies that they are
uniformly bounded in $\mathrm{C}([0,T];\mathcal H)$. Applying
\eqref{eq:LinearDerivativeEstimate} to
\eqref{eq:TruncatedPicardIteration} on $[0,t]$ and using the two
Lipschitz estimates for $\mathcal G_R$, we obtain a constant $c>0$,
independent of $n$, such that, for every $t\in[0,T]$,
\[
\|\dot w^{(n+1)}\|_{\Lp{\infty}([0,t];\mathcal H)}
\leq
c
+
c\int_0^t
\|\dot w^{(n)}\|_{\Lp{\infty}([0,s];\mathcal H)}
\,\dd s.
\]
Since $w^{(0)}$ is constant, induction yields, for every $n\in\N$ and
$t\in[0,T]$,
\[
\|\dot w^{(n)}\|_{\Lp{\infty}([0,t];\mathcal H)}
\leq
ce^{ct}.
\]
Hence the iterates are uniformly Lipschitz continuous, and their
uniform limit satisfies
$w_R\in\Wkp{1,\infty}([0,T];\mathcal H)$.

The mapping $t\mapsto\mathcal G_R(t,w_R(t))+g(t)$ then belongs to
$\Wkp{1,1}([0,T];\mathcal H)$. Let $\widetilde w_R$ be
the strong solution supplied by \cite[Thm.~4.19]{Barbu2010} of the
equation obtained from \eqref{eq:TruncatedEvolutionEquation} by
keeping $w_R$ fixed in the right-hand side, that is,
\[
\dot{\widetilde w}_R(t)
+
\mathcal M_R(t)\widetilde w_R(t)
=
\mathcal G_R\bigl(t,w_R(t)\bigr)+g(t),
\qquad
\widetilde w_R(0)=w_0.
\] Continuous dependence and
\eqref{eq:TotalPerturbationStateLipschitz} give
\[
\|\widetilde w_R(t)-w^{(n+1)}(t)\|_{\mathcal H}
\leq
L_R\int_0^t
\|w_R(s)-w^{(n)}(s)\|_{\mathcal H}\,\dd s.
\]
Letting $n\to\infty$ shows that $\widetilde w_R=w_R$. Thus $w_R$ is
a strong solution of \eqref{eq:TruncatedEvolutionEquation}.

If $w_{R,1}$ and $w_{R,2}$ are two strong solutions with the same
initial value, accretivity and
\eqref{eq:TotalPerturbationStateLipschitz} give
\[
\|w_{R,1}(t)-w_{R,2}(t)\|_{\mathcal H}
\leq
L_R\int_0^t
\|w_{R,1}(s)-w_{R,2}(s)\|_{\mathcal H}\,\dd s.
\]
Gronwall's inequality yields $w_{R,1}=w_{R,2}$.

Finally, rearranging \eqref{eq:TruncatedEvolutionEquation} gives, for
almost every $t\in[0,T]$,
\[
\mathcal M_R(t)w_R(t)
=
\mathcal G_R\bigl(t,w_R(t)\bigr)+g(t)-\dot w_R(t).
\]
All terms on the right-hand side belong to
$\Lp{\infty}([0,T];\mathcal H)$. Hence
$\mathcal M_R(\cdot)w_R(\cdot)
\in
\Lp{\infty}([0,T];\mathcal H)$.
Since
\[\mathcal A_{R,{\rm a}}(t)w_R(t) = \mathcal M_R(t)w_R(t)-\omega_Rw_R(t),\]
we conclude that
$\mathcal A_{R,{\rm a}}(\cdot)w_R(\cdot)\in
\Lp{\infty}([0,T];\mathcal H)$.
\end{proof}

\section{Proof of Proposition~\ref{prop:ClosedLoopTransformation}}
\label{app:ClosedLoopTransformation}

\begin{proof}
We first prove the forward implication. Let
$[a,b]\subset\mathcal I$ be compact. Since $\chi([a,b])$ is a
compact subset of $\mathcal B_1^\U(0)$, the derivative $\gamma'$ is
bounded on this set. Consequently, the relations
\[
\gamma(\chi)
\in
\Wkp{1,\infty}([a,b];\U),
\qquad
\frac{\dd}{\dd t}\gamma(\chi)
=
\gamma'(\chi)\dot\chi
\]
hold almost everywhere on $[a,b]$. It follows from
\eqref{eq:TransformedState} and the regularity of the prescribed data
that
$\widehat z \in \Wkp{1,\infty}([a,b];\X)$.
Since $[a,b]$ was arbitrary,
$\widehat z \in \Wkp{1,\infty}_{\loc}(\mathcal I;\X)$.

The normalized controller relation is
$\tau_{\rm f}\varphi
\bigl(
u-u_{\ext}
\bigr)
=
-k_{\rm a}\eta-\beta\gamma(\eta)$.
Hence
\eqref{eq:TransformationDomainIdentity}, the linearity of
$\dom(S)$, Lemma~\ref{lem:ZeroOutputLifting}, and
$Q\U\subset\dom(A)$ show that
$\spvek{ \widehat z }{ -k_{\rm a}\eta-\beta\gamma(\eta) } \in\dom(S)$
almost everywhere on $\mathcal I$.

Define
\begin{equation}
\label{eq:ForwardTransformedF}
\widehat f
:={}
\tau_{\rm f}\varphi\dot x
-
\tau_{\rm f}\varphi G_\mu u_{\ext}
-
\tau_{\rm f}\varphi AQ\dot y_{\rm ref}
+
\tau_{\rm f}qAQ\chi
+
\tau_{\rm f}\alpha AQ\gamma(\chi).
\end{equation}
Then
$\widehat f \in \Lp{\infty}_{\loc}(\mathcal I;\X)$.
Applying the system node to
\eqref{eq:TransformationDomainIdentity} gives
\[
\sbvek{A\&B}{C\&D}
\spvek{
\widehat z
}{
-k_{\rm a}\eta-\beta\gamma(\eta)
}
=
\spvek{
\widehat f
}{
\tau_{\rm f}\varphi
\bigl(
v-\dot y_{\rm ref}
\bigr)
+
\tau_{\rm f}q\chi
+
\tau_{\rm f}\alpha\gamma(\chi)
}.
\]
By \eqref{eq:NormalizedVelocityRelation}, the second component on the
right-hand side equals $\eta$. Hence
\[
\spvek{
\widehat f
}{
\eta
}
=
\sbvek{A\&B}{C\&D}
\spvek{
\widehat z
}{
-k_{\rm a}\eta-\beta\gamma(\eta)
}
\]
almost everywhere on $\mathcal I$.

Differentiating \eqref{eq:TransformedState}, using
$\frac{\dd}{\dd t}\gamma(\chi) = \gamma'(\chi)\dot\chi$,
and substituting
\eqref{eq:ChiDynamics},
\eqref{eq:ForwardTransformedF}, and
the definition of $q$
gives the first equation in
\eqref{eq:TransformedClosedLoop}. Together with
\eqref{eq:ChiDynamics} and the preceding system-node relation, this
proves the forward implication.

We now prove the converse implication. Again let
$[a,b]\subset\mathcal I$ be compact. Since $\chi([a,b])$ is a
compact subset of $\mathcal B_1^\U(0)$, the derivative $\gamma'$ is
bounded on this set. Consequently,
$\gamma(\chi) \in \Wkp{1,\infty}([a,b];\U)$.

The resolvent representation of the system-node relation gives
\begin{equation}
\label{eq:EquivalenceTransferIdentity}
P(\mu)
\left(
k_{\rm a}\eta+\beta\gamma(\eta)
\right)
+
\eta
=
C(\mu I-A)^{-1}
\bigl(
\mu\widehat z-\widehat f
\bigr).
\end{equation}
The right-hand side belongs to
$\Lp{\infty}([a,b];\U)$. Hence there exists some $c_{a,b}>0$ such
that its norm is bounded by $c_{a,b}$ almost everywhere.

Since
$\eta = \bigl( 1-\|\eta\|_\U^2 \bigr)\gamma(\eta)$,
one has
\[
\scprod{
P(\mu)\eta
}{
\gamma(\eta)
}_\U
=
\bigl(
1-\|\eta\|_\U^2
\bigr)
\scprod{
P(\mu)\gamma(\eta)
}{
\gamma(\eta)
}_\U
\geq0.
\]
Taking the inner product of the left-hand side of
\eqref{eq:EquivalenceTransferIdentity} with $\gamma(\eta)$, estimating
the right-hand side accordingly, and using
$k_{\rm a}\geq0$ and
\eqref{eq:TransferFunctionCoercivity}, we obtain
$\beta m_\mu \|\gamma(\eta)\|_\U^2 \leq c_{a,b} \|\gamma(\eta)\|_\U$
almost everywhere on $[a,b]$. Therefore,
$\gamma(\eta) \in \Lp{\infty}([a,b];\U)$.
Since $[a,b]$ was arbitrary,
$\gamma(\eta) \in \Lp{\infty}_{\loc}(\mathcal I;\U)$.

The reconstruction formulas
\eqref{eq:ReconstructionFormulas} now imply
$x
\in
\Wkp{1,\infty}_{\loc}(\mathcal I;\X)$,
$u,v
\in
\Lp{\infty}_{\loc}(\mathcal I;\U)$,
as well as
$y \in \Wkp{1,\infty}_{\loc}(\mathcal I;\U)$.
Indeed, the additional term $k_{\rm a}\eta$ in the formula for $u$
is locally essentially bounded because
$k_{\rm a}
\in
\Wkp{1,\infty}_{\loc}(\mathcal I)
\subset
\Lp{\infty}_{\loc}(\mathcal I)$
and $\|\eta\|_\U<1$ almost everywhere.

Rearranging the first and third reconstruction formulas gives
\begin{equation}
\label{eq:InverseTransformationDomainIdentity}
\tau_{\rm f}\varphi
\spvek{x}{u}
={}
\spvek{
\widehat z
}{
-k_{\rm a}\eta-\beta\gamma(\eta)
}
+
\tau_{\rm f}\varphi
\spvek{L_\mu u_{\ext}}{u_{\ext}}
+
\tau_{\rm f}\varphi
\spvek{Q\dot y_{\rm ref}}{0}
-
\tau_{\rm f}
\spvek{Qq\chi}{0}
-
\tau_{\rm f}\alpha
\spvek{Q\gamma(\chi)}{0}.
\end{equation}
Each element on the right-hand side belongs to $\dom(S)$ almost
everywhere. Since $\dom(S)$ is a linear subspace, it follows that
$\spvek{x}{u} \in\dom(S)$
almost everywhere.

Applying the system node to
\eqref{eq:InverseTransformationDomainIdentity} yields
\begin{equation}
\label{eq:InverseTransformationNodeIdentity}
\tau_{\rm f}\varphi
\sbvek{A\&B}{C\&D}
\spvek{x}{u}
={}
\spvek{\widehat f}{\eta}
+
\tau_{\rm f}\varphi
\spvek{G_\mu u_{\ext}}{0}
+
\tau_{\rm f}\varphi
\spvek{
AQ\dot y_{\rm ref}
}{
\dot y_{\rm ref}
}
-
\tau_{\rm f}
\spvek{
qAQ\chi
}{
q\chi
}
-
\tau_{\rm f}\alpha
\spvek{
AQ\gamma(\chi)
}{
\gamma(\chi)
}.
\end{equation}
Its second component is precisely the formula for $v$ in
\eqref{eq:ReconstructionFormulas}. Thus
$C\&D\spvek{x}{u}=v$
almost everywhere.

Differentiating the reconstruction formula for $x$ and using the
first two relations in
\eqref{eq:TransformedClosedLoop},
\eqref{eq:ChiDynamics}, and
the definition of $q$
gives
\[
\tau_{\rm f}\varphi\dot x
={}
\widehat f
+
\tau_{\rm f}\varphi G_\mu u_{\ext}
+
\tau_{\rm f}\varphi AQ\dot y_{\rm ref}
-
\tau_{\rm f}qAQ\chi
-
\tau_{\rm f}\alpha AQ\gamma(\chi).
\]
Comparison with the first component of
\eqref{eq:InverseTransformationNodeIdentity} yields
$\dot x = A\&B\spvek{x}{u}$
almost everywhere. Hence $(x,u,v)$ is a strong trajectory of
\eqref{eq:VelocityNode}.

It remains to verify the position and controller relations. Since
$e=\varphi^{-1}\chi$,
we have
\[
\dot e
=
\varphi^{-1}
\left(
\dot\chi-(q-\delta)\chi
\right)
=
\frac{1}{\tau_{\rm f}\varphi}
\left(
\eta
-
\tau_{\rm f}\alpha\gamma(\chi)
-
\tau_{\rm f}q\chi
\right)
=
v-\dot y_{\rm ref}.
\]
Thus
$\dot y=v$.
Moreover,
\[
\dot e
+
\left(
\frac{\alpha}
{1-\varphi^2\|e\|_\U^2}
+
q
\right)e
=
\frac{1}{\tau_{\rm f}\varphi}\eta
=
e_{\rm r}.
\]
Finally,
\[
u
=
u_{\ext}
-
\frac{1}{\tau_{\rm f}\varphi}
\left(
k_{\rm a}\eta
+
\beta\gamma(\eta)
\right)
=
u_{\ext}
-
k_{\rm a}e_{\rm r}
-
\frac{\beta e_{\rm r}}
{1-\tau_{\rm f}^2\varphi^2
\|e_{\rm r}\|_\U^2}.
\]
Hence
$\dot y=v$, $e=y-y_{\rm ref}$,
and the first three relations in
\eqref{eq:ClosedLoopControlVelocity} hold almost everywhere.

The transformation formulas
\eqref{eq:TransformedState} and
\eqref{eq:ReconstructionFormulas} are algebraic inverses of each
other for the same gain $k_{\rm a}$. This completes the proof.
\end{proof}

\bibliographystyle{abbrv-doi}  \bibliography{References}%

@article{BerIlcRya2021,
title = {Funnel control of nonlinear systems},
journal ={Math.Control Signals Systems},
volume = {33},
pages = {151-194},
year = {2021},
author = {Berger, Thomas and Ilchmann, Achim and Ryan, Eugene P},
doi = {10.1007/s00498-021-00277-z},
}

@article{Reis2026_Dissip,
  author       = {Reis, Timo},
  title        = {The infinite-dimensional dissipation inequality},
  journal      = {ESAIM Control Optim. Calc. Var.},
  volume       = {32},
  year         = {2026},
  doi          = {10.1051/cocv/2026053},
  eprint       = {2508.18436},
  eprinttype   = {arXiv},
  primaryclass = {math.FA}
}

@article{Ber2018,
  title={Funnel control for nonlinear systems with known strict relative degree},
  author={Berger, Thomas and L{\^e}, Huy Ho{\`a}ng and Reis, Timo},
  journal={Automatica J. IFAC},
  volume={87},
  pages={345--357},
  year={2018},
  publisher={Elsevier},
  doi = {10.1016/j.automatica.2017.10.017},
}

@article{Ber2020,
title = {Funnel control in the presence of infinite-dimensional internal dynamics},
journal = {Systems Control Lett.},
volume = {139},
pages = {104678},
year = {2020},
author = {Thomas Berger and Marc Puche and Felix L. Schwenninger},
doi = {10.1016/j.sysconle.2020.104678}
}

@article{IlcRyaSan2002,
author = {Ilchmann, Achim and Ryan, E.P. and Sangwin, Christopher},
journal = {ESAIM Control Optim. Calc. Var.},
pages = {471--493},
volume = {7},
year = {2002},
doi = {10.1051/cocv:2002064},
title = {Tracking with prescribed transient behaviour},
}

@article{IlcTre2004,
title = {Input constrained funnel control with applications to chemical reactor models},
journal = {Systems Control Lett.},
volume = {53},
number = {5},
pages = {361-375},
year = {2004},
doi = {10.1016/j.sysconle.2004.05.014},
author = {Achim Ilchmann and Stephan Trenn}
}

@article{BerRei2014,
title = {Zero dynamics and funnel control for linear electrical circuits},
journal = {J. Franklin Inst.},
volume = {351},
number = {11},
pages = {5099-5132},
year = {2014},
author = {Thomas Berger and Timo Reis},
doi = {10.1016/j.jfranklin.2014.08.006}
}

@INPROCEEDINGS{Hac2014,
  author       = {Hackl, C. M.},
  title        = {Funnel control for wind turbine systems},
  booktitle    = {2014 IEEE Conference on Control Applications (CCA)},
  year         = {2014},
  pages        = {1377--1382},
  organization = {IEEE},
  doi          = {10.1109/CCA.2014.6981516}
}

@article{MarTimFel2021, 
title = {Funnel control for boundary control systems},
journal = {Evol. Equ. Control Theory},
volume = {10},
number = {3},
pages = {519-544},
year = {2021},
doi = {10.3934/eect.2020079},
author = {Marc Puche and Timo Reis and Felix L. Schwenninger}
}

@ARTICLE{Ber2024,
  author={Berger, Thomas},
  journal={IEEE Trans. Automat. Control}, 
  title={Input-Constrained Funnel Control of Nonlinear Systems}, 
  year={2024},
  volume={69},
  number={8},
  pages={5368-5382},
  doi={10.1109/TAC.2024.3352362}
}

@Book{sta2005,
	title     = {Well-Posed Linear Systems},
	publisher = {Cambridge University Press},
	year      = {2005},
	author    = {Staffans, Olof J},
        series = {Encyclopedia of Mathematics and Its Applications},
	volume    = {103},
        doi={10.1017/CBO9780511543197}
}

@misc{GovHasPauRei2025,
  author     = {Govindaraj, Thavamani and Hastir, Anthony and Paunonen, Lassi and Reis, Timo},
  title      = {Funnel Control for Passive Infinite-Dimensional Systems},
  year       = {2025},
  eprint     = {2510.01027},
  howpublished = {arxiv},
  doi = {10.48550/arXiv.2510.01027}
}

@INPROCEEDINGS{HacKen2012,
  author       = {Hackl, C. M. and Kennel, R. M.},
  title        = {Position funnel control for rigid revolute joint robotic manipulators with known inertia matrix},
  booktitle    = {2012 20th Mediterranean Conference on Control \& Automation (MED)},
  year         = {2012},
  pages        = {615--620},
  organization = {IEEE},
  doi          = {10.1109/MED.2012.6265706}
}

@article{IlcRyaTre2005,
title = {Tracking control: Performance funnels and prescribed transient behaviour},
journal = {Systems Control Lett.},
volume = {54},
number = {7},
pages = {655-670},
year = {2005},
issn = {0167-6911},
doi = {10.1016/j.sysconle.2004.11.005},
author = {Achim Ilchmann and Eugene P. Ryan and Stephan Trenn},
}

@article{ReiSel2015,
author = {Reis, Timo and Selig, Tilman},
title = {Funnel Control for the Boundary Controlled Heat Equation},
journal = {SIAM J. Control Optim.},
volume = {53},
number = {1},
pages = {547-574},
year = {2015},
doi = {10.1137/140971567}
}

@article{Ber2022,
title = {Funnel control for a moving water tank},
journal = {Automatica J. IFAC},
volume = {135},
pages = {109999},
year = {2022},
issn = {0005-1098},
doi = {10.1016/j.automatica.2021.109999},
author = {Thomas Berger and Marc Puche and Felix L. Schwenninger},
}

@article{HasWinDoc2023,
title = {Funnel control for a class of nonlinear infinite-dimensional systems},
journal = {Automatica J. IFAC},
volume = {152},
pages = {110964},
year = {2023},
issn = {0005-1098},
doi = {10.1016/j.automatica.2023.110964},
author = {Anthony Hastir and Joseph J. Winkin and Denis Dochain},
}

@article{IlchmannRyan2008,
  author       = {A. Ilchmann and E. P. Ryan},
  title        = {High-gain control without identification: a survey},
  journal      = {GAMM-Mitt.},
  year         = {2008},
  volume       = {31},
  number       = {1},
  pages        = {115--125},
  doi = {10.1002/gamm.200890000},
}

@book{Rudin1987,
  author    = {Rudin, Walter},
  title     = {Real and Complex Analysis},
  edition   = {3},
  publisher = {McGraw--Hill},
  address   = {New York},
  year      = {1987}
}

@article{IlchmannSeligTrunk2016,
  author       = {Achim Ilchmann and Tilman Selig and Carsten Trunk},
  title        = {The {B}yrnes--{I}sidori Form for Infinite-Dimensional Systems},
  journal      = {SIAM J. Control Optim.},
  year         = {2016},
  volume       = {54},
  number       = {3},
  pages        = {1504--1534},
  doi          = {10.1137/130942413},
}

@article{BergerBreitenPucheReis2021,
  author       = {Thomas Berger and Tobias Breiten and Marc Puche and Timo Reis},
  title        = {Funnel control for the monodomain equations with the {F}itz{H}ugh--{N}agumo model},
  journal      = {J. Differential Equations},
  year         = {2021},
  volume       = {286},
  pages        = {164--214},
  doi          = {10.1016/j.jde.2021.03.012}, }

@book{AdamsFournier2003,
  author       = {Adams, Robert A. and Fournier, John J. F.},
  title        = {Sobolev Spaces},
  series       = {Pure and Applied Mathematics},
  volume       = {140},
  edition      = {Second},
  publisher    = {Elsevier / Academic Press},
  address      = {Amsterdam, NL},
  year         = {2003},
  pages        = {xiv + 305},
  isbn         = {978-0-12-044143-3},
}

@book{DiestelUhl1977,
  author    = {Diestel, J. and Uhl, J. J.},
  title     = {Vector Measures},
  series    = {Mathematical Surveys and Monographs},
  volume    = {15},
  publisher = {American Mathematical Society},
  address   = {Providence, RI},
  year      = {1977},
  pages     = {xiii+322},
  isbn      = {978-0-8218-1515-0},
doi = {10.1090/surv/015}
}

@book{EngelNagel2000,
  author    = {Engel, K.-J. and Nagel, R.},
  title     = {One-Parameter Semigroups for Linear Evolution Equations},
  series    = {Graduate Texts in Mathematics},
  volume    = {194},
  publisher = {Springer-Verlag},
  address   = {New York},
  year      = {2000},
  pages     = {xxii+586},
  isbn      = {978-0-387-98463-1},
  doi       = {10.1007/b97696},
}

@article{Sta02,
	author = {Staffans, O. J.},
	journal = {Math. Control Signal},
	number = {4},
	pages = {291--315},
	title = {Passive and conservative continuous-time impedance and scattering systems. {Part I}: {W}ell-posed systems},
	volume = {15},
	year = {2002},
    doi={10.1007/s004980200012}}

@book{Barbu2010,
  author    = {Barbu, Viorel},
  title     = {Nonlinear Differential Equations of Monotone Types in Banach Spaces},
  series    = {Springer Monographs in Mathematics},
  publisher = {Springer},
  address   = {New York},
  year      = {2010},
  doi       = {10.1007/978-1-4419-5542-5},
  url       = {https://link.springer.com/book/10.1007/978-1-4419-5542-5}
}
\end{document}